\documentclass[a4paper]{article}

\usepackage{algorithm}
\usepackage{algpseudocode}
\usepackage{amsmath} 

\usepackage{float}

\usepackage{hyperref}

\usepackage{booktabs}
\usepackage{subcaption}

\usepackage{multirow}

\usepackage{graphicx}
\usepackage{subcaption}

\usepackage{xcolor}
\usepackage[top=2.3cm, bottom=2.3cm, left=2cm, right=2cm]{geometry}

\usepackage{tikz}

\usepackage{multirow}

\usepackage{booktabs, multirow}

\usepackage{amsmath}
\usepackage{amsthm}
\usepackage{amssymb}
\usepackage{bbm}

\usepackage{mathtools}

\usepackage{soul}

\usepackage{lipsum}

\usepackage[utf8]{inputenc}
\usepackage{hyperref}
\hypersetup{
	unicode,
	pdfauthor={Author One, Author Two, Author Three},
	pdftitle={A simple article template},
	pdfsubject={A simple article template},
	pdfkeywords={article, template, simple},
	pdfproducer={LaTeX},
	pdfcreator={pdflatex}
}

\usepackage[sort&compress,numbers,square]{natbib}
\theoremstyle{plain}
\newtheorem{theorem}{Theorem}
\newtheorem{corollary}[theorem]{Corollary}
\newtheorem{lemma}[theorem]{Lemma}

\newtheorem{proposition}[theorem]{Proposition}

\theoremstyle{definition}

\newtheorem{remark}[theorem]{Remark}

\usepackage{enumitem}

\makeatletter
\newcommand{\mylabel}[1]{%
	\item[\textbf{(#1)}]%
	\protected@edef\@currentlabel{%
		\noexpand\textbf{(#1)}%
	}%
}
\makeatother

\usepackage{graphicx, color}
\graphicspath{{fig/}}

\usepackage{mathrsfs} 

\usepackage{lipsum}

\title{A simple article template}
\author{Author One$^1$\thanks{Author One was partially supported by Grant XXX} \and Author Two$^2$ \and Author Three$^1$}

\date{
	$^1$Organization 1 \\ \texttt{\{auth1, auth3\}@org1.edu}\\%
	$^2$Organization 2 \\ \texttt{auth3@inst2.edu}\\[2ex]%
}

\title{A Deep Backward Scheme for Reflected BSDEs via Penalization : Approximation and Convergence Analysis}
\author{
E. H. Essaky \\
{\small Cadi Ayyad University} \\
{\small Faculty of Applied Sciences, Safi} \\
{\small Department of Mathematics and Computer Science} \\
{\small P.O. Box 4162, Safi, Morocco.} \\
{\small e-mail : essaky@uca.ac.ma}
\and
M. A. Jalal \\
{\small Cadi Ayyad University} \\
{\small Faculty of Applied Sciences, Safi} \\
{\small Department of Mathematics and Computer Science} \\
{\small P.O. Box 4162, Safi, Morocco.} \\
{\small e-mail : m.jalal.ced@uca.ac.ma}
\and
Y. Ouknine \\
{\small Mohammed VI Polytechnic University} \\
{\small Africa Business School} \\
{\small Avenue Mohammed Ben Abdellah Regragui,} \\
{\small Madinat Al Irfane, BP 6380, Rabat, Morocco.} \\
{\small e-mail : youssef.ouknine@um6p.ma}
}

\date{}

\begin{document}
\maketitle

\begin{abstract}
	
	We propose and analyze a deep backward numerical scheme for Markovian reflected backward stochastic differential equations with a generator independent of the control variable, based on the penalization method. The reflection constraint is replaced by penalized BSDEs, which are approximated backward in time by neural networks, with particular attention to the dependence of the errors on the penalization parameter \(\lambda\).
	We derive quantitative error estimates for both the value process \(Y\) and the control process \(Z\). A key ingredient is an \(L^2\)-regularity estimate for the control component with explicit dependence on \(\lambda\). Exploiting the monotonicity of the penalization term, we obtain stability estimates uniform in \(\lambda\) and prove uniform spatial Lipschitz bounds for both the penalized value functions and the discrete backward targets. Consequently, the GroupSort Lipschitz constraints and the associated neural-network approximation error can be controlled uniformly in \(\lambda\).
	The global analysis separates penalization, time-discretization, forward, and neural-network approximation errors. Balancing the dominant terms yields a mean-square error of order \(O(|\pi|^{1/3})\) for the value component, together with a corresponding estimate for a post-training reconstruction of \(Z\). Numerical experiments with an explicit reflected BSDE and American option problems illustrate the scheme on one-dimensional and structured high-dimensional benchmarks.
	
	\medskip
	\noindent\textbf{Keywords:} Reflected backward stochastic differential equations, penalization method, deep learning, backward dynamic programming, GroupSort networks, error analysis. \\ \\ 
	\textbf{AMS Classification}\textbf{: 65C30,
		60H35, 60H20, 68T07.}
	
\end{abstract}
	
\section{Introduction}
	Backward stochastic differential equations (BSDEs), originating from
the work of Bismut \cite{bismut} and subsequently developed in their
nonlinear form by Pardoux and Peng \cite{pardoux peng}, have become a
fundamental tool in stochastic analysis. They arise naturally in
stochastic control, mathematical finance, and the probabilistic
representation of nonlinear partial differential equations through
nonlinear Feynman--Kac formulas.

An important extension of the classical BSDE framework is provided by
reflected backward stochastic differential equations (RBSDEs),
introduced by El Karoui et al.\ \cite{karoui}. In this setting, the
first component of the solution is constrained to remain above a
prescribed obstacle. More precisely, given an obstacle process \(S\),
a solution is a triplet \((Y,Z,K)\) satisfying a backward equation of
the form
\[
Y_t
=
\xi
+
\int_t^T f(s,X_s,Y_s)\,\mathrm ds
+
K_T-K_t
-
\int_t^T Z_s\,\mathrm dW_s,
\qquad 0\leq t\leq T,
\]
together with the constraint
\[
Y_t\geq S_t,
\qquad 0\leq t\leq T,
\]
where \(K\) is a nondecreasing process satisfying the Skorokhod
minimality condition
\[
\int_0^T (Y_t-S_t)\,\mathrm dK_t=0.
\]
The process \(K\) therefore acts only when the solution reaches the
obstacle and represents the minimal correction required to maintain
the constraint. Under suitable Lipschitz and integrability conditions,
existence and uniqueness of the solution can be established. RBSDEs
are closely related to optimal stopping problems, American option
pricing, and obstacle problems for nonlinear parabolic PDEs.

From a numerical point of view, the presence of the reflection process
makes RBSDEs more difficult to approximate than standard BSDEs. A
classical way of handling the constraint is the penalization method.
For a penalization parameter \(\lambda>0\), the reflected equation is
replaced by the standard BSDE
\[
Y_t^\lambda
=
g(X_T)
+
\int_t^T
f^\lambda(s,X_s,Y_s^\lambda)\,\mathrm ds
-
\int_t^T Z_s^\lambda\,\mathrm dW_s,
\]
where
\[
f^\lambda(t,x,y)
=
f(t,x,y)
+
\lambda\bigl(y-h(t,x)\bigr)^-,
\qquad
r^-:=\max(-r,0),
\]
and the Markovian obstacle is given by
\[
S_t=h(t,X_t).
\]
As \(\lambda\) tends to infinity, the penalty increasingly prevents
violations of the constraint, and, under appropriate assumptions,
\((Y^\lambda,Z^\lambda)\) converges to the corresponding components
\((Y,Z)\) of the reflected BSDE.

The penalization approach is particularly convenient from a numerical
perspective since, for each fixed \(\lambda\), it reduces the reflected
problem to a standard BSDE. The latter can then be discretized in time,
leading to a backward dynamic programming scheme involving conditional
expectations. In high dimension, however, the numerical evaluation of
these conditional expectations becomes a major computational
difficulty and motivates the use of neural-network approximations.\\


The numerical approximation of BSDEs is by now a well-established field. Discrete-time schemes based on backward induction and the computation of conditional expectations were developed among others by Zhang \cite{zhang} and Bouchard and Touzi \cite{touzi}, providing quantitative error estimates under standard Lipschitz assumptions. A central ingredient in these analyses is the \(L^2\)-path regularity of the control component, which plays a crucial role in the convergence analysis of time-discretization schemes. Classical implementations typically rely on regression Monte Carlo or quantization methods to approximate the conditional expectations arising at each time step; their computational cost, however, may increase significantly with the dimension.

The numerical approximation of reflected BSDEs has also been extensively studied. In the Markovian setting, Ma and Zhang \cite{ma zhang} established representation and regularity results for reflected BSDEs and introduced discrete-time approximation procedures, showing in particular how the \(L^2\)-regularity of the martingale integrand determines the convergence rate of the numerical scheme. Bouchard and Touzi \cite{touzi} extended their discrete-time and Monte Carlo methodology to the reflected case. Subsequently, Bouchard and Chassagneux \cite{bouchard chassagneux} provided a detailed analysis of discrete-time approximations for both discretely and continuously reflected BSDEs, obtaining convergence estimates for the components of the reflected solution. These works provide the main theoretical foundations for time-discretization schemes for reflected BSDEs.

An alternative approach consists in approximating the reflection through penalization. The reflected BSDE is then replaced by a family of standard BSDEs whose drivers contain a penalty term enforcing the obstacle constraint. Lepeltier and Xu \cite{lepeltier} established the convergence of penalized solutions for an r.c.l.l.\ obstacle, while Gobet and Wang \cite{gobet} obtained quantitative estimates with respect to the penalization parameter. From a numerical perspective, the main advantage of this approach is that it reduces the reflected problem to a sequence of standard BSDEs, making it possible to combine penalization with existing numerical approximation methods for BSDEs.

Combining penalization with a deep backward scheme for reflected BSDEs raises a difficulty absent from the non reflected case. The approximation error now accumulates four contributions: penalization, time discretization, Euler scheme for the forward process, and neural-network approximation, which cannot be handled independently, since the penalized generator $f^\lambda$ is Lipschitz in $y$ with a
constant of order $\lambda$. A direct application of the standard stability estimates therefore yields bounds that grow with the penalization parameter and the number of time steps must be coupled, together with a careful handling of the Lipschitz constants. Resolving this point is the starting point of our analysis.

In this paper, we propose and analyze a deep backward scheme for reflected BSDEs based on penalization. We restrict our analysis to Markovian reflected BSDEs whose generator is independent of the control variable \(z\). This structure allows us to construct the approximation of the \(Y\)-component without introducing a separate neural network for \(Z\).

For a fixed penalization parameter \(\lambda\), we first approximate the forward diffusion by its Euler scheme. The penalized BSDE is then discretized backward in time by an implicit Euler scheme. At each time step, the deterministic function representing the \(Y\)-component is approximated by a neural network through the minimization of a quadratic loss. Thus, the neural networks are constructed recursively backward from the terminal condition rather than being trained jointly over the whole time grid.

A central feature of the proposed method is that only the \(Y\)-component is involved in the training procedure. Once the approximations $ \mathcal U_i^\lambda(X_{i}^\pi)$, 
have been computed, an approximation of the control process can be reconstructed by
\[ \mathcal Z_i^\lambda = \frac{1}{\Delta t_i} \mathbb E_{t_i} \left[ \mathcal U_{i+1}^\lambda(X_{i+1}^\pi) \Delta W_{t_i} \right]. \]
Consequently, the approximation of \(Z\) is obtained as a post-processing step and does not require the training of an additional neural network.

Our analysis separates the different sources of approximation error. For the penalized solution, we establish an estimate of the form
\[ \max_{0\leq i\leq N-1} \left\| Y_{t_i}^\lambda - \mathcal U_i^\lambda(X_{i}^\pi) \right\|_2 \leq C\left[ |\pi|^{1/2} +\lambda|\pi|^{1/2} + \sum_{i=0}^{N-1} \sqrt{\varepsilon_i^{\mathcal N,v^\lambda}} \right], \]

where
\[ \varepsilon_i^{\mathcal N,v^\lambda} = \inf_{\theta} \mathbb E \left[ \left| v_i^\lambda(X_{i}^\pi) - \mathcal U(X_{i}^\pi;\theta) \right|^2 \right], \]
denotes the neural-network approximation error at time \(t_i\) for a deterministic function $v_i^\lambda$. This implies,
\[ \max_{0\leq i\leq N-1} \mathbb E \left[ \left| Y_{t_i}^\lambda - \mathcal U_i^\lambda(X_{i}^\pi) \right|^2 \right] \leq C\left[ |\pi| + \lambda^2|\pi| + N \sum_{i=0}^{N-1}\varepsilon_i^{\mathcal N,v^\lambda} \right]. \]

This estimate makes explicit the interaction between the time discretization, the penalization parameter, and the neural-network approximation errors, and requires a careful handling of the Lipschitz property of $f^\lambda$, to ensure that the constant $C$ in these estimates does not depend on the penalty parameter $\lambda$.

For the reconstructed control process, we establish the quantitative estimate
\[ \mathbb E \left[ \sum_{i=0}^{N-1} \int_{t_i}^{t_{i+1}} \left| Z_t^\lambda-\mathcal Z_i^\lambda \right|^2 \,\mathrm dt \right] \leq C\left[ |\pi| + \lambda^2|\pi| + N\sum_{i=0}^{N-1} \varepsilon_i^{\mathcal N,v^\lambda} \right]. \]

This provides an explicit error estimate for \(Z\), even though no neural network is directly trained to approximate this component.

To obtain quantitative bounds for the neural-network approximation errors, we use GroupSort neural networks. Their main advantage in our setting is that they provide quantitative uniform approximation of Lipschitz functions while preserving a prescribed Lipschitz bound. We prove that the deterministic functions that we approximate with the backward scheme have Lipschitz constants uniformly bounded independently of \(\lambda\). The quantitative approximation properties of GroupSort networks can then be used to control~\(\varepsilon_i^{\mathcal N,v^\lambda}\) explicitly in terms of the network architecture.

By combining the penalization error with the discretization and neural-network approximation errors, we establish the global error estimate between the exact value process $Y$ of the RBSDE and the learned approximation
\begin{equation}
\max_{0\leq i\leq N-1} \mathbb E \left[ \left| Y_{t_i} - \mathcal U_i^\lambda(X_{i}^\pi) \right|^2 \right] \leq C|\pi|^\frac{1}{3},
\label{eq:main_rate1}
\end{equation}

and for the reconstructed control process 
\begin{equation}
\mathbb E \left[ \sum_{i=0}^{N-1} \int_{t_i}^{t_{i+1}} \left| Z_t-\mathcal Z_i^\lambda \right|^2\mathrm{d}t \right]\leq C|\pi|^\frac{1}{3},
\label{eq:main_rate2}
\end{equation}
with $\lambda = c_0|\pi|^{-1/3}$ for some $c_0>0$.\\

The remainder of the paper is organized as follows. Section~\ref{sec:deep learning} introduces the deep learning setting. Section~\ref{sec:preliminary} presents the mathematical framework of RBSDEs and the penalization approach, together with the main assumptions and the reference convergence results. Section~\ref{sec:deep_scheme} describes the DBDP-PBSDE algorithm, loss function and training procedure. Section~\ref{sec:L2 regularity} establishes the $L^2$-regularity results for the solution of the penalized BSDE that are needed to derive the convergence rate for the $Z$-component. Section~\ref{sec:convergence} contains the convergence analysis and the proofs of the main estimates \eqref{eq:main_rate1} and
\eqref{eq:main_rate2}. Section~\ref{sec:numerics} reports the numerical experiments, including American option pricing, and illustrates the accuracy of the scheme in high dimension.

\section{Approximation of functions using deep neural networks}
\label{sec:deep learning}
Let \(d_0,d_1\in\mathbb N^*\) denote the input and output dimensions, respectively. We fix \(\ell\in\mathbb N^*\), the number of hidden layers, and let \( m=(m_0,\ldots,m_{\ell-1}) \) denote the numbers of neurons in the hidden layers.

A feedforward neural network is a function from \(\mathbb R^{d_0}\) to \(\mathbb R^{d_1}\) defined as the composition
\begin{equation*}
\Phi = \phi_\ell\circ\phi_{\ell-1}\circ\cdots\circ\phi_0,
\end{equation*}
where
\[ \phi_0:\mathbb R^{d_0}\longrightarrow\mathbb R^{m_0},\qquad \phi_i:\mathbb R^{m_{i-1}}\longrightarrow\mathbb R^{m_i}, \qquad i=1,\ldots,\ell-1,\quad \text{and}\quad \phi_\ell:\mathbb R^{m_{\ell-1}}\longrightarrow\mathbb R^{d_1}.\]

The first \(\ell\) layer functions are of the form
\[ \phi_i(x)=\rho_i(\mathcal W_i x+b_i), \qquad i=0,\ldots,\ell-1, 
\quad\text{whereas the output layer is affine:}\quad
 \phi_\ell(x)=\mathcal W_\ell x+b_\ell. \]
Here, \(\mathcal W_i\) and \(b_i\) are respectively the weight matrix and bias vector of the \(i\)-th layer. Usually, the activation functions are applied componentwise through a one-dimensional activation function
\[ \widehat\rho_i:\mathbb R\longrightarrow\mathbb R,\]
namely
\[ \rho_i(x_1,\ldots,x_{m_i}) = \bigl( \widehat\rho_i(x_1),\ldots,\widehat\rho_i(x_{m_i}) \bigr). \]

Standard examples include the sigmoid, ReLU, ELU, and hyperbolic tangent activation functions.

We denote by
\[ \mathcal N_{d_0,d_1,\ell,m}^{\rho} \]
the set of feedforward neural networks with input dimension \(d_0\), output dimension \(d_1\), \(\ell\) hidden layers, hidden-layer dimensions \(m=(m_i)_{i=0,\ldots,\ell-1}\), and activation functions \(\rho=(\rho_i)_{i=0,\ldots,\ell-1}\).

A given neural network \(\Phi\in\mathcal N_{d_0,d_1,\ell,m}^{\rho}\) is determined by its weight and bias parameters
\[ \theta = (\mathcal W_0,b_0,\ldots,\mathcal W_\ell,b_\ell) \in \Theta_m, \]
where $\Theta_m$ denotes the set of all possible parameters of networks of $\mathcal N_{d_0,d_1,\ell,m}^{\rho}$. We shall sometimes write~\(\Phi=\Phi(\cdot;\theta)\).\\

The fundamental universal approximation theorem of \cite{hornik 89} justifies the use of neural networks as function approximators.

\begin{theorem}[Universal approximation theorem]
\label{thm:universal-approximation}
Assume that the activation function \(\rho\) is continuous and nonconstant. Then
\[ \mathcal N_{d_0,d_1,\ell}^{\rho} := \bigcup_{m_0,\ldots,m_{\ell-1}\geq1} \mathcal N_{d_0,d_1,\ell,m}^{\rho} \]
is dense in \(L^2(\nu)\), for any finite measure \(\nu\) on \(\mathbb R^{d_0}\).
\end{theorem}

In particular, for every \(f\in L^2(\nu)\) and every \(\varepsilon>0\), there exist an architecture \(m=(m_0,\ldots,m_{\ell-1})\) and a neural network \(\Phi\in\mathcal N_{d_0,d_1,\ell,m}^{\rho}\) such that 
\[ \int_{\mathbb R^{d_0}} |f(x)-\Phi(x)|^2\,\nu(\mathrm{d}x) \leq\varepsilon. \]

The universal approximation theorem ensures density but does not provide an explicit approximation rate. In our error analysis, we need a quantitative approximation result for Lipschitz functions. For this purpose, following \cite{germain}, we consider GroupSort neural networks.

Let \(\kappa\in\mathbb N^*\), \(\kappa\geq2\), be the grouping size, and assume that
$ m_i=\kappa n_i,\, i=0,\ldots,\ell-1,$ 
for some \(n_i\in\mathbb N^*\). The quantity $\displaystyle{\sum_{i=0}^{\ell-1}m_i} $
will be referred to as the width of the network, while \(\ell+1\) is its depth.

We denote $\rho^\kappa = (\rho^\kappa_i)_{i=0,\ldots,\ell-1}$ the sequence of GroupSort activations. For each $i\in\{0,\ldots,\ell-1\}$, $\rho^\kappa_i$ divides its input into groups of size \(\kappa\) and sorts the entries of each group in decreasing order. More precisely, if \(m_i=\kappa n_i\) and \(x=(x_1,\ldots,x_{m_i})\in\mathbb R^{m_i}\), then
\[ \rho^\kappa_i(x) = \Big( \operatorname{sort}(x_1,\ldots,x_\kappa), \ldots, \operatorname{sort} (x_{(n_i-1)\kappa+1},\ldots,x_{n_i\kappa}) \Big).\]

We denote by \(\mathcal S_{d_0,\ell,m}^{\rho^\kappa}\) the class of scalar-valued GroupSort neural networks
\[ \Phi(\cdot,\theta) \in \mathcal N_{d_0,1,\ell,m}^{\rho^\kappa},\qquad \theta = (\mathcal W_0,b_0,\ldots,\mathcal W_\ell,b_\ell), \]

whose parameters satisfy
\begin{equation*}
\sup_{|x|_2=1} |\mathcal W_0x|_\infty \leq1, \qquad \sup_{|x|_\infty=1} |\mathcal W_i x|_\infty \leq1, \quad i=1,\ldots,\ell,
\quad  \text{and} \quad 
|b_j|_\infty\leq B, \qquad j=0,\ldots,\ell,
\end{equation*}
for some \(B>0\). Under these constraints, the corresponding GroupSort neural networks are \(1\)-Lipschitz.

For \(K>0\), we denote by \(\mathcal G_{K,d_0,\ell,m}^{\rho^\kappa}\) the set of functions
$ \Psi:\mathbb R^{d_0}\longrightarrow\mathbb R, $
of the form
\begin{equation*}
\Psi(x) = K\beta\, \Phi\left(\frac{x+\alpha}{\beta}\right), \qquad \Phi\in \mathcal S_{d_0,\ell,m}^{\rho^\kappa}, \quad \alpha\in\mathbb R^{d_0}, \quad \beta>0.
\end{equation*}
We recall also that every function in \(\mathcal G_{K,d_0,\ell,m}^{\rho^\kappa}\) is \(K\)-Lipschitz.

We shall use the following quantitative approximation result for Lipschitz functions, which follows from the GroupSort approximation result recalled in~\cite{germain}.
	
\begin{proposition}
\label{thm:univ approx2}
Let $R>0$ and let $ f:[-R,R]^{d_0}\longrightarrow\mathbb R $ be a \(K\)-Lipschitz function. Then, for every \(\varepsilon>0\), there exists a GroupSort neural network $ g\in \mathcal G_{K,d_0,\ell,m}^{\rho^\kappa} $ such that
\begin{equation*}
\sup_{x\in[-R,R]^{d_0}} |f(x)-g(x)| \leq 2RK\varepsilon.
\end{equation*}
If \(d_0>1\), the network can be chosen with grouping size  $ \kappa = \left\lceil \frac{2\sqrt{d_0}}{\varepsilon} \right\rceil $, depth $ \ell+1=O(d_0^2)$, and width $ \sum_{i=0}^{\ell-1}m_i = O\left( \left( \frac{2\sqrt{d_0}}{\varepsilon} \right)^{d_0^2-1} \right).$ If \(d_0=1\), the same result holds with grouping size $ \kappa = \left\lceil\frac1{\varepsilon}\right\rceil$,  depth $ \ell+1=3$, and width $ \sum_{i=0}^{\ell-1}m_i =  O\left(\frac1{\varepsilon}\right).$
\end{proposition}
	
\section{Preliminaries, notations and assumptions}
\label{sec:preliminary}
\subsection{Notations}
Let \((W_t)_{0\leq t\leq T}\) be a \(d\)-dimensional Brownian motion defined on a filtered probability space
$
(\Omega,\mathcal F,(\mathcal F_t)_{0\leq t\leq T},\mathbb P),
$
where \((\mathcal F_t)_{0\leq t\leq T}\) is the standard augmentation of the filtration generated by \(W\).\\

For \(p\geq2\), we introduce the following notation and spaces:
\begin{itemize}
\item For \(x\in\mathbb R^d\), \(d\in\mathbb N^*\), we denote by \(|x|\) its Euclidean norm.

\item For \(t\in[0,T]\), we use the notation
\[ \mathbb E_t[\cdot] := \mathbb E[\cdot\mid\mathcal F_t] \]
for the conditional expectation with respect to \(\mathcal F_t\).

\item \(\mathcal P\) denotes the predictable \(\sigma\)-algebra on \([0,T]\times\Omega\) associated with the filtration \((\mathcal F_t)_{0\leq t\leq T}\).
	
	\item \(\mathcal S_c^{p,d}\) is the space of all \(\mathbb R^d\) valued, adapted, continuous processes \(Y=(Y_t)_{0\leq t\leq T}\) such that
\[ \mathbb E\left[ \sup_{0\leq t\leq T}|Y_t|^p \right]<\infty. \]
	
	\item \(\mathcal L^{p,d}\) is the space of all \(\mathbb R^d\)-valued, \(\mathcal P\)-measurable processes \(Z=(Z_t)_{0\leq t\leq T}\) such that
\[ \mathbb E\left[ \left( \int_0^T|Z_t|^2\mathrm{d}t \right)^{p/2} \right]<\infty. \]
	
	\item \(\mathcal K_c^p\) is the space of all processes \(K\in\mathcal S_c^{p,1}\) such that \(K\) is nondecreasing and \(K_0=0\).
	\item If $\varphi$ is a Lipschitz continuous function, we denote by $[\varphi]$ its Lipschitz constant.
\end{itemize}

\subsection{RBSDE and approximation via penalization}
\subsubsection{Reflected BSDEs in Markovian case}

Throughout this paper, we consider the following forward stochastic differential equation:
\begin{equation}
	\label{SDE}
	X_t
	=
	x_0
	+\int_0^t \mu(s,X_s)\,\mathrm{d}s
	+\int_0^t \sigma(s,X_s)\,\mathrm{d}W_s,
	\qquad 0\leq t\leq T,
\end{equation}
where
\[
\mu:[0,T]\times\mathbb{R}^d\longrightarrow\mathbb{R}^d,
\qquad
\sigma:[0,T]\times\mathbb{R}^d\longrightarrow\mathbb{R}^{d\times d}
\]
are measurable functions and \(x_0\in\mathbb{R}^d\).

We are also given measurable functions
\[
f:[0,T]\times\mathbb{R}^d\times\mathbb{R}\times\mathbb{R}^d
\longrightarrow\mathbb{R},
\qquad
h:[0,T]\times\mathbb{R}^d\longrightarrow\mathbb{R},
\qquad
g:\mathbb{R}^d\longrightarrow\mathbb{R},
\]
representing, respectively, the driver, the obstacle, and the terminal condition.

We consider the Markovian reflected backward stochastic differential equation (RBSDE) associated with the forward process \(X\), the driver \(f\), the terminal condition \(g(X_T)\), and the obstacle \(h(t,X_t)\). A solution is a triplet of stochastic processes \((Y,Z,K)\) satisfying
\begin{equation}
	\label{RBSDE}
	\begin{cases}
		(Y,Z,K)\in
		\mathcal{S}_c^{2,1}\times
		\mathcal{L}^{2,d}\times
		\mathcal{K}_c^2,
		\\[2mm]
		\displaystyle
		Y_t
		=
		g(X_T)
		+\int_t^T f(s,X_s,Y_s,Z_s)\,\mathrm{d}s
		+K_T-K_t
		-\int_t^T Z_s\,\mathrm{d}W_s,
		\qquad 0\leq t\leq T,
		\\[2mm]
		Y_t\geq h(t,X_t),
		\qquad 0\leq t\leq T,\quad \mathbb{P}\text{-a.s.},
		\\[2mm]
		\displaystyle
		\int_0^T
		\bigl(Y_t-h(t,X_t)\bigr)\,\mathrm{d}K_t
		=0.
	\end{cases}
\end{equation}

We impose the following assumptions on the data.

\begin{enumerate}
	\mylabel{H\textsubscript{L}}\label{HL}
	
	\item \label{HL_1}
	There exist constants \(C>0\) and \([\mu],[\sigma]\geq0\) such that,
	for all \(t_1,t_2\in[0,T]\) and \(x_1,x_2\in\mathbb R^d\),
	\[
	|\mu(t_1,x_1)-\mu(t_2,x_2)|
	\leq
	C|t_1-t_2|^{1/2}
	+
	[\mu]|x_1-x_2|,
	\]
	and
	\[
	|\sigma(t_1,x_1)-\sigma(t_2,x_2)|
	\leq
	C|t_1-t_2|^{1/2}
	+
	[\sigma]|x_1-x_2|.
	\]
	
	\item \label{HL_2}
	The coefficients of the forward equation satisfy
	\[
	\sup_{0\leq t\leq T}|\mu(t,0)|
	+
	\sup_{0\leq t\leq T}|\sigma(t,0)|
	<+\infty.
	\]
	
	\item \label{HL_3}
	There exist constants \(C>0\) and \([f]\geq0\) such that, for all
	$
	(t_1,x_1,y_1,z_1),
	(t_2,x_2,y_2,z_2)
	\in
	[0,T]\times\mathbb R^d\times\mathbb R\times\mathbb R^d,
$
	\[
	\begin{aligned}
		&
		|f(t_1,x_1,y_1,z_1)-f(t_2,x_2,y_2,z_2)|
		\leq
		C|t_1-t_2|^{1/2}
		+
		[f]\bigl(
		|x_1-x_2|
		+
		|y_1-y_2|
		+
		|z_1-z_2|
		\bigr).
	\end{aligned}
	\]
	
	\item \label{HL_4}
	The generator satisfies
	\[
	\sup_{0\leq t\leq T}
	|f(t,0,0,0)|
	<+\infty.
	\]
	
	\item \label{HL_5}
	There exist constants \(C>0\) and \([h]\geq0\) such that, for all
	\(t_1,t_2\in[0,T]\) and \(x_1,x_2\in\mathbb R^d\),
	\[
	|h(t_1,x_1)-h(t_2,x_2)|
	\leq
	C|t_1-t_2|^{1/2}
	+
	[h]|x_1-x_2|.
	\]
	
	\item \label{HL_6}
	There exists a constant \([g]\geq0\) such that, for all
	\(x_1,x_2\in\mathbb R^d\),
	\[
	|g(x_1)-g(x_2)|
	\leq
	[g]|x_1-x_2|.
	\]
	
	\item \label{HL_7}
	The terminal condition is compatible with the obstacle:
	\[
	h(T,x)\leq g(x),
	\qquad
	x\in\mathbb R^d.
	\]
	
\end{enumerate}

It is well known that under Hypotheses \ref{HL}(\ref{HL_1},\ref{HL_2}), the SDE \eqref{SDE} admits a unique strong solution $X\in \mathcal{S}^{2,d}_c$. Moreover, for every \(p\geq 2\),
\begin{equation}\label{Xmoment}
\mathbb{E}\left[\sup_{0\leq t\leq T}|X_t|^p\right]\leq C_p(1+|x_0|^p),
\end{equation}
and
\begin{equation}
\mathbb E\left[|X_t-X_s|^p\right] \leq C_p(1+|x_0|^p)|t-s|^{p/2},\qquad 0\leq s<t\leq T,
\label{eq:standard_X_estimate}
\end{equation}
where \(C_p\) depends only on \(p\), \(T\), and the coefficients of the forward SDE.

On the other hand, it is known from \cite{karoui} that, under \ref{HL}, there exists a triplet of stochastic processes $(Y,Z,K)$ satisfying the RBSDE \eqref{RBSDE}.

\subsubsection{Existence via penalization}
For $\lambda>0$, we consider the following penalized BSDE (PBSDE)
\begin{eqnarray}
	Y^\lambda_t &= g(X_T) + \displaystyle\int_t^Tf(s,X_s,Y^\lambda_s,Z^\lambda_s)\mathrm{d}s + \lambda\int_t^T(Y_s^\lambda-h(s,X_s))^-\mathrm{d}s - \int_t^TZ^\lambda_s\mathrm{d}W_s \nonumber\\
	&= g(X_T) + \displaystyle\int_t^Tf^\lambda(s,X_s,Y^\lambda_s,Z^\lambda_s)\mathrm{d}s - \int_t^TZ^\lambda_s\mathrm{d}W_s ,\label{PBSDE}
\end{eqnarray}

where 
$$f^\lambda(t,x,y,z) = f(t,x,y,z) + \lambda(y-h(s,x))^-.$$

Equivalently, by defining
$$ K_t^\lambda := \lambda\int_0^t(Y_s^\lambda-h(s,X_s))^-\mathrm{d}s ,$$

we can write
\begin{equation*}
Y_t^\lambda = g(X_T) + \int_t^T f(s,X_s,Y_s^\lambda,Z^\lambda_s)\,\mathrm{d}s + K_T^\lambda-K_t^\lambda - \int_t^T Z_s^\lambda\,\mathrm{d}W_s.
\end{equation*}

From \cite{karoui,gobet}, under \ref{HL}, the solution $(Y^\lambda,Z^\lambda,K^\lambda)$ of penalized BSDE \eqref{PBSDE} satisfies the following uniform estimates
\begin{equation}
\sup_{\lambda\geq1} \mathbb E \left[ \sup_{0\leq t\leq T}|Y_t^\lambda|^2 + \int_0^T|Z_s^\lambda|^2\mathrm{d}s + |K_T^\lambda|^2 \right] \leq C,
\label{eq:uniform-penalization}
\end{equation}
		
and the penalized solution converges to the solution of the RBSDE as \(\lambda\to+\infty\), namely
$$ \lim_{\lambda\to+\infty}\mathbb{E}\left[\sup_{0\leq  t\leq  T}|Y_t^\lambda - Y_t|^2 + \int_0^T|Z_t^\lambda - Z_t|^2\mathrm{d}t + \sup_{0\leq  t\leq  T}|K_t^\lambda - K_t|^2\right] = 0.$$

On the other hand, from the theory of Markovian BSDEs, for fixed $\lambda>0$, we know that the process $Y^\lambda$ is a deterministic function of $X$, namely we have 
$$Y_t^\lambda = u^\lambda(t,X_t) \;\; , \;\; t\in[0,T] $$
for some deterministic function $ u^\lambda : [0,T]\times\mathbb{R}^d\rightarrow \mathbb{R}$ (see e.g. \cite{zhang book}). This function satisfies the Markov representation
$$ u^\lambda(t,x) = Y^{\lambda, t, x}_t , \qquad (t,x) \in[0,T]\times \mathbb{R}^d, $$

where $Y^{\lambda, t,x}$ is the solution of PBSDE \eqref{PBSDE} associated to the forward process $X^{t,x}$ that starts from $x$ at time $t$. Moreover, $u^\lambda$ is associated to the following semi-linear parabolic PDE

\begin{equation*} 
\left\{
\begin{aligned}
\displaystyle \partial_t u(t,x) + \mathcal{L} u(t,x) + f(t,x,u(t,x),\sigma^\top D_x u(t,x)) + \lambda(u(t,x)-h(t,x))^- &= 0 \; , \; (t,x)\in[0,T[\times\mathbb{R}^d\\
u(T,x) &= g(x) \; , \; x\in\mathbb{R}^d,
\end{aligned}
\right.
\end{equation*}

where $\mathcal{L}$ denotes the infinitesimal generator associated with the forward diffusion
$$ \mathcal{L}u = \frac{1}{2}\text{Tr}(\sigma \sigma^\top D_x^2 u) + \mu\cdot D_x u .$$

Since the complexity and approximation error of the GroupSort networks introduced in Section~\ref{sec:convergence} depend explicitly on the Lipschitz constant of the target function, it is important to determine how the spatial Lipschitz regularity of the penalized value function \(u^\lambda(t,\cdot)\) depends on the penalization parameter \(\lambda\). The following proposition shows, in fact, that the family of penalized value functions is uniformly Lipschitz continuous with respect to the state variable. Thus, although the Lipschitz constant of the penalized generator \(f^\lambda\) may grow with \(\lambda\), this growth does not propagate to the spatial Lipschitz constant of the corresponding value function $u^\lambda$.
\begin{proposition}
	\label{thm:lip}
	Assume \ref{HL}(\ref{HL_1}--\ref{HL_6}). Then there exists a constant \(C>0\), such that for every 
	\(t\in[0,T]\), \(x,x'\in\mathbb R^d\), and \(\lambda>0\), 
	\begin{equation*}
		|u^\lambda(t,x)-u^\lambda(t,x')|
		\leq
		C|x-x'|.
	\end{equation*}
	In particular, $[u^\lambda]\leq C$, $\lambda>0.$
\end{proposition}

\begin{proof}
	Fix \(t\in[0,T]\) and \(x,x'\in\mathbb R^d\). Let
$
	(X^{t,x},Y^{\lambda,t,x},Z^{\lambda,t,x})
	\quad\text{and}\quad
	(X^{t,x'},Y^{\lambda,t,x'},Z^{\lambda,t,x'})
$
	denote the solutions of the forward SDE, starting from \(x\) and \(x'\), respectively, at time \(t\), and the corresponding penalized BSDEs.
	
	For \(s\in[t,T]\), set
	\[
	\delta X_s
	:=
	X_s^{t,x}-X_s^{t,x'},
	\qquad
	\delta Y_s
	:=
	Y_s^{\lambda,t,x}-Y_s^{\lambda,t,x'},
	\qquad
	\delta Z_s
	:=
	Z_s^{\lambda,t,x}-Z_s^{\lambda,t,x'},
	\]
	and
	\[
	\delta g_T
	:=
	g(X_T^{t,x})-g(X_T^{t,x'}),
	\qquad
	\delta h_s
	:=
	h(s,X_s^{t,x})-h(s,X_s^{t,x'}).
	\]
	
	The difference $\delta Y$ of the two penalized BSDEs satisfies
	\begin{align}
		\delta Y_s
		={}&
		\delta g_T
		+
		\int_s^T
		\Big[
		f(r,X_r^{t,x},Y_r^{\lambda,t,x},Z_r^{\lambda,t,x})
		-
		f(r,X_r^{t,x'},Y_r^{\lambda,t,x'},Z_r^{\lambda,t,x'})
		\Big]\mathrm{d}r
		\nonumber\\
		&+
		\lambda
		\int_s^T
		\Big[
		(Y_r^{\lambda,t,x}-h(r,X_r^{t,x}))^-
		-
		(Y_r^{\lambda,t,x'}-h(r,X_r^{t,x'}))^-
		\Big]\mathrm{d}r
	-
		\int_s^T
		\delta Z_r\,\mathrm{d}W_r.
		\label{eq:delta-BSDE-lip}
	\end{align}
	
	We linearize the generator. Define
	\[
	\beta_r
	:=
	\frac{
		f(r,X_r^{t,x},Y_r^{\lambda,t,x},Z_r^{\lambda,t,x})
		-
		f(r,X_r^{t,x},Y_r^{\lambda,t,x'},Z_r^{\lambda,t,x})
	}{
		\delta Y_r
	}
	\mathbf 1_{\{\delta Y_r\neq0\}},
	\]
	and
	\[
	\rho_r
	:=
	f(r,X_r^{t,x},Y_r^{\lambda,t,x'},Z_r^{\lambda,t,x'})
	-
	f(r,X_r^{t,x'},Y_r^{\lambda,t,x'},Z_r^{\lambda,t,x'}).
	\]
	
	For the \(z\)-variable, define
	\begin{align}
		\gamma_r
		:={}&
		\frac{
			f(r,X_r^{t,x},Y_r^{\lambda,t,x'},Z_r^{\lambda,t,x})
			-
			f(r,X_r^{t,x},Y_r^{\lambda,t,x'},Z_r^{\lambda,t,x'})
		}{
			|\delta Z_r|^2
		}
		\delta Z_r
		\nonumber \times
		\mathbf 1_{\{|\delta Z_r|\neq0\}}.
		\label{eq:gamma-definition}
	\end{align}
	Then
	\[
	\gamma_r\delta Z_r = \langle\gamma_r,\delta Z_r\rangle
	=
	f(r,X_r^{t,x},Y_r^{\lambda,t,x'},Z_r^{\lambda,t,x})
	-
	f(r,X_r^{t,x},Y_r^{\lambda,t,x'},Z_r^{\lambda,t,x'}).
	\]
	
	By \ref{HL},
	\begin{equation}
		|\beta_r|
		\leq
		[f],
		\qquad
		|\gamma_r|
		\leq
		[f],
		\qquad
		|\rho_r|
		\leq
		[f]|\delta X_r|.
		\label{eq:beta-gamma-rho-bounds}
	\end{equation}
	
	For the penalization term, define
	\begin{align}
		\alpha_r^\lambda
		:={}&
		\lambda
		\frac{
			(Y_r^{\lambda,t,x}-h(r,X_r^{t,x}))^-
			-
			(Y_r^{\lambda,t,x'}-h(r,X_r^{t,x'}))^-
		}{
			\delta Y_r-\delta h_r
		}
		\nonumber 
		\mathbf 1_{\{\delta Y_r-\delta h_r\neq0\}}.
		\label{eq:alpha-lambda-definition}
	\end{align}
	Since the mapping $r\longmapsto r^-$
	is non-increasing and \(1\)-Lipschitz, we have
	\begin{equation*}
		-\lambda
		\leq
		\alpha_r^\lambda
		\leq
		0.
	\end{equation*}
	
	Consequently, \eqref{eq:delta-BSDE-lip} can be written as
	\begin{equation*}
		\delta Y_s
		=
		\delta g_T
		+
		\int_s^T
		\left[
		\beta_r\delta Y_r
		+
		\gamma_r\delta Z_r
		+
		\rho_r
		+
		\alpha_r^\lambda(\delta Y_r-\delta h_r)
		\right]\mathrm{d}r
		-
		\int_s^T
		\delta Z_r\,\mathrm{d}W_r.
	\end{equation*}
Let $\eta>0$, and define
	\begin{equation*}
		\Theta_s
		:=
		\exp\left(
		\eta(s-t)
		+
		\int_t^s
		\alpha_r^\lambda\,\mathrm{d}r
		\right),
		\qquad
		s\in[t,T].
	\end{equation*}
	Since \(\alpha_r^\lambda\leq0\),
	\begin{equation}
		0<\Theta_s \leq e^{\eta T}\exp{\left(\int_t^s\alpha^\lambda_r\mathrm{d}r\right)}\leq e^{\eta T},
		\qquad s\in[t,T].
		\label{eq:Theta-bound}
	\end{equation}
	Applying It\^o's formula to $\Theta_s|\delta Y_s|^2$ between \(t\) and \(T\), we obtain
	\begin{align}
		\label{eq:Ito-uniform-lip}
		|\delta Y_t|^2
		+
		\int_t^T
		\Theta_r|\delta Z_r|^2\,\mathrm{d}r
		={}&
		\Theta_T|\delta g_T|^2
	+
		\int_t^T
		\Theta_r
		\Big[
		-(\eta+\alpha_r^\lambda)|\delta Y_r|^2
		+
		2\beta_r|\delta Y_r|^2
	\\
		&
		+
		2\delta Y_r
		\gamma_r\delta Z_r
		+
		2\rho_r\delta Y_r
		\nonumber
		+
		2\alpha_r^\lambda
		\delta Y_r(\delta Y_r-\delta h_r)
		\Big]\mathrm{d}r
		-
		2\int_t^T
		\Theta_r\delta Y_r\delta Z_r\,\mathrm{d}W_r.
	\end{align}
	
By \eqref{eq:beta-gamma-rho-bounds}, we have
\begin{equation*}
	2|\delta Y_r|
	\left| \gamma_r\delta Z_r
	\right|
	\leq
	2[f]|\delta Y_r||\delta Z_r|
	\leq
	\frac12|\delta Z_r|^2
	+
	2[f]^2|\delta Y_r|^2.
\end{equation*}
Similarly,
	\begin{equation*}
		2|\rho_r||\delta Y_r|
		\leq
		|\rho_r|^2
		+
		|\delta Y_r|^2.
	\end{equation*}
	For the penalization term, since $\alpha_r^\lambda\leq0$, we have
	\begin{align*}
		-\alpha_r^\lambda|\delta Y_r|^2
		+
		2\alpha_r^\lambda
		\delta Y_r(\delta Y_r-\delta h_r)
		&=
		\alpha_r^\lambda|\delta Y_r|^2
		-
		2\alpha_r^\lambda\delta Y_r\delta h_r
		=
		\alpha_r^\lambda
		|\delta Y_r-\delta h_r|^2
		-
		\alpha_r^\lambda|\delta h_r|^2
		\\
		&\leq
		-\alpha_r^\lambda|\delta h_r|^2,
	\end{align*}
	Combining these estimates and taking $\eta=2[f]+2[f]^2+1$, we obtain
\begin{align}
-(\eta+\alpha_r^\lambda)|\delta Y_r|^2 + 2\beta_r|\delta Y_r|^2 + & 2\delta Y_r \gamma_r\delta Z_r + 2\rho_r\delta Y_r+ 2\alpha_r^\lambda \delta Y_r(\delta Y_r-\delta h_r) \nonumber\\
	&\qquad\leq\frac12|\delta Z_r|^2 + |\rho_r|^2 - \alpha_r^\lambda|\delta h_r|^2.
\label{eq:drift-final}
\end{align}
	
Substituting \eqref{eq:drift-final} into \eqref{eq:Ito-uniform-lip}, and taking expectation after using a localization procedure, we obtain
\begin{align}
\mathbb E\left[|\delta Y_t|^2\right] + \frac12 \mathbb E\left[ \int_t^T \Theta_r|\delta Z_r|^2\mathrm{d}r\right] &\leq \mathbb E \left[ \Theta_T|\delta g_T|^2 \right] +  \mathbb E\left[\int_t^T \Theta_r|\rho_r|^2\,\mathrm{d}r \right]\nonumber\\
	&+ \mathbb E \left[\int_t^T \Theta_r (-\alpha_r^\lambda)| \delta h_r|^2\mathrm{d}r\right].
\label{eq:Y-delta-three-terms}
\end{align}
By using the Lipschitz continuity of \(g\), \(h\), and \(f\), and  \eqref{eq:Theta-bound}, the first two terms on the right-hand side of \eqref{eq:Y-delta-three-terms} satisfy
\begin{equation}
\mathbb E \left[ \Theta_T|\delta g_T|^2 \right] + \mathbb E\left[ \int_t^T \Theta_r|\rho_r|^2\mathrm{d}r\right] \leq C \mathbb E \left[ \sup_{t\leq r\leq T} |\delta X_r|^2 \right].
\label{eq:first-two-terms}
\end{equation}
	
	For the last term, using \eqref{eq:Theta-bound}, we obtain
\begin{equation*}
\int_t^T \Theta_r (-\alpha_r^\lambda) |\delta h_r|^2\mathrm{d}r \leq C[h]^2 \sup_{t\leq r\leq T} |\delta X_r|^2 \int_t^T (-\alpha_r^\lambda) \exp\left( \int_t^r \alpha_u^\lambda\,\mathrm{d}u \right)\mathrm{d}r.
\end{equation*}
	
Since
\begin{align}
			\int_t^T
		(-\alpha_r^\lambda)
		\exp\left(
		\int_t^r
		\alpha_u^\lambda\,\mathrm{d}u
		\right)
		\mathrm{d}r
		\nonumber=
		1-
		\exp\left(
		\int_t^T
		\alpha_u^\lambda\,\mathrm{d}u
		\right)
		\leq
		1,
	\end{align}
it follows that,
	\begin{equation}
		\mathbb E
		\left[\int_t^T
		\Theta_r
		(-\alpha_r^\lambda)
		|\delta h_r|^2\mathrm{d}r\right]
		\leq
		C
		\mathbb E
		\left[
		\sup_{t\leq r\leq T}
		|\delta X_r|^2
		\right],
		\label{eq:penalty-final}
	\end{equation}
	where the constant \(C\) is independent of \(\lambda\).
	
	Combining \eqref{eq:Y-delta-three-terms},
	\eqref{eq:first-two-terms}, and
	\eqref{eq:penalty-final}, we obtain
	\begin{equation*}
		\mathbb E\left[|\delta Y_t|^2\right]
		\leq
		C
		\mathbb E
		\left[
		\sup_{t\leq r\leq T}
		|\delta X_r|^2
		\right].
	\end{equation*}
	
	Finally, under \ref{HL}(\ref{HL_1},\ref{HL_2}), the standard stability estimate for the forward SDE gives
	\begin{equation*}
		\mathbb E\left[|\delta Y_t|^2\right]
		\leq
		C|x-x'|^2.
	\end{equation*}
	
	By the Markovian representation, we have
	\[
	|u^\lambda(t,x)-u^\lambda(t,x')|^2
	\leq
	C|x-x'|^2.
	\]
	Taking square roots, and changing the constant \(C\) if necessary,
	we conclude that
	\[
	|u^\lambda(t,x)-u^\lambda(t,x')|
	\leq
	C|x-x'|,
	\]
	where \(C\) is independent of \(\lambda\). This completes the proof.
\end{proof}

\section{Deep Backward Dynamic Programming scheme for penalized BSDE}
\label{sec:deep_scheme}

\subsection{Discrete-time approximation}
\label{sec:euler approximation}

Throughout the remainder of the paper, we assume that \(f\) does not
depend on \(z\), that is,
$$ f(t,x,y,z)=f(t,x,y), \qquad (t,x,y,z)\in [0,T]\times\mathbb R^d\times\mathbb R\times\mathbb R^d.$$

We consider a time discretization, which is with no loss of generality, supposed to be uniform
$$ \pi:\quad 0=t_0<t_1<\cdots<t_N=T, $$
and set
$$ \Delta t_i=t_{i+1}-t_i, \qquad \Delta W_i=W_{t_{i+1}}-W_{t_i}, \qquad |\pi|=\frac{T}{N}.$$

The Euler approximation \(X^\pi\) of the diffusion process \(X\) is
defined recursively by
$$ X_{i+1}^\pi = X_{i}^\pi + \mu(t_i,X_i^\pi)\Delta t_i + \sigma(t_i,X_i^\pi)\Delta W_{i}, \qquad X_0^\pi=x_0, $$
for \(i=0,\ldots,N-1\). Moreover, under \ref{HL}(\ref{HL_1},\ref{HL_2}), the standard Euler estimates yield (see \cite{platen})
\begin{equation}
\label{platen}
\max_{0\leq i\leq N-1} \mathbb E\left[ |X_{t_i}-X_i^\pi|^2 + \sup_{t_i\leq t\leq t_{i+1}} |X_t-X_i^\pi|^2 \right] \leq C|\pi|,
\end{equation}
for some constant \(C>0\) independent of \(|\pi|\). Fix \(\lambda>0\). For \(i\in\{0,\ldots,N-1\}\), the penalized BSDE
satisfies
\begin{align}
Y_{t_i}^\lambda &= Y_{t_{i+1}}^\lambda + \int_{t_i}^{t_{i+1}} f^\lambda(s,X_s,Y_s^\lambda)\mathrm{d}s - \int_{t_i}^{t_{i+1}} Z_s^\lambda\,\mathrm{d}W_s \nonumber\\
	&\approx  Y_{t_{i+1}}^\lambda + f^\lambda(t_i,X_{t_i}, Y_{t_{i}}^\lambda)\Delta t_i -  Z_{t_{i}}^\lambda\Delta W_i. \label{BSDE_increment}
\end{align}

Taking conditional expectation with respect to \(\mathcal F_{t_i}\) in the first equation gives
$$ Y_{t_i}^\lambda = \mathbb E_{t_i} \left[ Y_{t_{i+1}}^\lambda + \int_{t_i}^{t_{i+1}} f^\lambda(s,X_s,Y_s^\lambda)\mathrm{d}s \right].
$$

The backward Euler approximation consists in freezing the generator at \(t_i\), \(X_i^\pi\), and the value at time \(t_i\). This formally leads to the implicit relation
$$ Y_{t_i}^\lambda \approx \mathbb E_{t_i} \left[ Y_{t_{i+1}}^\lambda \right] + f^\lambda(t_i,X_i^\pi,Y_{t_i}^\lambda)\Delta t_i.$$

For the $Z^\lambda$-component, multiplying the BSDE increment \eqref{BSDE_increment} by \(\Delta W_i\) and taking conditional expectation we write
$$ Z_{t_i}^\lambda \approx \frac{1}{\Delta t_i} \mathbb{E}_{t_i} \left[ Y_{t_{i+1}}^\lambda \Delta W_i \right], $$

which motivates the discrete-time approximation
\begin{equation}
\label{Z_tilde}
\widetilde Z_i^\lambda := \frac{1}{\Delta t_i} \mathbb{E}_{t_i} \left[ Y_{t_{i+1}}^\lambda \Delta W_i \right].
\end{equation}
Notice that \(\widetilde Z_i^\lambda\) should be interpreted as an approximation of the time average of \(Z^\lambda\) over \([t_i,t_{i+1}]\), rather than, in general, as the point-wise value \(Z_{t_i}^\lambda\).\\

In the Markovian setting, it is natural to represent $\widetilde Z_i^\lambda$ as
$ \widetilde Z_i^\lambda = \widetilde z_i^\lambda(X_{t_i}), $
where
\begin{equation*}
\widetilde z_i^\lambda(x) = \frac{1}{\Delta t_i} \mathbb{E} \left[ u^\lambda(t_{i+1},X_{t_{i+1}}) \Delta W_i\,\middle|\, X_{t_i}=x \right].
\end{equation*}

\subsection{DBDP for the penalized BSDE}
\label{subsec:DBDP-penalized}

We now combine the Euler approximation of the forward process with a
neural-network approximation of the Markovian value functions
\(u^\lambda(t_i,\cdot)\). Under the standing assumption that \(f\) is independent of \(z\), no neural network is introduced for the control process \(Z^\lambda\), and the approximation of the value component can be trained independently of the control component. This leads to a reduced DBDP scheme involving only one neural network at each time step. At each time level \(t_i\), a network $\mathcal U(\cdot;\theta): \mathbb R^d\longrightarrow\mathbb R, $ belonging to the class
$\mathcal N_{d,1,\ell,m}^{\rho}$ is used to approximate the function
\(u^\lambda(t_i,\cdot)\). We denote this approximation by $\mathcal{U}_i^\lambda$.\\

Suppose that \(\mathcal U_{i+1}^\lambda\) has already been constructed.
For \(\theta\in\Theta_m\), define
\begin{equation}
\label{lossY}
L_i^\lambda(\theta) := \mathbb E\Big[ \big| \mathcal U_{i+1}^\lambda(X_{i+1}^\pi) - \mathcal U(X_{i}^\pi;\theta) + f^\lambda \big( t_i, X_{i}^\pi, \mathcal U(X_{i}^\pi;\theta) \big) \Delta t_i \big|^2 \Big].
\end{equation}

The parameter \(\theta_{i,\lambda}^*\) of network $\mathcal U_{i}^\lambda$ is obtained by minimizing the loss $L_i^\lambda$, that is
\begin{equation}
\label{minimisation_problem}
\theta_{i,\lambda}^* \in \underset{\theta\in\Theta_m}{\operatorname{arg\,min}}\,	L_i^\lambda(\theta).
\end{equation}

In practical implementations, the expectation in \eqref{lossY} is
replaced by an empirical average over simulated trajectories, and the
minimization problem \eqref{minimisation_problem} is solved approximately
using stochastic gradient methods.

The resulting backward algorithm is summarized as follow:

\begin{algorithm}[H]
	\caption{Deep Backward Dynamic Programming for Penalized BSDE
		(DBDP-PBSDE)}
	\label{alg:DBDP-PBSDE}
	\begin{algorithmic}[1]
		\State \textbf{Initialization:}
		\[
		\mathcal U_N^\lambda(x)=g(x),
		\qquad x\in\mathbb R^d.
		\]
		\For{\(i=N-1,\ldots,0\)}
		\State \textbf{Step 1: Choose the neural-network approximation.}
		\[
		\mathcal U(\cdot;\theta)
		\in
		\mathcal N_{d,1,\ell,m}^{\rho}.
		\]
		\State \textbf{Step 2: Define the loss function.}
		$$ L_i^\lambda(\theta) = \mathbb E\Big[ \big| \mathcal U_{i+1}^\lambda(X_{i+1}^\pi) - \mathcal U(X_{i}^\pi;\theta) + f^\lambda \big( t_i, X_{i}^\pi, \mathcal U(X_{i}^\pi;\theta) \big) \Delta t_i \big|^2 \Big].$$
		
		\State \textbf{Step 3: Train the neural network.}
		Determine
		\begin{equation*}
		\theta_{i,\lambda}^*
		\in
		\underset{\theta\in\Theta_m}{\operatorname{arg\,min}}\;
		L_i^\lambda(\theta), 
		\end{equation*}
		\State or approximately minimize \(L_i^\lambda\) using stochastic
		gradient descent.
		
		\State \textbf{Step 4: Update the value-function approximation.}
		\[
		\mathcal U_i^\lambda(\cdot)
		=
		\mathcal U(\cdot;\theta_{i,\lambda}^*).
		\]
		
		\EndFor
		
		\State \textbf{Return:}
		\[
		\left( \mathcal U_i^\lambda \right)_{0\leq i\leq N}.
		\]
		
	\end{algorithmic}
\end{algorithm}

\subsubsection{Reconstruction of the control process}
\label{subsubsec:Z-reconstruction}

Although the control process is not required for the construction of
the value approximation, it can be reconstructed after the
corresponding value function has been learned.

For \(i=0,\ldots,N-1\), define
\begin{equation}
\label{Zlearned}
\mathcal Z_i^\lambda:=\frac{1}{\Delta t_i}\mathbb E_{t_i}\left[\mathcal U_{i+1}^\lambda(X_{i+1}^\pi)\Delta W_i\right],
\end{equation}
where this choice is motivated by representation \eqref{Z_tilde}.

By the Markovian structure of the Euler scheme, there exists a
deterministic measurable function \( z_i^\lambda:\mathbb R^d\longrightarrow\mathbb R^d \)
such that
\[
\mathcal Z_i^\lambda = z_i^\lambda(X_{i}^\pi),
\]
where
\begin{equation}
\label{zfunction}
z_i^\lambda(x) = \frac{1}{\Delta t_i} \mathbb E \left[ \mathcal U_{i+1}^\lambda(X_{i+1}^\pi) \Delta W_i \,\middle|\, X_{i}^\pi=x \right].
\end{equation}

The quantity \(\mathcal Z_i^\lambda\) should therefore be viewed as a
reconstructed discrete approximation of the control process rather
than as an additional variable involved in the training of
\(\mathcal U_i^\lambda\). In particular, the approximation
\(\mathcal Z_i^\lambda\) does not enter the loss function
\eqref{lossY}.

\section{$L^2$-regularity results for the solution of PBSDE}
\label{sec:L2 regularity}
For the error analysis of the reconstructed control, it is useful to
introduce the \(L^2\)-projection
\begin{equation}
\bar Z_{i}^\lambda := \frac{1}{\Delta t_i} \mathbb E_{t_i} \left[ \int_{t_i}^{t_{i+1}} Z_t^\lambda\mathrm{d}t \right].
\label{eq:Zbar}
\end{equation}
The quantity \(\bar Z_{i}^\lambda\) represents the conditional time
average of the exact control process over
\([t_i,t_{i+1}]\). It is this quantity that naturally appears in the
\(L^2\)-regularity estimate for \(Z^\lambda\).\\

We define
$$ \varepsilon^{Z^\lambda}(\pi) := \mathbb{E}\left[\sum_{i=0}^{N-1}\int_{t_i}^{t_{i+1}}|Z_t^\lambda - \bar{Z}_{i}^\lambda|^2\mathrm{d}t\right].$$
	
By Theorem 5.2.4 in \cite{zhang book}, under \ref{HL}(\ref{HL_1}--\ref{HL_6}), we have for fixed $\lambda>0$ we have
$$\varepsilon^{Z^\lambda}(\pi) \leq  C_\lambda |\pi| ,$$
where $C_\lambda$ is a constant which depends on $\lambda$. In this section we will trace this dependence in order to obtain an explicit error estimate for the reconstructed control process of $Z$.\\

For fixed $\lambda>0$ and $\varepsilon>0$, let $(Y^{\lambda,\varepsilon},Z^{\lambda,\varepsilon})$ be the solution of regularized PBSDE with generator 
$$ f^{\lambda,\varepsilon}(t,x,y) = f(t,x,y) + \lambda\delta_\varepsilon(y-h(t,x)) ,$$
where $ \delta_\varepsilon : \mathbb{R}\rightarrow \mathbb{R}^+$ such that 
\[
\delta_\varepsilon(x)
=
\begin{cases}
	\dfrac{(x-\varepsilon)^2}{4\varepsilon}
	& |x|\leq\varepsilon\\[2mm]
	x^-
	& |x|>\varepsilon.
\end{cases}
\]

We have the following Lemma.

\begin{lemma}
	\label{lemma:delta-epsilon}
For every \(x,y\in\mathbb R\) and $\varepsilon>0$,
	\begin{equation}
		(x-y)\bigl(x^--\delta_\varepsilon(y)\bigr)
		\leq
		2\varepsilon^2.
		\label{eq:delta-key}
	\end{equation}
\end{lemma}

\begin{proof}
	First observe that \( \delta_\varepsilon(y)\geq y^-, y\in\mathbb R.\)
	
	If \(|y|>\varepsilon\), then
	\(\delta_\varepsilon(y)=y^-\), and since \(r\mapsto r^-\) is
	nonincreasing,
	\[
	(x-y)(x^--\delta_\varepsilon(y))
	=
	(x-y)(x^--y^-)
	\leq0.
	\]
	
	Assume now that \(|y|\leq\varepsilon\). If
	\[
	(x-y)(x^--\delta_\varepsilon(y))\leq0,
	\]
	there is nothing to prove. Otherwise the product is positive. The two
	factors must then be negative. Indeed, they cannot both be positive:
	if \(x>y\), then
	$
	x^-\leq y^-\leq\delta_\varepsilon(y),
	$
	it implies \(x^--\delta_\varepsilon(y)\leq0\), which leads to a contradiction.
	
	Hence
	\[
	x<y,
	\qquad
	x^-<\delta_\varepsilon(y).
	\]
	Since \(\delta_\varepsilon(y)\leq\varepsilon\), we have
	\(-x<\varepsilon\). Together with \(y\leq\varepsilon\), this gives
	$
	0<y-x<2\varepsilon.
$
	Moreover,
	\[
	0<
	\delta_\varepsilon(y)-x^-
	\leq\varepsilon.
	\]
	Therefore,
	\[
	\begin{aligned}
		(x-y)(x^--\delta_\varepsilon(y))
		&=
		(y-x)(\delta_\varepsilon(y)-x^-)\\
		&\leq
		2\varepsilon^2.
	\end{aligned}
	\]
\end{proof}

We have the following convergence result.

\begin{theorem}
	\label{thm bounds}
	Assume that \ref{HL}(\ref{HL_1}-\ref{HL_6}) hold true. Then there exists a constant \(C>0\), independent of \(\lambda\geq1\) and
	\(\varepsilon>0\), such that
	\begin{equation}
		\label{eq:Z-convergence}
		\mathbb E\left[
		\sup_{0\leq t\leq T}
		|Y_t^\lambda-Y_t^{\lambda,\varepsilon}|^2
		+
		\int_0^T
		|Z_t^\lambda-Z_t^{\lambda,\varepsilon}|^2\,\mathrm{d}t
		\right]
		\leq
		C\lambda\varepsilon^2.
	\end{equation}
	Consequently, for every fixed \(\lambda\geq1\),
	\[
	\lim_{\varepsilon\to0}
	\mathbb E\left[
	\sup_{0\leq t\leq T}
	|Y_t^\lambda-Y_t^{\lambda,\varepsilon}|^2
	+
	\int_0^T
	|Z_t^\lambda-Z_t^{\lambda,\varepsilon}|^2\,\mathrm{d}t
	\right]
	=0.
	\]
\end{theorem}

\begin{proof}
	Set
	\[
	\Delta Y_t
	:=
	Y_t^\lambda-Y_t^{\lambda,\varepsilon},
	\qquad
	\Delta Z_t
	:=
	Z_t^\lambda-Z_t^{\lambda,\varepsilon}.
	\]
	Applying It\^o's formula to \(|\Delta Y_t|^2\) on \([t,T]\) and
	taking expectations give
	\begin{align*}
		\mathbb E\left[|\Delta Y_t|^2\right]
		+
		\mathbb E\left[\int_t^T|\Delta Z_s|^2\,\mathrm{d}s\right]
		&=
		2\mathbb E\left[\int_t^T
		\Delta Y_s
		\left[
		f(s,X_s,Y_s^\lambda)
		-
		f(s,X_s,Y_s^{\lambda,\varepsilon})
		\right]\mathrm{d}s\right]
		\\
		&\quad+
		2\lambda\mathbb E\Bigg[\int_t^T
		\Delta Y_s
		\left[
		(Y_s^\lambda-h(s,X_s))^- 
		- \delta_\varepsilon
		(Y_s^{\lambda,\varepsilon}-h(s,X_s))
		\right]\mathrm{d}s \Bigg].
	\end{align*}
	
	For the penalization term,  applying Lemma \ref{lemma:delta-epsilon} with
$
	x=Y_s^\lambda-h(s,X_s),
	\quad
	y=Y_s^{\lambda,\varepsilon}-h(s,X_s),
$
	we obtain
	\[
	\Delta Y_s
	\left[
	(Y_s^\lambda-h(s,X_s))^-
	-
	\delta_\varepsilon
	(Y_s^{\lambda,\varepsilon}-h(s,X_s))
	\right]
	\leq
	2\varepsilon^2.
	\]
	
	Combining this inequality with the Lipschitz continuity of $f$ in $y$, yields
	\begin{equation}
		\mathbb E\left[|\Delta Y_t|^2\right]
		+
		\mathbb E\left[\int_t^T|\Delta Z_s|^2\,\mathrm{d}s\right]
		\leq
		C\int_t^T
		\mathbb E\left[|\Delta Y_s|^2\right]\mathrm{d}s
		+
		C\lambda\varepsilon^2.
		\label{eq:regularized-basic}
	\end{equation}
	
	By Gronwall's lemma,
	\begin{equation}
		\sup_{0\leq t\leq T}
		\mathbb E\left[|\Delta Y_t|^2\right]
		\leq
		C\lambda\varepsilon^2.
		\label{eq:Y-reg-conv}
	\end{equation}
	Returning to \eqref{eq:regularized-basic} and taking \(t=0\), we also
	obtain
	\begin{equation}
		\mathbb E\left[\int_0^T|\Delta Z_s|^2\mathrm{d}s\right]
		\leq
		C\lambda\varepsilon^2.
		\label{eq:Z-reg-conv}
	\end{equation}
	
	It remains to estimate the supremum of \(\Delta Y\). Keeping the
	martingale term in It\^o's formula and using
	\eqref{eq:delta-key}, we obtain
	\[
	\sup_{0\leq t\leq T}|\Delta Y_t|^2
	\leq
	C\int_0^T|\Delta Y_s|^2\,\mathrm{d}s
	+
	C\lambda\varepsilon^2
	+
	2\sup_{0\leq t\leq T}
	\left|
	\int_t^T
	\Delta Y_s\Delta Z_s\,\mathrm{d}W_s
	\right|.
	\]
	By the Burkholder--Davis--Gundy inequality,
	\begin{align*}
		&\mathbb E\left[
		\sup_{0\leq t\leq T}
		\left|
		\int_t^T
		\Delta Y_s\Delta Z_s\,\mathrm{d}W_s
		\right|
		\right] \leq
		C\,
		\mathbb E\left[
		\sup_{0\leq s\leq T}|\Delta Y_s|
		\left(
		\int_0^T|\Delta Z_s|^2\,\mathrm{d}s
		\right)^{1/2}
		\right].
	\end{align*}
	Using Young's inequality,
	\[
	\begin{aligned}
		&\mathbb E\left[
		\sup_{0\leq t\leq T}
		\left|
		\int_t^T
		\Delta Y_s\Delta Z_s\,\mathrm{d}W_s
		\right|
		\right] \leq
		\frac14
		\mathbb E\left[
		\sup_{0\leq s\leq T}|\Delta Y_s|^2
		\right]
		+
		C
		\mathbb E\left[\int_0^T|\Delta Z_s|^2\mathrm{d}s\right].
	\end{aligned}
	\]
	Using \eqref{eq:Y-reg-conv} and \eqref{eq:Z-reg-conv} and absorbing
	the first term into the left-hand side, we obtain
	\[
	\mathbb E\left[
	\sup_{0\leq t\leq T}|\Delta Y_t|^2
	\right]
	\leq
	C\lambda\varepsilon^2,
	\]
this completes the proof.
\end{proof}

We consider the following assumption :
\begin{enumerate}\mylabel{H\textsubscript{D}}\label{HD}
	\item\label{HD_1} $\mu$ and $\sigma$ are continuously differentiable in $x$, with bounded derivatives uniformly on $t$.
	\item\label{HD_2} The derivatives of $\mu$ and $\sigma$ are Lipschitz in $x$,  uniformly in $t$.
	\item\label{HD_3} $f$ is continuously differentiable in $(x,y)$ with bounded derivatives uniformly on $t$, and $g$ and $h$ are continuously differentiable in $x$ with bounded derivatives uniformly on $t$.
\end{enumerate}

For later use, we recall the following results:
\begin{lemma}[Lemma 5.2.3 in \cite{zhang book}]
	\label{lemma:nabla}
	Let
$
	\bigl(Y^{\lambda,\varepsilon},
	Z^{\lambda,\varepsilon}\bigr),
$
	be the solution of the regularized penalized BSDE. Assume that
	\ref{HL}(\ref{HL_1}--\ref{HL_6}) and \ref{HD}(\ref{HD_1},\ref{HD_3}) hold. Then the associated value function
	\(u^{\lambda,\varepsilon}\) is continuously differentiable with respect
	to the space variable and the following hold
	\begin{equation*}
		\partial_x
		u^{\lambda,\varepsilon}(t,X_t)
		=
		\nabla Y_t^{\lambda,\varepsilon}
		(\nabla X_t)^{-1},
	\end{equation*}

	\begin{equation}
		Z_t^{\lambda,\varepsilon}
		=
		\nabla Y_t^{\lambda,\varepsilon}
		(\nabla X_t)^{-1}
		\sigma(t,X_t).
		\label{representation_Z}
	\end{equation}
Moreover 
	\begin{align*}
		\nabla X_t
		&=
		I_d
		+
		\int_0^t
		\partial_x\mu(s,X_s)\nabla X_s\,\mathrm ds
		+
		\sum_{j=1}^d
		\int_0^t
		\partial_x\sigma^j(s,X_s)
		\nabla X_s\,\mathrm dW_s^j,
	\end{align*}
	and
	\begin{align*}
		\nabla Y_t^{\lambda,\varepsilon}
		={}
		\partial_xg(X_T)\nabla X_T
		+	\nonumber &
		\int_t^T
		\Big[
		\partial_xf^{\lambda,\varepsilon}
		\bigl(s,X_s,Y_s^{\lambda,\varepsilon}\bigr)
		\nabla X_s
		+
		\partial_yf^{\lambda,\varepsilon}
		\bigl(s,X_s,Y_s^{\lambda,\varepsilon}\bigr)
		\nabla Y_s^{\lambda,\varepsilon}
		\Big]\,\mathrm ds
	\\ &
		-
		\sum_{k=1}^d
		\int_t^T
		\nabla Z_s^{\lambda,\varepsilon,k}\,
		\mathrm dW_s^k.
	\end{align*}
	Here \(\sigma^j\) denotes the \(j\)-th column of \(\sigma\), $\partial_x\mu$ takes values in $\mathbb{R}^{d\times d}$ with its $(i,j)$-$th$ component $\partial_{x_j}\mu^i$, and similarly for the other matrix valued derivatives.
\end{lemma}

\begin{lemma}[Theorem 2.4 \cite{peter imkeller}]
\label{lemma:nabla 2}
Assume that \ref{HL}(\ref{HL_1},\ref{HL_2}) and \ref{HD}(\ref{HD_1},\ref{HD_2}) hold true, and there exists a constant $c>0$ such that
\begin{equation} \label{deg:lemma:7} y^\top\sigma(t,x)\sigma^\top(t,x)y\geq  c|y|^2 \;\; \;\; \forall(t,x,y)\in[0,T]\times\mathbb{R}^d\times\mathbb{R}^d,
\end{equation}

then for any $t\in[0,T]$, $\nabla X_t$ is invertible, its inverse $(\nabla X_t)^{-1}$ satisfies a linear SDE, and for any $p\geq  2$ there exists a positive constant $C_p$ such that
\begin{equation}
\label{nabla X}
\mathbb{E}\left[\sup_{0\leq  t\leq  T} |\nabla X_t|^p\right]^{1/p} + \mathbb{E}\left[\sup_{0\leq  t\leq  T} |(\nabla X_t)^{-1}|^p\right]^{1/p} \leq  C_p
\end{equation}

and 
\begin{equation}
\label{nabla X 2}
\mathbb{E}\left[\sup_{s\leq  u\leq  t} |\nabla X_u - \nabla X_s|^p	 + \sup_{s\leq  u\leq  t} |(\nabla X_u)^{-1} - (\nabla X_s)^{-1}|^p\right] \leq  C_p|t-s|^{p/2} .
\end{equation}
\end{lemma}

Let $i\in\{1,\cdots,d\}$, we have
\begin{align}
(\nabla Y^{\lambda,\varepsilon}_t)^i &= \partial_xg(X_T)(\nabla X_T)^i + \int_t^T\left(\partial_x f^{\lambda,\varepsilon}(s,X_s,Y^{\lambda,\varepsilon}_s)(\nabla X_s)^i +\partial_y f^{\lambda,\varepsilon}(s,X_s,Y^{\lambda,\varepsilon}_s)(\nabla Y_s^{\lambda,\varepsilon})^i \right)\mathrm{d}s \label{lin}\\
		&\hspace{4cm}- \int_t^T\left(\left(\nabla Z^{\lambda,\varepsilon}_s\right)^\top\right)^i\mathrm{d}W_s, \nonumber
\end{align}

where $(\nabla Y^{\lambda,\varepsilon}_t)^i$ is the $i$-$th$ component of the (row) vector $\nabla Y^{\lambda,\varepsilon}_t$, $(\nabla X_t)^i$ is the $i$-$th$ column of the matrix $\nabla X_t$, and $\left(\left(\nabla Z^{\lambda,\varepsilon}_s\right)^\top\right)^i$ is the $i$-$th$ row of the transposed matrix of $\nabla Z^{\lambda, \varepsilon} = \begin{pmatrix} \nabla Z^{\lambda, \varepsilon,1} \\ \vdots \\ \nabla Z^{\lambda, \varepsilon,d} \end{pmatrix} $.

We prove now the following proposition:
\begin{proposition}
	\label{prop:p}
	Assume that the hypotheses of Lemmas~\ref{lemma:nabla} and
	\ref{lemma:nabla 2} hold. Then, for every \(p\geq2\), there exists a
	constant \(C_p>0\), independent of \(\lambda\geq1\) and
	\(\varepsilon>0\), such that
	\begin{equation}
		\mathbb E\left[
		\sup_{0\leq t\leq T}
		|\nabla Y_t^{\lambda,\varepsilon}|^p
		\right]
		\leq C_p \qquad \text{and}\qquad
		\mathbb E\left[
		\left(
		\int_0^T
		|\nabla Z_t^{\lambda,\varepsilon}|^2\,\mathrm{d}t
		\right)^{p/2}
		\right]
		\leq
		C_p\lambda^{p/2}.
		\label{eq:nablaY_nablaZ_bounds}
	\end{equation}
\end{proposition}

\begin{proof}
	For \(i=1,\ldots,d\), set
	\[
	U_t^i
	=
	(\nabla Y_t^{\lambda,\varepsilon})^i,
	\qquad
	V_t^i
	=
	\left(
	(\nabla Z_t^{\lambda,\varepsilon})^\top
	\right)^i,
	\qquad
	\xi^i
	=
	\partial_xg(X_T)(\nabla X_T)^i.
	\]

\[a_s^{\lambda,\varepsilon} = \partial_yf^{\lambda,\varepsilon}(s,X_s,Y_s^{\lambda,\varepsilon}),\qquad b_s^i = \partial_x f^{\lambda,\varepsilon}(s,X_s,Y^{\lambda,\varepsilon}_s)(\nabla X_s)^i.\]

	Define
	\[
	q_s^{\lambda,\varepsilon}
	:=
	-\lambda
	\delta_\varepsilon'
	\bigl(
	Y_s^{\lambda,\varepsilon}-h(s,X_s)
	\bigr).
	\]
	Since \(-1\leq\delta_\varepsilon'\leq0\) we have
	\begin{equation*}
		0\leq q_s^{\lambda,\varepsilon}\leq\lambda.
	\end{equation*}
	
	Moreover, set
	\[
	\alpha_s
	:=
	\partial_yf(s,X_s,Y_s^{\lambda,\varepsilon}),\qquad 
	\beta_s^i
	:=
	\partial_xf(s,X_s,Y_s^{\lambda,\varepsilon})
	(\nabla X_s)^i,\qquad 
	\eta_s^i
	:=
	\partial_xh(s,X_s)(\nabla X_s)^i.
	\]
	Then
	\[
	a_s^{\lambda,\varepsilon}
	=
	\alpha_s-q_s^{\lambda,\varepsilon},
	\qquad
	b_s^i
	=
	\beta_s^i
	+
	q_s^{\lambda,\varepsilon}\eta_s^i,
	\]
	and the linear BSDE \eqref{lin} can be written as
	\begin{equation}
		U_t^i
		=
		\xi^i
		+
		\int_t^T
		\left[
		\beta_s^i
		+
		\alpha_sU_s^i
		+
		q_s^{\lambda,\varepsilon}
		(\eta_s^i-U_s^i)
		\right]\mathrm{d}s
		-
		\int_t^T
		V_s^i\,\mathrm{d}W_s.
		\label{eq:component-lin-improved}
	\end{equation}
	
	\medskip
	\noindent
	\textit{Step 1. Uniform estimate for \(\nabla Y^{\lambda,\varepsilon}\).} Let
$\Gamma_{t,s}^{\lambda,\varepsilon}
	:=
	\exp\left(
	\int_t^s
	\bigl(
	\alpha_r-q_r^{\lambda,\varepsilon}
	\bigr)\,\mathrm{d}r
	\right).
$
Since $\partial f_y$ is uniformly bounded by some constant $L_f$ we have
	\[
	0<
	\Gamma_{t,s}^{\lambda,\varepsilon}
	\leq
	e^{L_fT}
	\exp\left(
	-\int_t^s q_r^{\lambda,\varepsilon}\,\mathrm{d}r
	\right) \leq e^{L_fT}.
	\]
	Consequently,
	\begin{align}
		\int_t^T
		\Gamma_{t,s}^{\lambda,\varepsilon}
		q_s^{\lambda,\varepsilon}\,\mathrm{d}s
		&\leq
		e^{L_fT}
		\int_t^T
		q_s^{\lambda,\varepsilon}
		\exp\left(
		-\int_t^s
		q_r^{\lambda,\varepsilon}\,\mathrm{d}r
		\right)\mathrm{d}s
		\nonumber\\
		&=
		e^{L_fT}
		\left[
		1-
		\exp\left(
		-\int_t^Tq_r^{\lambda,\varepsilon}\,\mathrm{d}r
		\right)
		\right]
		\nonumber\\
		&\leq
		e^{L_fT}.
		\label{eq:Gamma-q}
	\end{align}
	
	One the other hand, since the BSDE \eqref{eq:component-lin-improved} is linear, it gives
	\[
	U_t^i
	=
	\mathbb E_t\left[
	\Gamma_{t,T}^{\lambda,\varepsilon}\xi^i
	+
	\int_t^T
	\Gamma_{t,s}^{\lambda,\varepsilon}
	\left(
	\beta_s^i+
	q_s^{\lambda,\varepsilon}\eta_s^i
	\right)\mathrm{d}s
	\right].
	\]
	By boundedness of \(\partial_xg\), \(\partial_xf\) and \(\partial_xh\), we have
$
	|\xi^i|
	+
	|\beta_s^i|
	+
	|\eta_s^i|
	\leq
	C\sup_{0\leq r\leq T}|\nabla X_r|.
	$
	Using \eqref{eq:Gamma-q}, we obtain
	\[
	|U_t^i|
	\leq
	C\,
	\mathbb E_t\left[
	\sup_{0\leq r\leq T}
	|\nabla X_r|
	\right],
	\]
	where \(C\) is independent of \(\lambda\) and \(\varepsilon\).
	
	Hence
	\[
	|\nabla Y_t^{\lambda,\varepsilon}|
	\leq
	C\,
	\mathbb E_t\left[
	\sup_{0\leq r\leq T}
	|\nabla X_r|
	\right].
	\]
	By Doob's maximal inequality and the estimate \eqref{nabla X} for
	\(\nabla X\), we have for $p\geq 2$
	\[
	\mathbb E\left[
	\sup_{0\leq t\leq T}
	|\nabla Y_t^{\lambda,\varepsilon}|^p
	\right]
	\leq
	C_p.
	\]
	This proves the first estimate.
	
	\medskip
	\noindent
\medskip
\noindent
\textit{Step 2. Estimate for \(\nabla Z^{\lambda,\varepsilon}\).} Applying It\^o's formula to \(|U_t^i|^2\) in
\eqref{eq:component-lin-improved} gives
\begin{align*}
	|U_t^i|^2
	+
	\int_t^T|V_s^i|^2\,\mathrm{d}s
	&=
	|\xi^i|^2
	+
	2\int_t^T U_s^i\beta_s^i\,\mathrm{d}s
	+
	2\int_t^T\alpha_s|U_s^i|^2\,\mathrm{d}s
	\\
	&\quad+
	2\int_t^T
	q_s^{\lambda,\varepsilon}
	U_s^i(\eta_s^i-U_s^i)\,\mathrm{d}s
	-
	2\int_t^T
	U_s^iV_s^i\,\mathrm{d}W_s.
\end{align*}

Since \(q_s^{\lambda,\varepsilon}\geq0\), we have 
\[ 2q_s^{\lambda,\varepsilon}
U_s^i(\eta_s^i-U_s^i)
=
-2q_s^{\lambda,\varepsilon}\left(U_s^i-\frac{\eta_s^i}{2}\right)^2
+q_s^{\lambda,\varepsilon}\frac{|\eta_s^i|^2}{2}
\leq
\frac12
q_s^{\lambda,\varepsilon}
|\eta_s^i|^2.
\]
Moreover,
\[
2U_s^i\beta_s^i
\leq
|U_s^i|^2+|\beta_s^i|^2,
\qquad
|\alpha_s|\leq C.
\]
Therefore, taking \(t=0\)
\begin{align*}
	\int_0^T|V_s^i|^2\,\mathrm{d}s
	\leq C\Bigg(
	&|\xi^i|^2
	+
	\sup_{0\leq s\leq T}|U_s^i|^2
	+
	\int_0^T|\beta_s^i|^2\,\mathrm{d}s
	+
	\int_0^T
	q_s^{\lambda,\varepsilon}|\eta_s^i|^2\,\mathrm{d}s
	+
	\left|
	\int_0^T U_s^iV_s^i\,\mathrm{d}W_s
	\right|
	\Bigg).
\end{align*}

Let $p\geq 2$, from the boundedness of
\(\partial_xg\), \(\partial_xf\), and \(\partial_xh\), the moment
estimates for \(\nabla X\), and \eqref{eq:nablaY_nablaZ_bounds}, we have
\[
\mathbb E\left[|\xi^i|^p\right]\leq C_p,
\qquad
\mathbb E\left[\sup_{0\leq s\leq T}|U_s^i|^p\right]\leq C_p,\qquad \text{and}\qquad
\mathbb E\left[
\left(
\int_0^T|\beta_s^i|^2\,\mathrm{d}s
\right)^{p/2}
\right]
\leq C_p.
\]
Consequently, it follows that
\begin{align*}
	\mathbb E\left[
	\left(
	\int_0^T|V_s^i|^2\,\mathrm{d}s
	\right)^{p/2}
	\right]
	\leq C_p\Bigg(
	&1
	+
	\mathbb E\left[
	\left(
	\int_0^T
	q_s^{\lambda,\varepsilon}|\eta_s^i|^2\,\mathrm{d}s
	\right)^{p/2}
	\right] +
	\mathbb E\left[
	\left|
	\int_0^T
	U_s^iV_s^i\,\mathrm{d}W_s
	\right|^{p/2}
	\right]
	\Bigg).
\end{align*}

By using \(0\leq q_s^{\lambda,\varepsilon}\leq\lambda\) and \(
|\eta_s^i|
\leq
C|\nabla X_s|,
\), together with \eqref{nabla X}, we obtain

\begin{equation}
	\mathbb E\left[
	\left(
	\int_0^T
	q_s^{\lambda,\varepsilon}|\eta_s^i|^2\,\mathrm{d}s
	\right)^{p/2}
	\right]
	\leq
	C_p\lambda^{p/2}.
	\label{eq:q-eta}
\end{equation}

To estimate the martingale term, one may first use a standard localization procedure, and then apply the Burkholder--Davis--Gundy inequality to have
\begin{align*}
	\mathbb E\left[
	\left|
	\int_0^T
	U_s^iV_s^i\,\mathrm{d}W_s
	\right|^{p/2}
	\right]
	&\leq
	C_p
	\mathbb E\left[
	\left(
	\int_0^T
	|U_s^i|^2|V_s^i|^2\,\mathrm{d}s
	\right)^{p/4}
	\right]\\
	&\leq
	C_p
	\mathbb E\left[
	\sup_{0\leq s\leq T}|U_s^i|^{p/2}
	\left(
	\int_0^T|V_s^i|^2\,\mathrm{d}s
	\right)^{p/4}
	\right].
\end{align*}
Thus, by Young's inequality, for every \(\eta>0\),
\begin{align*}
	\mathbb E\left[
	\left|
	\int_0^T
	U_s^iV_s^i\,\mathrm{d}W_s
	\right|^{p/2}
	\right]
	&\leq
	\eta
	\mathbb E\left[
	\left(
	\int_0^T|V_s^i|^2\,\mathrm{d}s
	\right)^{p/2}
	\right]+
	C_{p,\eta}
	\mathbb E\left[
	\sup_{0\leq s\leq T}|U_s^i|^p
	\right].
\end{align*}

Choosing \(\eta>0\) sufficiently small, and using
the uniform estimate of $U_s^i$, \eqref{eq:q-eta} and $\lambda\geq 1$, we obtain
\begin{align*}
\mathbb E\left[
\left( \int_0^T|V_s^i|^2\mathrm{d}s \right)^{p/2} \right] &\leq
C_p\left(1+\lambda^{p/2}\right)\\
	&\leq C_p\lambda^{p/2}.
\end{align*}

Finally,
\[
|\nabla Z_s^{\lambda,\varepsilon}|^2
=
\sum_{i=1}^d|V_s^i|^2.
\]
Since \(p/2\geq1\),
\[
\left(\sum_{i=1}^d a_i\right)^{p/2}
\leq
d^{p/2-1}\sum_{i=1}^d a_i^{p/2},
\qquad a_i\geq0.
\]
Therefore,
\[
\mathbb E\left[
\left(
\int_0^T
|\nabla Z_s^{\lambda,\varepsilon}|^2\,\mathrm{d}s
\right)^{p/2}
\right]
\leq
C_{p,d}\lambda^{p/2},
\]
which proves the second estimate.
\end{proof}

\begin{theorem}
	\label{thm:Z-regularity}
	Assume \ref{HL}(\ref{HL_1}--\ref{HL_6}), \ref{HD} and the non-degeneracy assumption (\ref{deg:lemma:7}). Then there exists a constant \(C>0\), independent of \(\lambda\geq1\) and $|\pi|$ such that
	\begin{equation}
		\varepsilon^{Z^\lambda}(\pi)
		\leq
		C\left(
		\lambda|\pi|
		+
		\lambda^2|\pi|^2
		\right),
		\label{eq:Z-regularity}
	\end{equation}
	
	In particular, if $\lambda|\pi|\leq C_0$ for some constant \(C_0>0\) independent of \(\lambda\) and \(|\pi|\), then
	\begin{equation}
		\varepsilon^{Z^\lambda}(\pi)
		\leq
		C\lambda|\pi|,
		\label{eq:Z-regularity-improved}
	\end{equation}
	where \(C>0\) may depend on \(C_0\), but is independent of
	\(\lambda\) and \(|\pi|\).
\end{theorem}

\begin{proof}
	Throughout the proof, \(C>0\) denotes a generic constant which may
	change from line to line, but is independent of
	\(\lambda\geq1\), \(\varepsilon>0\), and the partition \(\pi\).
	
	We first establish the estimate for the regularized penalized BSDE.
	For \(\varepsilon>0\), recall that
	\[
	Y_t^{\lambda,\varepsilon}
	=
	g(X_T)
	+
	\int_t^T
	f^{\lambda,\varepsilon}
	(s,X_s,Y_s^{\lambda,\varepsilon})\,\mathrm{d}s
	-
	\int_t^T
	Z_s^{\lambda,\varepsilon}\,\mathrm{d}W_s,
	\]
	where \(
	f^{\lambda,\varepsilon}(t,x,y)
	=
	f(t,x,y)
	+
	\lambda\delta_\varepsilon(y-h(t,x)).
	\)
	
	Set
	\[
	U_t:=\nabla Y_t^{\lambda,\varepsilon},
	\qquad
	B_t:=(\nabla X_t)^{-1},
	\qquad
	D_t:=B_t\sigma(t,X_t).
	\]
	By the representation formula \eqref{representation_Z},
	\begin{equation*}
		Z_t^{\lambda,\varepsilon}=U_tD_t.
	\end{equation*}
	
	\medskip
	\noindent
	\textit{Step 1. Estimates for \(D\).} Let \(p\geq2\). By the assumptions \ref{HL}(\ref{HL_1},\ref{HL_2}) on \(\sigma\),
	\eqref{Xmoment}, \eqref{nabla X}, and Hölder's inequality,
	\begin{align}
		\mathbb E\left[
		\sup_{0\leq t\leq T}|D_t|^p
		\right]
		&\leq
		\left(\mathbb E
		\sup_{0\leq t\leq T}|B_t|^{2p}
		\right)^{1/2}
		\left(\mathbb E
		\sup_{0\leq t\leq T}
		|\sigma(t,X_t)|^{2p}
		\right)^{1/2}
		\nonumber\\
		&\leq
		C_p(1+|x_0|^p).
		\label{eq:D-moment}
	\end{align}
	
	Moreover, for \(0\leq s\leq t\leq T\),
	\[
	D_t-D_s
	=
	(B_t-B_s)\sigma(t,X_t)
	+
	B_s
	\bigl(
	\sigma(t,X_t)-\sigma(s,X_s)
	\bigr).
	\]
	Hence, by Hölder's inequality, \eqref{nabla X 2},
	assumption \ref{HL}(\ref{HL_1},\ref{HL_2}), and the standard estimate
	\eqref{eq:standard_X_estimate},
	\begin{equation}
		\mathbb E\left[|D_t-D_s|^p\right]
		\leq
		C_p(1+|x_0|^p)|t-s|^{p/2},
		\qquad p\geq2.
		\label{estimate:D}
	\end{equation}
	
	\medskip
	\noindent
	\textit{Step 2. Decomposition of the increment of
		\(Z^{\lambda,\varepsilon}\).} Fix \(i\in\{0,\ldots,N-1\}\) and
	\(t\in[t_i,t_{i+1}]\). We have
	\[
	Z_t^{\lambda,\varepsilon}
	-
	Z_{t_i}^{\lambda,\varepsilon}
	=
	(U_t-U_{t_i})D_{t_i}
	+
	U_t(D_t-D_{t_i}).
	\]
	Consequently,
	\begin{equation}
		|Z_t^{\lambda,\varepsilon}
		-Z_{t_i}^{\lambda,\varepsilon}|^2
		\leq
		2|(U_t-U_{t_i})D_{t_i}|^2
		+
		2|U_t(D_t-D_{t_i})|^2.
		\label{eq:Z-decomposition-square}
	\end{equation}
	
	We first estimate the second term. By Proposition~\ref{prop:p} we have
	\begin{equation*}
		\mathbb E\left[
		\sup_{0\leq t\leq T}|U_t|^4
		\right]
		\leq C.
	\end{equation*}
	Therefore, by using Hölder's inequality and \eqref{estimate:D}
	\[
	\begin{aligned}
		\mathbb E\left[|U_t(D_t-D_{t_i})|^2\right]
		&\leq
		\left(\mathbb E|U_t|^4\right)^{1/2}		
		\left(\mathbb E|D_t-D_{t_i}|^4
		\right)^{1/2}\\
		&\leq
		C(1+|x_0|^2)(t-t_i).
	\end{aligned}
	\]
	Since \(x_0\) is fixed, its contribution is absorbed into \(C\).
	Hence
	\begin{align}
		\sum_{i=0}^{N-1}
		\mathbb E\left[
		\int_{t_i}^{t_{i+1}}
		|U_t(D_t-D_{t_i})|^2\,\mathrm{d}t\right] &\leq C
		\sum_{i=0}^{N-1}
		\int_{t_i}^{t_{i+1}}
		(t-t_i)\mathrm{d}t
		\nonumber\\
		&\leq C|\pi|.
		\label{eq:forward-term-improved}
	\end{align}
	
	\medskip
	\noindent
	\textit{Step 3. Estimate of \((U_t-U_{t_i})D_{t_i}\).} Using the notation in the proof of Proposition \ref{prop:p}, and set
	\[
	F_s^j
	:=
	b_s^j
	+
	a_s^{\lambda,\varepsilon}U_s^j,
	\qquad
	F_s=(F_s^1,\ldots,F_s^d).
	\]
	For \(t\in[t_i,t_{i+1}]\), we have
	\[U_t^i = \xi^i +
		\int_t^T
		\left[
		b_s^i
		+
		a_s^{\lambda,\varepsilon}U_s^i
		\right]\mathrm{d}s
		-
		\int_t^T
		V_s^i\,\mathrm{d}W_s, \]
		which gives
	\[
	U_t-U_{t_i}
	=
	-\int_{t_i}^tF_s\,\mathrm{d}s
	+
	\int_{t_i}^t
	\nabla Z_s^{\lambda,\varepsilon}\,\mathrm{d}W_s. \]
	Therefore,
	\begin{equation}
		|(U_t-U_{t_i})D_{t_i}|^2
		\leq
		2|D_{t_i}|^2
		\left|
		\int_{t_i}^tF_s\,\mathrm{d}s
		\right|^2 +
		2\left|
		\left(
		\int_{t_i}^t
		\nabla Z_s^{\lambda,\varepsilon}\,\mathrm{d}W_s
		\right)
		D_{t_i}
		\right|^2.
		\label{eq:A-increment-decomposition}
	\end{equation}
	
	We estimate the two terms separately.
	
For the first term, by the boundedness of \(\partial_yf\) and
	\(|\delta_\varepsilon'|\leq1\),
	\[
	|a_s^{\lambda,\varepsilon}|
	\leq
	C+\lambda
	\leq
	C\lambda,
	\qquad
	\lambda\geq1.
	\]
	Moreover, from boundedness of $\partial_xf$ and $\partial_xh$ we have
	\[
	|b_s^j|
	\leq
	C\lambda|(\nabla X_s)^j|.
	\]
	It follows that
	\[
	|F_s|
	\leq
	C\lambda
	\left(
	|\nabla X_s|+|U_s|
	\right).
	\]
	Consequently, using \eqref{nabla X} and the uniform estimate of $U_s$
	\begin{equation}
		\sup_{0\leq s\leq T}
		\mathbb E|F_s|^4
		\leq
		C\lambda^4.
		\label{eq:F-fourth}
	\end{equation}
	
	By Hölder's inequality,
	\[
	\left|
	\int_{t_i}^tF_s\,\mathrm{d}s
	\right|^2
	\leq
	(t-t_i)
	\int_{t_i}^t|F_s|^2\,\mathrm{d}s.
	\]
	Therefore, by Hölder's inequality in the expectation, \eqref{eq:D-moment} and \eqref{eq:F-fourth},
\begin{align*}
\mathbb E\left[ |D_{t_i}|^2 \left| \int_{t_i}^tF_s\,\mathrm{d}s \right|^2 \right] &\leq (t-t_i) \int_{t_i}^t \mathbb E \left[ |D_{t_i}|^2|F_s|^2 \right]\mathrm{d}s\\
	&\leq (t-t_i) \int_{t_i}^t \bigl( \mathbb E|D_{t_i}|^4 \bigr)^{1/2} \bigl( \mathbb E|F_s|^4 \bigr)^{1/2}\mathrm{d}s\\
	&\leq C\lambda^2(t-t_i)^2.
\end{align*}

Integrating over \(t\) and summing over \(i\), we obtain
	\begin{align}
		\sum_{i=0}^{N-1}
		\int_{t_i}^{t_{i+1}}
		\mathbb E\left[
		|D_{t_i}|^2
		\left|
		\int_{t_i}^tF_s\,\mathrm{d}s
		\right|^2
		\right]\mathrm{d}t &\leq
		C\lambda^2
		\sum_{i=0}^{N-1}
		(\Delta t_i)^3
		\nonumber\\
		&\leq
		C\lambda^2|\pi|^2.
		\label{eq:drift-improved}
	\end{align}
	
	For the second term, set
	\[
	M_t^i
	:=
	\int_{t_i}^t
	\nabla Z_s^{\lambda,\varepsilon}\,\mathrm{d}W_s.
	\]
	By submultiplicativity of the matrix norm,
	\[
	|M_t^iD_{t_i}|^2
	\leq
	C|M_t^i|^2|D_{t_i}|^2.
	\]
	Since \(D_{t_i}\) is
	\(\mathcal F_{t_i}\)-measurable, the conditional Itô isometry
	gives
	\[
	\begin{aligned}
		\mathbb E\left[|M_t^iD_{t_i}|^2\right]
		&\leq
		C\mathbb E\left[
		|D_{t_i}|^2
		\mathbb E_{t_i}|M_t^i|^2
		\right]\\
		&\leq
		C\mathbb E\left[
		|D_{t_i}|^2
		\int_{t_i}^t
		|\nabla Z_s^{\lambda,\varepsilon}|^2
		\,\mathrm{d}s
		\right].
	\end{aligned}
	\]
	Integrating with respect to \(t\) and summing over \(i\),

	\begin{equation*}
		\sum_{i=0}^{N-1}
		\int_{t_i}^{t_{i+1}}
		\mathbb E\left[|M_t^iD_{t_i}|^2\right]\mathrm{d}t \leq
		C|\pi|
		\mathbb E\left[
		\sup_{0\leq t\leq T}|D_t|^2
		\int_0^T
		|\nabla Z_s^{\lambda,\varepsilon}|^2
		\mathrm{d}s
		\right].
	\end{equation*}
	By Hölder's inequality,
	\[
	\begin{aligned}
		&\mathbb E\left[
		\sup_{0\leq t\leq T}|D_t|^2
		\int_0^T
		|\nabla Z_s^{\lambda,\varepsilon}|^2
		\,\mathrm{d}s
		\right]\\
		&\qquad\leq
		\left(
		\mathbb E
		\sup_{0\leq t\leq T}|D_t|^4
		\right)^{1/2}
		\left[
		\mathbb E
		\left(
		\int_0^T
		|\nabla Z_s^{\lambda,\varepsilon}|^2
		\,\mathrm{d}s
		\right)^2
		\right]^{1/2}.
	\end{aligned}
	\]
	By using \eqref{eq:D-moment} and the estimate of
	Proposition~\ref{prop:p} we have
	\[
	\mathbb E\left[
	\sup_{0\leq t\leq T}|D_t|^2
	\int_0^T
	|\nabla Z_s^{\lambda,\varepsilon}|^2
	\,\mathrm{d}s
	\right]
	\leq
	C\lambda.
	\]
	Thus
	\begin{equation}
		\sum_{i=0}^{N-1}
		\int_{t_i}^{t_{i+1}}
		\mathbb E\left[|M_t^iD_{t_i}|^2\right]\mathrm{d}t
		\leq
		C\lambda|\pi|.
		\label{eq:martingale-improved}
	\end{equation}
	
	Combining \eqref{eq:A-increment-decomposition},
	\eqref{eq:drift-improved}, and
	\eqref{eq:martingale-improved}, we obtain
	\begin{equation}
		\sum_{i=0}^{N-1}
		\mathbb E\left[
		\int_{t_i}^{t_{i+1}}
		|(U_t-U_{t_i})D_{t_i}|^2\,\mathrm{d}t\right]
		\leq
		C\left(
		\lambda|\pi|
		+
		\lambda^2|\pi|^2
		\right).
		\label{eq:A-part-improved}
	\end{equation}
	
	Combining \eqref{eq:Z-decomposition-square},
	\eqref{eq:forward-term-improved}, and
	\eqref{eq:A-part-improved}, we get
	\begin{align}
		\sum_{i=0}^{N-1}
		\mathbb E\left[
		\int_{t_i}^{t_{i+1}}
		|Z_t^{\lambda,\varepsilon}
		-Z_{t_i}^{\lambda,\varepsilon}|^2\,\mathrm{d}t\right]
		&\leq
		C\left(
		|\pi|
		+
		\lambda|\pi|
		+
		\lambda^2|\pi|^2
		\right).
		\label{eq:Z-grid-regularity-improved}
	\end{align}
	Since \(\lambda\geq1\), we obtain
	\begin{equation}
		\sum_{i=0}^{N-1}
		\mathbb E
		\int_{t_i}^{t_{i+1}}
		|Z_t^{\lambda,\varepsilon}
		-Z_{t_i}^{\lambda,\varepsilon}|^2\,\mathrm{d}t
		\leq
		C\left(
		\lambda|\pi|
		+
		\lambda^2|\pi|^2
		\right).
		\label{eq:Z-grid-final-improved}
	\end{equation}
	
	\medskip
	\noindent
	\textit{Step 4.} Define
	\[
	\bar Z_i^{\lambda,\varepsilon}
	:=
	\frac{1}{\Delta t_i}
	\mathbb E_{t_i}\left[
	\int_{t_i}^{t_{i+1}}
	Z_t^{\lambda,\varepsilon}\,\mathrm{d}t
	\right].
	\]
	For every square-integrable
	\(\mathcal F_{t_i}\)-measurable random variable \(\zeta\),
	the orthogonality property of conditional expectation gives
	\begin{equation*}
		\mathbb E
		\left[\int_{t_i}^{t_{i+1}}
		|Z_t^{\lambda,\varepsilon}-\zeta|^2
		\,\mathrm{d}t\right] =
		\mathbb E\left[
		\int_{t_i}^{t_{i+1}}
		|Z_t^{\lambda,\varepsilon}
		-\bar Z_i^{\lambda,\varepsilon}|^2
		\,\mathrm{d}t\right]
		+
		\Delta t_i
		\mathbb E\left[ |\bar Z_i^{\lambda,\varepsilon}-\zeta|^2\right].
	\end{equation*}
	Taking $ \zeta=Z_{t_i}^{\lambda,\varepsilon}$ and using \eqref{eq:Z-grid-final-improved}, we obtain
	\begin{equation}
		\mathbb E\left[
		\sum_{i=0}^{N-1}
		\int_{t_i}^{t_{i+1}}
		|Z_t^{\lambda,\varepsilon}
		-\bar Z_i^{\lambda,\varepsilon}|^2
		\,\mathrm{d}t
		\right]
		\leq
		C\left(
		\lambda|\pi|
		+
		\lambda^2|\pi|^2
		\right).
		\label{eq:regularized-final-improved}
	\end{equation}
	
	\medskip
	\noindent
	\textit{Step 5. Passage to the limit
		\(\varepsilon\to0\).} By definition, we have
	\[
	\bar Z_i^{\lambda,\varepsilon}
	-\bar Z_i^\lambda
	=
	\frac{1}{\Delta t_i}
	\mathbb E_{t_i}\left[
	\int_{t_i}^{t_{i+1}}
	\left(
	Z_t^{\lambda,\varepsilon}-Z_t^\lambda
	\right)\mathrm{d}t
	\right].
	\]
	Conditional Jensen's inequality and Hölder give
	\[
	\Delta t_i\,
	\mathbb E\left[
	|\bar Z_i^{\lambda,\varepsilon}
	-\bar Z_i^\lambda|^2\right] \leq
	\mathbb E\left[
	\int_{t_i}^{t_{i+1}}
	|Z_t^{\lambda,\varepsilon}
	-Z_t^\lambda|^2\,\mathrm{d}t\right].
	\]
	Summing over \(i\) and using
	Theorem~\ref{thm bounds}, we obtain
	\[
	\sum_{i=0}^{N-1}
	\Delta t_i\,
	\mathbb E\left[
	|\bar Z_i^{\lambda,\varepsilon}
	-\bar Z_i^\lambda|^2\right]
	\underset{\varepsilon\to0}{\longrightarrow}
	0.
	\]
	
	Define the processes
	\[
	\bar Z_t^{\lambda,\varepsilon,\pi}
	:=
	\sum_{i=0}^{N-1}
	\bar Z_i^{\lambda,\varepsilon}
	\mathbf 1_{[t_i,t_{i+1})}(t), \qquad \text{and}\qquad
	\bar Z_t^{\lambda,\pi}
	:=
	\sum_{i=0}^{N-1}
	\bar Z_i^\lambda
	\mathbf 1_{[t_i,t_{i+1})}(t).
	\]
	Then
	\[
	\bar Z^{\lambda,\varepsilon,\pi}
	\longrightarrow
	\bar Z^{\lambda,\pi}
	\qquad
	\text{in }
	L^2(\Omega\times[0,T]).
	\]
	Moreover, Theorem~\ref{thm bounds} gives
	\[
	Z^{\lambda,\varepsilon}
	\longrightarrow
	Z^\lambda
	\qquad
	\text{in }
	L^2(\Omega\times[0,T]).
	\]
	Therefore,
	\[
	Z^{\lambda,\varepsilon}
	-\bar Z^{\lambda,\varepsilon,\pi}
	\longrightarrow
	Z^\lambda-\bar Z^{\lambda,\pi}
	\qquad
	\text{in }
	L^2(\Omega\times[0,T]).
	\]
	By continuity of the \(L^2\)-norm,
	\begin{equation*}
		\mathbb E\left[
		\int_0^T
		|Z_t^\lambda-\bar Z_t^{\lambda,\pi}|^2\,\mathrm{d}t\right]
		=
		\lim_{\varepsilon\to0}
		\mathbb E\left[
		\int_0^T
		|Z_t^{\lambda,\varepsilon}
		-\bar Z_t^{\lambda,\varepsilon,\pi}|^2
		\,\mathrm{d}t\right].
	\end{equation*}
	Since the constant in
	\eqref{eq:regularized-final-improved}
	is independent of \(\varepsilon\), and
	\[
	\mathbb E\left[
		\sum_{i=0}^{N-1}
		\int_{t_i}^{t_{i+1}}
		|Z_t^{\lambda,\varepsilon}
		-\bar Z_i^{\lambda,\varepsilon}|^2
		\,\mathrm{d}t\right]
		=
		\mathbb E\left[
		\int_0^T
		|Z_t^{\lambda,\varepsilon}
		-\bar Z_t^{\lambda,\varepsilon,\pi}|^2
		\,\mathrm{d}t\right],
	\]

	we conclude that
	\[
	\mathbb E\left[
	\int_0^T
	|Z_t^\lambda-\bar Z_t^{\lambda,\pi}|^2
	\,\mathrm{d}t\right]
	\leq
	C\left(
	\lambda|\pi|
	+
	\lambda^2|\pi|^2
	\right).
	\]
	Equivalently,
	\[
	\varepsilon^{Z^\lambda}(\pi)
	\leq
	C\left(
	\lambda|\pi|
	+
	\lambda^2|\pi|^2
	\right),
	\]
	which proves \eqref{eq:Z-regularity}.
	
	Finally, if $
	\lambda|\pi|\leq C_0,
	$
	then 
	\[
	\varepsilon^{Z^\lambda}(\pi)
	\leq
	C(1+C_0)\lambda|\pi| \leq
	C\lambda|\pi|,
	\]
	after absorbing \(1+C_0\) into \(C\). This proves \eqref{eq:Z-regularity-improved}.
\end{proof}

\section{Convergence analysis}
\label{sec:convergence}
The objective of this section is to establish some explicit error estimates for the \(L^2\)-error between the solution \(Y\) of the RBSDE \eqref{RBSDE}, and the learned solution produced from the DBDP-PBSDE scheme, together with explicit error estimates of the \(L^2\)-error for the reconstructed control process.

Recall that, under our standing assumption that the driver is independent of \(z\), the backward learning procedure approximates only the value process \(Y^\lambda\). More precisely, the DBDP-PBSDE scheme produces the family of learned functions
\[
\left(
\mathcal U_i^\lambda
\right)_{0\leq i\leq N-1},
\]
and the corresponding approximation of the value process $Y$ is
\[
\mathcal U_i^\lambda(X_{i}^\pi),
\qquad i=0,\ldots,N-1.
\]

The convergence analysis is therefore divided into three steps. First,
for a fixed penalization parameter \(\lambda>0\), we estimate the
discretization and neural-network approximation errors for the
\(Y\)-component of the penalized BSDE \eqref{PBSDE}. Second, we derive an error estimate for the reconstructed control process $\mathcal{Z}^\lambda$ which is- considered as approximation of $Z^\lambda$. Finally, we combine these estimates with the convergence rate of the penalization method in order to obtain global error estimates with respect to the solution \((Y,Z)\) of the reflected BSDE.

\subsection{Error estimates for the penalized BSDE}
We first analyze the approximation of the value component of the penalized BSDE. For a fixed \(\lambda>0\), define
\begin{equation*}
\mathcal E_Y^\lambda(\pi) := \max_{0\leq i\leq N-1} \mathbb E\left[ \left| Y_{t_i}^\lambda - \mathcal U_i^\lambda(X_{i}^\pi) \right|^2 \right].
\end{equation*}

We begin with a preliminary result concerning the implicit map associated with the penalized generator.
	
\begin{lemma}
\label{lemma:inv}
For \(i\in\{0,\ldots,N-1\}\), define
\[\mathcal A_i(x,y):=y-f^\lambda(t_i,x,y)\Delta t_i,\qquad(x,y)\in\mathbb R^d\times\mathbb R.\]

Assume \ref{HL}(\ref{HL_3},\ref{HL_5}) and $[f]\Delta t_i<1. $ Then, for every \(x\in\mathbb R^d\), the mapping $ y\longmapsto\mathcal A_i(x,y) $ is a bijection from \(\mathbb R\) onto \(\mathbb R\). Moreover, for every $\hat{y}_1,\hat{y}_2\in\mathbb{R}$,
\begin{equation}
\left|\mathcal A_i^{-1}(x,\hat y_1)-\mathcal A_i^{-1}(x,\hat y_2)\right|\leq\frac{1}{1-[f]\Delta t_i}|\hat y_1-\hat y_2|.
\label{eq:inverse-Lipschitz}
\end{equation}
In particular, if $[f]\Delta t_i\leq\frac12,$ then
\begin{equation}
\left|\mathcal A_i^{-1}(x,\hat y_1)-\mathcal A_i^{-1}(x,\hat y_2)\right|\leq(1+2[f]\Delta t_i)|\hat y_1-\hat y_2|.
\label{eq:inverse-Lipschitz-simple}
\end{equation}
\end{lemma}

\begin{proof}
Fix \(x\in\mathbb R^d\) and \(y_1,y_2\in\mathbb R\). Since \( y\longmapsto (y-h(t_i,x))^- \) is non-increasing,
\[ \left[ (y_1-h(t_i,x))^- - (y_2-h(t_i,x))^- \right](y_1-y_2) \leq0. \]

Combining this inequality with the Lipschitz continuity of \(f\) with respect to \(y\) gives
\begin{equation*}
\left[ f^\lambda(t_i,x,y_1) - f^\lambda(t_i,x,y_2) \right](y_1-y_2) \leq [f]|y_1-y_2|^2.
\end{equation*}
It follows that
\begin{align}
\left[ \mathcal A_i(x,y_1) - \mathcal A_i(x,y_2) \right](y_1-y_2) & = |y_1-y_2|^2 - \Delta t_i \left[ f^\lambda(t_i,x,y_1) - f^\lambda(t_i,x,y_2) \right](y_1-y_2) \nonumber\\
	& \geq (1-[f]\Delta t_i)|y_1-y_2|^2.
\label{eq:A-strong-monotonicity}
\end{align}
Thus, since \([f]\Delta t_i<1\), \(\mathcal A_i(x,\cdot)\) is strictly increasing.

Taking \(y_2=0\) in \eqref{eq:A-strong-monotonicity} yields, for every $y \in\mathbb{R}$
\[ \left[ \mathcal A_i(x,y)-\mathcal A_i(x,0) \right]y \geq (1-[f]\Delta t_i)|y|^2. \]
Hence $ \mathcal A_i(x,y)\longrightarrow+\infty$ as $y\longrightarrow+\infty$, and $\mathcal A_i(x,y)\longrightarrow-\infty$ as $y\longrightarrow-\infty.$ Since \(\mathcal A_i(x,\cdot)\) is continuous, it is therefore a bijection from \(\mathbb R\) onto \(\mathbb R\).

Furthermore, \eqref{eq:A-strong-monotonicity} implies
\begin{equation}
|\mathcal A_i(x,y_1)-\mathcal A_i(x,y_2)| \geq (1-[f]\Delta t_i)|y_1-y_2|. \label{inq:lower-estimate-A}
\end{equation}
Setting \( y_j=\mathcal A_i^{-1}(x,\hat y_j)\), for $\hat y_j\in\mathbb{R}$ and \(j=1,2\), gives \eqref{eq:inverse-Lipschitz}. Finally, if \([f]\Delta t_i\leq1/2\), then
\[ \frac{1}{1-[f]\Delta t_i} \leq 1+2[f]\Delta t_i, \]
which proves \eqref{eq:inverse-Lipschitz-simple}.
\end{proof}

We now introduce the auxiliary implicit process \((V_i^\lambda)_{0\leq i\leq N-1}\), which will serve as an intermediate target in the error analysis. This process is introduced only for theoretical purposes and is not an additional output of the DBDP-PBSDE algorithm.
	
\begin{lemma}
\label{lemma-v}
Assume \ref{HL}(\ref{HL_1},\ref{HL_6}) and $[f]\Delta t_i<1.$ Then, for every \(\lambda>0\) and \(i\in\{0,\ldots,N-1\}\), there exists a unique $\mathcal{F}_{t_i}$-measurable random variable $V_i^\lambda \in L^2(\Omega)$ satisfying
\begin{equation}
V_i^\lambda = \mathbb E_{t_i}\left[ \mathcal U_{i+1}^\lambda (X_{i+1}^\pi) \right] + f^\lambda \left( t_i,X_{i}^\pi,V_i^\lambda \right)\Delta t_i. \label{fixed point}
\end{equation}

Moreover, there exists a deterministic measurable function $ v_i^\lambda:\mathbb R^d\longrightarrow\mathbb R $ such that

\begin{equation}
V_i^\lambda = v_i^\lambda(X_{i}^\pi).
\label{v function}
\end{equation}
\end{lemma}

\begin{proof}
Fix \(i\in\{0,\ldots,N-1\}\) and set
\[ \xi_i^\lambda := \mathbb E_{t_i}\left[ \mathcal U_{i+1}^\lambda (X_{i+1}^\pi) \right]. \]

By construction, \(\mathcal U_{i+1}^\lambda\) is either the terminal function \(g\) or an admissible neural network with Lipschitz activation functions. Hence, \(\mathcal U_{i+1}^\lambda\) is Lipschitz continuous and therefore satisfies
\[
|\mathcal U_{i+1}^\lambda(x)|
\leq
|\mathcal U_{i+1}^\lambda(0)|
+
L_{i+1}^\lambda|x|,
\]
for some constant \(L_{i+1}^\lambda<\infty\).
Since \(X_{i+1}^\pi\in L^2(\Omega)\), it follows that
\[
\mathcal U_{i+1}^\lambda(X_{i+1}^\pi)\in L^2(\Omega),
\]
and therefore $
\xi_i^\lambda
\in
L^2(\Omega,\mathcal F_{t_i}).
$

Equation \eqref{fixed point} is equivalent to
\[ \mathcal A_i(X_{i}^\pi,V_i^\lambda) = \xi_i^\lambda. \]
By Lemma~\ref{lemma:inv}, it admits the unique solution
\begin{equation}
V_i^\lambda = \mathcal A_i^{-1} \left( X_{i}^\pi,\xi_i^\lambda \right).
\label{eq:Vi-inverse}
\end{equation}

It remains to verify that \(V_i^\lambda\in L^2(\Omega)\). Applying \eqref{eq:inverse-Lipschitz} with \( \hat y_1=\xi_i^\lambda,\, \hat y_2=\mathcal A_i(X_{i}^\pi,0)\), and using
\( \mathcal A_i^{-1} \left( X_{i}^\pi, \mathcal A_i(X_{i}^\pi,0) \right) =0, \) we obtain
\begin{equation*}
|V_i^\lambda| \leq \frac{ |\xi_i^\lambda| + \Delta t_i |f^\lambda(t_i,X_{i}^\pi,0)| }{ 1-[f]\Delta t_i }.
\end{equation*}
		Under {\ref{HL}(\ref{HL_3},\ref{HL_4})},
		\[
		|f^\lambda(t,x,0)|
		\leq
		C(1+\lambda)(1+|x|).
		\]
		Therefore, since \(X_{i}^\pi\in L^2(\Omega)\) and
		\(\xi_i^\lambda\in L^2(\Omega)\), the above estimate of $V_i^\lambda$ implies
		\[
		V_i^\lambda\in L^2(\Omega).
		\]
		
		We now prove the Markovian representation. The Euler scheme satisfies
		\[
		X_{i+1}^\pi
		=
		X_{i}^\pi
		+
		\mu(t_i,X_{i}^\pi)\Delta t_i
		+
		\sigma(t_i,X_{i}^\pi)\Delta W_i,
		\]
		where the increment $ \Delta W_i=W_{t_{i+1}}-W_{t_i}$
		is independent of \(\mathcal F_{t_i}\). Define
		\[
		\psi_i^\lambda(x)
		:=
		\mathbb E\left[
		\mathcal U_{i+1}^\lambda
		\left(
		x+\mu(t_i,x)\Delta t_i+\sigma(t_i,x)\Delta W_i
		\right)
		\right].
		\]
		Then
		\[
		\xi_i^\lambda
		=
		\psi_i^\lambda(X_{i}^\pi).
		\]
		
		For \(x\in\mathbb R^d\), define
		\begin{equation*}
			v_i^\lambda(x)
			:=
			\mathcal A_i^{-1}
			\left(
			x,\psi_i^\lambda(x)
			\right).
		\end{equation*}
		Since both \(\mathcal A_i^{-1}\) and \(\psi_i^\lambda\) are measurable,
		\(v_i^\lambda\) is measurable. Combining
		\eqref{eq:Vi-inverse} with the representation of \(\xi_i^\lambda\),
		we obtain
		\[
		V_i^\lambda
		=
		v_i^\lambda(X_i^\pi),
		\]
		which proves \eqref{v function}.
	\end{proof}
	
	The functions \(v_i^\lambda\) provide the natural reference functions
	for measuring the approximation capacity of the neural-network classes.
	For \(i\in\{0,\ldots,N-1\}\), we therefore define
	\begin{equation}
		\varepsilon_i^{\mathcal N,v^\lambda}
		:=
		\inf_{\theta}
		\mathbb E\left[
		\left|
		v_i^\lambda(X_{i}^\pi)
		-
		\mathcal U(X_{i}^\pi;\theta)
		\right|^2
		\right],
		\label{eq:NN-v-error}
	\end{equation}
	where the infimum is taken over the admissible parameters of the
	neural-network class under consideration.
	
	We next relate this approximation error to the loss function used in
	the DBDP-PBSDE scheme. Set
	\[
	M_{i+1}^\lambda
	:=
	\mathcal U_{i+1}^\lambda(X_{i+1}^\pi) - \mathbb{E}_{t_i}\left[\mathcal U_{i+1}^\lambda(X_{i+1}^\pi) \right].
	\]
	Then
	\[
	\mathbb E_{t_i}[M_{i+1}^\lambda]=0,
	\]
	while \eqref{fixed point} gives
	\begin{equation}
		\mathbb{E}_{t_i}\left[\mathcal U_{i+1}^\lambda(X_{i+1}^\pi) \right]
		=
		\mathcal A_i(X_{i}^\pi,V_i^\lambda).
		\label{eq:mi-A}
	\end{equation}

	The loss function \eqref{lossY} at time \(t_i\) can be written
	\begin{align*}
		L_i^\lambda(\theta)
		&=
		\mathbb E\left[
		\left|
		\mathcal U_{i+1}^\lambda(X_{i+1}^\pi)
		-
		\mathcal U(X_{i}^\pi;\theta)
		+
		f^\lambda(t_i,X_{i}^\pi,\mathcal U(X_{i}^\pi;\theta))\Delta t_i
		\right|^2
		\right]
		\nonumber\\
		&=
		\mathbb E\left[
		\left|
		\mathcal U_{i+1}^\lambda(X_{i+1}^\pi)
		-
		\mathcal A_i(X_{i}^\pi,,\mathcal U(X_{i}^\pi;\theta))
		\right|^2
		\right].
	\end{align*}
	By writing 
	\[
	\mathcal U_{i+1}^\lambda(X_{i+1}^\pi) =\mathbb{E}_{t_i}\left[\mathcal U_{i+1}^\lambda(X_{i+1}^\pi) \right] + M_{i+1}^\lambda,
	\]
	and using
	\[
	\mathcal A_i(X_{i}^\pi,\mathcal U(X_{i}^\pi;\theta))
	\in\mathcal F_{t_i},
	\]
	together with expression \eqref{eq:mi-A} it follows
\begin{align}
L_i^\lambda(\theta) = \mathbb E\left[\left|M_{i+1}^\lambda\right|^2\right] + \mathbb E\left[ \left| \mathcal A_i(X_{i}^\pi,V_i^\lambda) - \mathcal A_i(X_{i}^\pi,\mathcal U(X_{i}^\pi;\theta)) \right|^2 \right].
\label{eq:loss-decomposition}
\end{align}
	
$\;$\\
The first term  in \eqref{eq:loss-decomposition} is the conditional-variance contribution associated with
 \(\mathcal U_{i+1}^\lambda(X_{i+1}^\pi)\) and is independent of the network parameter
 \(\theta\). Consequently, the minimization of
 \(L_i^\lambda\) is equivalent to minimizing
 \[
 \mathbb E\left[
 \left|
 \mathcal A_i(X_{i}^\pi,V_i^\lambda)
 -
 \mathcal A_i(X_{i}^\pi,\mathcal U(X_{i}^\pi;\theta))
 \right|^2
 \right].
 \]
 Thus, the loss compares the network output
 \(\mathcal{U}(X_{i}^\pi;\theta)\) with the auxiliary target
 \(V_i^\lambda=v_i^\lambda(X_{i}^\pi)\) through the transformation
 \(\mathcal A_i\). The coercivity and Lipschitz estimates for
 \(\mathcal A_i\) established above then allow this transformed error to
 be related directly to
 \(\mathbb E\left[\left|V_i^\lambda-\mathcal{U}(X_{i}^\pi;\theta)\right|^2\right]\).

\begin{lemma}
\label{lemma:local-consistency}
Assume \ref{HL}(\ref{HL_1}--\ref{HL_6}). For \(\lambda\geq1\), let \((Y^\lambda,Z^\lambda)\) be the solution of the penalized BSDE \eqref{PBSDE}. For \(i=0,\ldots,N-1\), define
				\[
				\mathcal E_i^{f,\lambda}
				:=
				\int_{t_i}^{t_{i+1}}
				\left[
				f^\lambda(s,X_s,Y_s^\lambda)
				-
				f^\lambda(t_i,X_{t_i},Y_{t_i}^\lambda)
				\right]\mathrm{d}s.
				\]
				Then there exists a constant \(C>0\), independent of
				\(\lambda\geq1\), \(N\), and \(|\pi|\), such that
				\begin{equation}
					\sum_{i=0}^{N-1}
					\frac{1}{|\pi|}
					\mathbb E
					\left[
					\left|
					\mathbb E_{t_i}
					\left[
					\mathcal E_i^{f,\lambda}
					\right]
					\right|^2
					\right]
					\leq
					C\left(
					|\pi|+\lambda^2|\pi|
					\right).
					\label{eq:local-consistency}
				\end{equation}
			\end{lemma}
		
\begin{proof}
Throughout the proof, $C>0$ denotes a generic constant independent of $\lambda$, $N$, and $|\pi|$.

Fix $\lambda \geq  1$. We write 
$$ \mathcal{E}_i^{f,\lambda} = \mathcal{E}_i^f + \lambda\mathcal{E}_i^\text{pen} $$
where
$$ \mathcal{E}_i^f = \int_{t_i}^{t_{i+1}} \left( f(s,X_s,Y_s^\lambda)- f(t_i,X_{t_i},Y_{t_i}^\lambda)\mathrm{d}s \right)$$
and
$$ \mathcal{E}_i^{\text{pen}} = \int_{t_i}^{t_{i+1}}((Y_s^\lambda - h(s,X_s))^- - (Y_{t_i}^\lambda - h(t_i,X_{t_i}))^-)\mathrm{d}s,$$
			
hence
\begin{equation}
\sum_{i=0}^{N-1}\frac{1}{|\pi|}\mathbb{E}\left[\left|\mathbb{E}_{t_i}\left[\mathcal{E}_i^{f,\lambda}\right]\right|^2\right] \leq  2I_f+2\lambda^2 I_{\text{pen}},
\label{global_bound}
\end{equation}

with
$$I_f = \sum_{i=0}^{N-1}\frac{1}{|\pi|}\mathbb{E}\left[\left|\mathbb{E}_{t_i}\left[\mathcal{E}_i^f\right]\right|^2\right] \;\; \text{and} \;\; I_{\text{pen}} = \sum_{i=0}^{N-1}\frac{1}{|\pi|}\mathbb{E}\left[\left|\mathbb{E}_{t_i}\left[\mathcal{E}_i^{\text{pen}}\right]\right|^2\right] . $$

\textit{Step 1. Estimation of $I_f$.} By conditional Jensen's inequality and Hölder's inequality we have
\begin{align*}
I_f &\leq  \sum_{i=0}^{N-1}\frac{1}{|\pi|}\mathbb{E}\left[\left|\mathcal{E}_i^f\right|^2\right]\\
&\leq  \sum_{i=1}^{N-1}\int_{t_i}^{t_{i+1}}\mathbb{E}\left[|f(s,X_s,Y_s) -f(t_i,X_{t_i},Y_{t_i}^\lambda)|^2\right]\mathrm{d}s
\end{align*}

Lipschitz continuity of $f$, combined with the estimate \eqref{eq:standard_X_estimate} gives
\begin{equation*}
I_f\leq  C\left( |\pi| + \sum_{i=1}^{N-1}\int_{t_i}^{t_{i+1}} \mathbb{E}\left[|Y_s - Y^\lambda_{t_i}|^2\right] \mathrm{d}s\right).
\end{equation*}

Now let \(s\in[t_i,t_{i+1}]\), from the expression of PBSDE \eqref{PBSDE}, we have
\begin{equation*}
Y_s^\lambda-Y_{t_i}^\lambda = -\int_{t_i}^s f(r,X_r,Y_r^\lambda)\mathrm{d}s -
(K_s^\lambda-K_{t_i}^\lambda) + \int_{t_i}^s Z_r^\lambda\,\mathrm{d}W_r,
\end{equation*}

it follows that
\begin{equation}
\mathbb{E}\left[ |Y_s^\lambda-Y_{t_i}^\lambda|^2 \right] \leq  C\left( \mathbb{E}\left[\left|\int_{t_i}^sf(r,X_r,Y_r^\lambda)\mathrm{d}r\right|^2\right]  + \mathbb{E}\left[|K^\lambda_s - K^\lambda_{t_i} |^2\right] + \mathbb{E}\left[\int_{t_i}^s|Z^\lambda_r|^2\mathrm{d}r\right] \right) \label{eq:Y-increment-bound}
\end{equation}

For the first term in the right side of \eqref{eq:Y-increment-bound}, by Hölder's inequality, we have
\begin{align*}
\left|\int_{t_i}^s f(r,X_r,Y_r^\lambda)\mathrm{d}r\right|^2 &\leq  (s-t_i)\int_{t_i}^s|f(r,X_r,Y^\lambda_r)|^2\mathrm{d}r\\
	&\leq  |\pi|\int_{t_i}^{t_{i+1}}|f(r,X_r,Y^\lambda_r)|^2\mathrm{d}r,
\end{align*}

thus
\begin{align*}
\sum_{i=0}^{N-1}\int_{t_i}^{t_{i+1}}\mathbb{E}\left[\left|\int_{t_i}^s f(r,X_r,Y_r^\lambda)\mathrm{d}r\right|^2\right]\mathrm{d}s &\leq  |\pi|^2 \sum_{i=0}^{N-1}\int_{t_i}^{t_{i+1}}\mathbb{E}\left[|f(r,X_r,Y^\lambda_r)|^2\right]\mathrm{d}r\\
	&= |\pi|^2\mathbb{E}\left[\int_0^T |f(r,X_r,Y^\lambda_r)|^2\mathrm{d}r \right]
\end{align*}

By \ref{HL}(\ref{HL_3},\ref{HL_4}), we have $ |f(r,X_r,Y_r^\lambda)|^2 \leq C\left( 1+|X_r|^2+|Y_r^\lambda|^2 \right)$, combining this inequality  with standard estimate \eqref{Xmoment} and \eqref{eq:uniform-penalization} gives
$$ \sup_{\lambda\geq1} \mathbb{E}\left[\int_0^T |f(r,X_r,Y^\lambda_r)|^2\mathrm{d}r\right] \leq  C$$

Consequently,
\begin{equation*}
\sum_{i=0}^{N-1}\int_{t_i}^{t_{i+1}}\mathbb{E}\left[\left|\int_{t_i}^s f(r,X_r,Y_r^\lambda)\mathrm{d}r\right|^2\right]\mathrm{d}s  \leq  C|\pi|^2.
\end{equation*}

For the second term in the right side of \eqref{eq:Y-increment-bound}, since $K^\lambda$ is increasing, then for $s\in[t_i,t_{i+1}]$,
$$0\leq K_s^\lambda-K_{t_i}^\lambda\leq K_{t_{i+1}}^\lambda-K_{t_i}^\lambda,$$
therefore
\begin{align*}
\sum_{i=0}^{N-1}\int_{t_i}^{t_{i+1}}\mathbb{E}\left[|K_s^\lambda-K_{t_i}^\lambda|^2\right]\mathrm{d}s &\leq  |\pi|\mathbb{E}\left[\sum_{i=0}^{N-1} |K_{t_{i+1}}^\lambda-K_{t_i}^\lambda|^2\right]\\
	&\leq  |\pi|\mathbb{E}\left[\left(\sum_{i=0}^{N-1} (K_{t_{i+1}}^\lambda-K_{t_i}^\lambda)\right)^2\right]\\
	&=|\pi|\mathbb{E}\left[|K^\lambda_T|^2\right].
\end{align*}

Hence, by \eqref{eq:uniform-penalization}
\begin{equation*}
\sum_{i=0}^{N-1}\int_{t_i}^{t_{i+1}}\mathbb{E}\left[|K_s^\lambda-K_{t_i}^\lambda|^2\right]\mathrm{d}s \leq  C|\pi|.
\end{equation*}

Finally, for the last term in the right side of \eqref{eq:Y-increment-bound}, by using \eqref{eq:uniform-penalization} we have
\begin{align*}
\sum_{i=0}^{N-1}\int_{t_i}^{t_{i+1}}\mathbb{E}\left[\int_{t_i}^s|Z_r^\lambda|^2\mathrm{d}r\right]\mathrm{d}s &\leq  \sum_{i=0}^{N-1}\int_{t_i}^{t_{i+1}} \mathbb{E}\left[\int_{t_i}^{t_{i+1}}|Z_r^\lambda|^2\mathrm{d}r\right]\mathrm{d}s \nonumber \\
	&\leq  |\pi|\sum_{i=0}^{N-1}\mathbb{E}\left[\int_{t_i}^{t_{i+1}}|Z_r^\lambda|^2\mathrm{d}r\right] \nonumber \\
	&= |\pi|\mathbb{E}\left[\int_0^T|Z^\lambda_r|^2\mathrm{d}r\right]\leq C|\pi|.
\end{align*}

Coming back to \eqref{eq:Y-increment-bound}, we have
\begin{equation}
\sum_{i=0}^{N-1} \int_{t_i}^{t_{i+1}}\mathbb{E}\left[|Y_s^\lambda-Y_{t_i}^\lambda|^2\right]\mathrm{d}s \leq  C|\pi|.
\label{sum_int_Y}
\end{equation}

Therefore
\begin{equation}
I_f \leq  C|\pi|.
\label{estim:I_f}
\end{equation}

\textit{Step 2. Estimation of $I_\text{pen}$.}
By conditional Jensen's inequality and Hölder's inequality we have
\begin{align*}
I_\text{pen} &\leq  \sum_{i=0}^{N-1}\frac{1}{|\pi|}\mathbb{E}\left[\left|\mathcal{E}_i^\text{pen}\right|^2\right]\\
&\leq  \sum_{i=1}^{N-1}\int_{t_i}^{t_{i+1}}\mathbb{E}\left[|(Y_s^\lambda - h(s,X_s))^- - (Y_{t_i}^\lambda - h(t_i,X_{t_i}))^-|^2\right]\mathrm{d}s.
\end{align*}

For $s\in[t_i,t_{i+1}]$, since $x\mapsto x^-$ is 1--Lipschitz, and $h$ satisfies \ref{HL}(\ref{HL_5}), we have 
\begin{align*}
|(Y_s^\lambda - h(s,X_s))^- - (Y_{t_i}^\lambda - h(t_i,X_{t_i}))^-|^2 &\leq  2|Y_s^\lambda - Y_{t_i}^\lambda|^2 + 2|h(s,X_s) - h(t_i,X_{t_i})|^2\\
	&\leq  C\left(|\pi| + |X_s - X_{t_i}|^2 + |Y_s^\lambda - Y_{t_i}^\lambda|^2\right),
\end{align*}

by using \eqref{eq:standard_X_estimate} and \eqref{sum_int_Y}, it follows that
\begin{align}
I_\text{pen} &\leq  C\sum_{i=0}^{N-1}\int_{t_i}^{t_{i+1}}\left(|\pi| + \mathbb{E}\left[|X_s - X_{t_i}|^2\right] + \mathbb{E}\left[|Y_s^\lambda - Y_{t_i}^\lambda|^2\right]\right)\mathrm{d}s \nonumber \\
	&\leq  C|\pi|. \label{estim:I_p}
\end{align}

By combining \eqref{global_bound} with \eqref{estim:I_f} and \eqref{estim:I_p}, we obtain the desired result.
\end{proof}

We prove now the first main estimate of this section.
\begin{theorem}
\label{thrm:0}
Assume that \ref{HL}(\ref{HL_1}--\ref{HL_6}) hold. Let \(\lambda\geq1\) and suppose that $[f]|\pi|\leq\frac12,\, \lambda|\pi|\leq C_0$, for some constant \(C_0>0\) independent of \(\lambda\) and \(|\pi|\). Assume moreover that at each time step the minimisation problem \eqref{minimisation_problem} admits a minimizer \(\theta_{i,\lambda}^*\in\Theta_m\). Then there exists a constant \(C>0\), independent of \(\lambda\), \(N\), and \(|\pi|\), such that
\begin{equation}
\max_{0\leq i\leq N-1} \left\| Y_{t_i}^\lambda - \mathcal U_i^\lambda(X_{i}^\pi) \right\|_2 \leq C\left[ (1+\lambda)|\pi|^{1/2} + \sum_{i=0}^{N-1} \sqrt{\varepsilon_i^{\mathcal N,v^\lambda}} \right].
\label{eq:main-Y-error-L2}
\end{equation}
\end{theorem}

\begin{proof}
Throughout the proof, \(C>0\) denotes a generic constant which may change from line to line, but is independent of \(\lambda\geq1\), \(N\), and \(|\pi|\).\\

Let $i\in\{0,\dots,N-1\}$, from Lemma \ref{lemma-v} we have

\begin{equation*}
V_i^\lambda = \mathbb{E}_{t_i}\left[\mathcal{U}^\lambda_{i+1}\left(X^\pi_{i+1}\right)\right]+f^\lambda(t_i,X_{i}^\pi,V_i^\lambda)\Delta t_i .
\end{equation*}

On the other hand, from the equation (\ref{PBSDE}) and by definition of $\mathcal{E}_i^{f,\lambda}$ we have 
\begin{equation*}
Y_{t_i}^\lambda = \mathbb{E}_{t_i}\left[Y^\lambda_{t_{i+1}}+f^\lambda(t_i,X_{t_i},Y_{t_i}^\lambda)\Delta t_i + \mathcal{E}_i^{f,\lambda}\right] .
\end{equation*}

By definition of the application $\mathcal{A}_i$ in Lemma \ref{lemma:inv} we have 
\begin{align*}
\mathcal{A}_i(X_{i}^\pi ; Y_{t_i}^\lambda) &= Y_{t_i}^\lambda - f^\lambda(t_i,X_{i}^\pi,Y_{t_i}^\lambda)\Delta t_i \nonumber \\
	&=\mathbb{E}_{t_i}\left[Y_{t_{i+1}}^\lambda+\mathcal{E}_i^{f,\lambda}\right] + \left( f^\lambda(t_i,X_{t_i},Y_{t_i}^\lambda) - f^\lambda(t_i,X_{i}^\pi,Y_{t_i}^\lambda)\right)\Delta t_i \nonumber \\
	&=\mathbb{E}_{t_i}\left[Y_{t_{i+1}}^\lambda+\mathcal{E}_i^{f,\lambda}\right] + R_i^{X,\lambda} ,
\end{align*}
and
\begin{align*}
\mathcal{A}_i(X_{i}^\pi ; V_i^\lambda) &= V_i^\lambda - f^\lambda(t_i,X_{i}^\pi,V_i^\lambda)\Delta t_i \nonumber  \\
	&=\mathbb{E}_{t_i}\left[\mathcal{U}_{i+1}^\lambda(X^\pi_{i+1})\right], 
\end{align*}

where
$$ R_i^{X,\lambda} := \left( f^\lambda(t_i,X_{t_i},Y_{t_i}^\lambda) - f^\lambda(t_i,X_{i}^\pi,Y_{t_i}^\lambda)\right)\Delta t_i .$$

\textit{Step 1 : Upper bound  of $\|Y_{t_i}^\lambda-V_i^\lambda\|_2$.}
Since the application $\mathcal{A}_i$ is invertible we have
\begin{align*}
\|Y_{t_i}^\lambda-V_i^\lambda\|_2 &= \left\Vert \mathcal{A}_i^{-1}\left(X_{i}^\pi,\mathbb{E}_{t_i}\left[Y_{t_{i+1}}^\lambda+\mathcal{E}_i^{f,\lambda}\right] + R_i^{X,\lambda}\right)  - \mathcal{A}_i^{-1}\left(X_{i}^\pi,\mathbb{E}_{t_i}\left[\mathcal{U}_{i+1}^\lambda(X^\pi_{i+1})\right]\right) \right\Vert_2,
\end{align*}

moreover, $\mathcal{A}_i^{-1}$ is $(1+2[f]|\pi|)$-Lipschitz, by setting
$$ e_i := Y_{t_i}^\lambda - \mathcal U_i^\lambda(X_{i}^\pi), $$
we obtain
\begin{equation*}
\|Y_{t_i}^\lambda-V_i^\lambda\|_2 \leq (1+C\Delta t_i) \Big( \|\mathbb E_{t_i}[e_{i+1}]\|_2 + \|\mathbb E_{t_i} [\mathcal E_i^{f,\lambda}]\|_2  + \|R_i^{X,\lambda}\|_2 \Big).
\end{equation*}

By using Lipschitz continuity of $f$ with respect to $y$, the Lipschitz continuity of \(r\mapsto r^-\), and the estimate \eqref{platen} we obtain
\begin{equation*}
\|R_i^{X,\lambda}\|_2 \leq  C\lambda|\pi|^{3/2},
\end{equation*}
where we used \(\lambda\geq  1\). It follows that
\begin{align}
\|Y_{t_i}^\lambda-V_i^\lambda\|_2 \leq (1+C\Delta t_i) \Big( & \|\mathbb E_{t_i}[e_{i+1}]\|_2 + \|\mathbb E_{t_i} [\mathcal E_i^{f,\lambda}]\|_2 + \lambda|\pi|^{3/2}\Big).
\label{eq:YV-L2-conditional}
\end{align}

Moreover, we have
\begin{align*}
\mathbb{E}\left[\left|\mathbb{E}_{t_i}\left[e_{i+1}\right]\right|^2\right] &\leq \mathbb{E}\left[\mathbb{E}_{t_i}\left[\left|e_{i+1}\right|^2\right]\right]\\
	&\leq \mathbb{E}\left[\left|e_{i+1}\right|^2\right],
\end{align*}

thus
\begin{align}
\|Y_{t_i}^\lambda-V_i^\lambda\|_2 \leq (1+C\Delta t_i) \Big( & \|e_{i+1}\|_2 + \|\mathbb E_{t_i} [\mathcal E_i^{f,\lambda}]\|_2 + \lambda|\pi|^{3/2}\Big).
\label{eq:YV-L2-final}
\end{align}

\textit{Step 2. Upper bound of \(\|V_i^\lambda - \mathcal{U}_i^\lambda(X_{i}^\pi) \|_2\).} By the loss decomposition \eqref{eq:loss-decomposition}, we have
$$  L_i^\lambda(\theta) = \mathbb{E}\left[|M_{i+1}^\lambda|^2\right] + \widetilde{L}_i^\lambda(\theta),$$

where
$$ \widetilde{L}_i^\lambda(\theta) := \mathbb E \left[ \left| \mathcal A_i(X_{i}^\pi,V_i^\lambda) - \mathcal A_i \left( X_{i}^\pi, \mathcal U(X_{i}^\pi;\theta) \right) \right|^2 \right], \qquad  M_{i+1}^\lambda = \mathcal U_{i+1}^\lambda(X_{i+1}^\pi) - \mathbb{E}_{t_i}\left[\mathcal U_{i+1}^\lambda(X_{i+1}^\pi) \right]. $$
	
Since the first term does not depend on \(\theta\), and \(\theta_{i,\lambda}^*\) is a minimizer of $L_i^\lambda$, then it also minimizes \(\widetilde{L}_i^\lambda\).

The lower estimate \eqref{inq:lower-estimate-A} for \(\mathcal A_i\) gives
$$ \widetilde{L}_i^\lambda(\theta) \geq (1-[f]\Delta t_i)^2 \mathbb{E}\left[ \left| V_i^\lambda - \mathcal U(X_{i}^\pi;\theta) \right|^2\right].$$

On the other hand, the Lipschitz continuity of $f$ and $r\mapsto r^-$, together with \(\lambda\Delta t\leq C_0\) imply
\begin{align*}
\widetilde{L}_i^\lambda(\theta) &\leq \left[ 1+([f]+\lambda)\Delta t_i \right]^2 \mathbb{E}\left[ \left| V_i^\lambda - \mathcal U(X_{i}^\pi;\theta) \right|^2\right]\\
	&\leq  C \mathbb{E}\left[ \left| V_i^\lambda - \mathcal U(X_{i}^\pi;\theta) \right|^2\right].
\end{align*}

By combining the above inequalities of $\widetilde{L}_i^\lambda$, and using $[f]\Delta t_i\leq \frac{1}{2}$ we obtain
\begin{align*}
\mathbb{E}\left[ \left| V_i^\lambda - \mathcal{U}_i^\lambda(X_{i}^\pi) \right|^2\right] &\leq  C \widetilde{L}_i^\lambda(\theta_{i,\lambda}^*)\\
	&\leq C\widetilde{L}_i^\lambda(\theta)\\
	&\leq C\mathbb{E}\left[ \left| V_i^\lambda - \mathcal U(X_{i}^\pi;\theta) \right|^2\right] .
\end{align*}

Since it holds for every parameter $\theta$, by taking the infimum we obtain
\begin{equation}
\mathbb E\left[|V_i^\lambda - \mathcal{U}_i^\lambda(X_{i}^\pi) |^2\right] \leq  C\varepsilon_i^{\mathcal N,v^\lambda}. \label{eq:V_U-NN}
\end{equation}

Consequently,
\begin{equation}
\|V_i^\lambda - \mathcal{U}_i^\lambda(X_{i}^\pi) \|_2 \leq  C \sqrt{\varepsilon_i^{\mathcal N,v^\lambda}}.
\label{eq:rho-NN-L2}
\end{equation}

\textit{Step 3. Global \(L^2\)-error estimate.} By adding and subtracting $V_i^\lambda$ in $e_i$, and using \eqref{eq:YV-L2-final} and \eqref{eq:rho-NN-L2}, we obtain
\begin{align*}
\|e_i\|_2 \leq{}& (1+C\Delta t_i)\|e_{i+1}\|_2 + 	C\|\mathbb E_{t_i} [\mathcal E_i^{f,\lambda}]\|_2  + C \lambda|\pi|^{3/2} + C\sqrt{\varepsilon_i^{\mathcal N,v^\lambda}}. 
\end{align*}

By using discrete Grönwall's inequality, we obtain
\begin{equation*}
\max_{0\leq i\leq N-1}\|e_i\|_2 \leq C\Bigg(  \|e_N\|_2 + \sum_{i=0}^{N-1} \|\mathbb E_{t_i} [\mathcal E_i^{f,\lambda}]\|_2  + N\lambda|\pi|^{3/2} + \sum_{i=0}^{N-1} \sqrt{\varepsilon_i^{\mathcal N,v^\lambda}} \Bigg).
\end{equation*}

By definition of $e_N$, and the initial network $\mathcal{U}_N^\lambda$, and with using the Lipschitz continuity of $g$, and the standard estimate \eqref{platen} we have
\begin{equation}
\|e_N\|_2 \leq  C|\pi|^{1/2}.
\label{eq:terminal-L2}
\end{equation}

On the other hand, \eqref{eq:local-consistency} of Lemma~\ref{lemma:local-consistency} gives

\begin{equation*}
\sum_{i=0}^{N-1} \left\Vert\mathbb E_{t_i}\left[ \mathcal E_i^{f,\lambda} \right] \right\Vert^2_2  \leq C\left( |\pi|^2+\lambda^2|\pi|^2 \right).
\end{equation*}

Using Hölder's inequality and \(N=\frac{T}{|\pi|}\),
\begin{align*}
\sum_{i=0}^{N-1} \|\mathbb E_{t_i} [\mathcal E_i^{f,\lambda}]\|_2 &\leq \sqrt{N} \left( \sum_{i=0}^{N-1} \left\Vert\mathbb E_{t_i}\left[ \mathcal E_i^{f,\lambda} \right] \right\Vert^2_2 \right)^{1/2} \nonumber\\
	&\leq C|\pi|^{1/2}+ C\lambda|\pi|^{1/2}.
\end{align*}

By using these estimates we conclude
$$
\max_{0\leq i\leq N-1} \left\| Y_{t_i}^\lambda - \mathcal U_i^\lambda(X_{i}^\pi) \right\|_2 \leq C\left[ |\pi|^{1/2} +\lambda|\pi|^{1/2} + \sum_{i=0}^{N-1} \sqrt{\varepsilon_i^{\mathcal N,v^\lambda}} \right],
$$

which completes the proof.
\end{proof}

\begin{remark}
	\label{rem:Gobet-Wang-comparison}
	Under the regularity assumptions of Gobet and Wang \cite{gobet}, including the semimartingale structure of the obstacle, their Tanaka--local-time argument improves the penalization contribution of $\mathcal{E}_i^{\lambda,\varepsilon}$ from $O(\lambda^2|\pi|)$ to $O(\lambda^2|\pi|^{3/2})$. At the \(L^2\)-norm level, this corresponds to an accumulated error of order $O(\lambda|\pi|^{3/4})$ instead of \(O(\lambda|\pi|^{1/2})\).
	
	In the present DBDP-PBSDE framework, however, the forward diffusion is approximated by the Euler scheme. The corresponding spatial error satisfies
	$
	\|R_i^{X,\lambda}\|_2
	\leq
	C\lambda|\pi|^{3/2},
$
	and hence, after summation over the time steps,
	$
	\sum_{i=0}^{N-1}\|R_i^{X,\lambda}\|_2
	\leq
	C\lambda|\pi|^{1/2}.
	$
	This term dominates the sharper contribution
	\(\lambda|\pi|^{3/4}\) of $\mathcal{E}_i^{\lambda,\varepsilon}$. Consequently, in the general Euler setting, the global estimate remains
	\[
	\max_i
	\left\|
	Y_{t_i}^\lambda
	-
	\mathcal U_i^\lambda(X_i^\pi)
	\right\|_2
	\leq
	C\left[
	(1+\lambda)|\pi|^{1/2}
	+
	\sum_i
	\sqrt{\varepsilon_i^{\mathcal N,v^\lambda}}
	\right].
	\]
	
Thus, the loss with respect to the sharper Gobet--Wang rate is mainly due to the additional approximation of the forward diffusion. 

\end{remark}

\begin{remark}
\label{rem:Y-squared-error}
The estimate of Theorem~\ref{thrm:0} can equivalently be expressed in terms of the squared \(L^2\)-error. Indeed, using the inequality $(a+b)^2\leq  2a^2+2b^2,$  we obtain
\begin{equation}
	\max_{0\leq i\leq N-1} \mathbb E \left[ \left| Y_{t_i}^\lambda - \mathcal U_i^\lambda(X_{i}^\pi) \right|^2 \right]  \leq C\left[ |\pi| + \lambda^2|\pi| + \left( \sum_{i=0}^{N-1} \sqrt{\varepsilon_i^{\mathcal N,v^\lambda}} \right)^2 \right]. \label{eq:main-Y-error-squared}
\end{equation}

Furthermore, by the Hölder's inequality,
$$ \left( \sum_{i=0}^{N-1} \sqrt{\varepsilon_i^{\mathcal N,v^\lambda}} \right)^2 \leq N\sum_{i=0}^{N-1} \varepsilon_i^{\mathcal N,v^\lambda}.
$$

Hence \eqref{eq:main-Y-error-squared} implies the estimate
\begin{equation}
 \max_{0\leq i\leq N-1} \mathbb E \left[ \left| Y_{t_i}^\lambda - \mathcal U_i^\lambda(X_{i}^\pi) \right|^2 \right] \leq  C\left( |\pi| + \lambda^2|\pi| + N\sum_{i=0}^{N-1} \varepsilon_i^{\mathcal N,v^\lambda} \right). \label{eq:coarser-estimate}
\end{equation}

Thus the latter estimate is a direct consequence of the sharper \(L^2\)-norm estimate \eqref{eq:main-Y-error-L2}. In particular, the factor \(N\) does not arise from the error analysis itself; it appears only when the accumulated approximation error $ (\sum_i\sqrt{\varepsilon_i^{\mathcal N,v^\lambda}})^2 $ is further bounded by Hölder's inequality.
\end{remark}

We give now an estimate of the error for the reconstructed control process defined by \eqref{Zlearned}.
\begin{theorem}
\label{Z-error}
Assume \ref{HL}(\ref{HL_1}--\ref{HL_6}), and the assumptions of Theorem~\ref{thm:Z-regularity}. Let \(\lambda\geq1\) and suppose that $[f]|\pi|\leq\frac12$ and $\lambda|\pi|\leq C_0$ for some constant \(C_0>0\) independent of \(\lambda\) and $|\pi|$.
 Assume, as in Theorem~\ref{thrm:0}, that at each time step the minimization problem \eqref{minimisation_problem} admits a minimizer \(\theta_{i,\lambda}^*\in\Theta_m\). Then there exists a constant \(C>0\), independent of \(\lambda\), \(N\), and \(|\pi|\), such that
\begin{equation}
\mathbb E \left[ \sum_{i=0}^{N-1} \int_{t_i}^{t_{i+1}} |Z_s^\lambda-\mathcal Z_i^\lambda|^2\mathrm{d}s \right] \leq C\left( |\pi| + \lambda^2|\pi| + N\sum_{i=0}^{N-1} \varepsilon_i^{\mathcal N,v^\lambda} \right).
\label{eq:main-Z-error}
\end{equation}

\end{theorem}

\begin{proof}
Throughout the proof, \(C>0\) denotes a generic constant which may change from line to line, but is independent of \(\lambda\), \(N\), and \(|\pi|\).

Let \(i=0,\ldots,N-1\), we have
\begin{align}
\mathbb E \left[\int_{t_i}^{t_{i+1}} |Z_s^\lambda-\mathcal Z_i^\lambda|^2\mathrm{d}s\right]  \leq 2\mathbb E \left[\int_{t_i}^{t_{i+1}} |Z_s^\lambda-\bar Z_{i}^\lambda|^2\mathrm{d}s\right] + 2\Delta t_i\, \mathbb E \left[|\bar Z_{i}^\lambda-\mathcal Z_i^\lambda|^2\right],
\label{eq:Z-decomposition}
\end{align}
where \(\bar Z_{i}^\lambda\) is defined in \eqref{eq:Zbar}.

\textit{Step 1. Estimate of \(\bar Z_{i}^\lambda-\mathcal Z_i^\lambda\).} By It\^o isometry
\begin{equation*}
\Delta t_i\,\bar Z_{i}^\lambda = \mathbb E_{t_i} \left[ \left( \int_{t_i}^{t_{i+1}}Z_s^\lambda\,dW_s \right) \Delta W_{i} \right].
\end{equation*}

On the other hand, by \eqref{Zlearned}, we have
\begin{equation*}
\Delta t_i\,\mathcal Z_i^\lambda = \mathbb E_{t_i} \left[ \mathcal U_{i+1}^\lambda(X_{i+1}^\pi) \Delta W_{i} \right].
\end{equation*}
	
From the PBSDE \eqref{PBSDE}, we have
\begin{equation*}
\int_{t_i}^{t_{i+1}}Z_s^\lambda\,\mathrm{d}W_s=Y_{t_{i+1}}^\lambda-Y_{t_i}^\lambda+ \int_{t_i}^{t_{i+1}} f^\lambda(s,X_s,Y_s^\lambda)\mathrm{d}s.
\end{equation*}

As in proof of Theorem \ref{thrm:0}, we denote \( e_i = Y_{t_i}^\lambda - \mathcal U_i^\lambda(X_{i}^\pi).\) Since \(Y_{t_i}^\lambda\) is \(\mathcal F_{t_i}\)-measurable and \(\mathbb E_{t_i}[\Delta W_{t_i}]=0\), it follows that

\begin{equation*}
\Delta t_i (\bar Z_{i}^\lambda-\mathcal Z_i^\lambda) =  \mathbb E_{t_i} [e_{i+1}\Delta W_{i}]  + \mathbb E_{t_i} \left[ \left( \int_{t_i}^{t_{i+1}} f^\lambda(s,X_s,Y_s^\lambda)\mathrm{d}s \right) \Delta W_{i} \right].
\end{equation*}
	
Because
\[ \mathbb E_{t_i}[e_{i+1}\Delta W_{i}] = \mathbb E_{t_i} \left[ \left( e_{i+1}-\mathbb E_{t_i}[e_{i+1}] \right) \Delta W_{i} \right], \]

Hölder's inequality, applied component-wise, yields
\begin{equation*}
\mathbb E\left[\left| \mathbb E_{t_i} [e_{i+1}\Delta W_{i}] \right|^2\right] \leq d\Delta t_i \left( \mathbb E\left[|e_{i+1}|^2\right] - \mathbb E\left[ |\mathbb E_{t_i}[e_{i+1}]|^2\right] \right).
\end{equation*}

Similarly,
\begin{equation*}
\mathbb E\left[\left| \mathbb E_{t_i} \left[ \left( \int_{t_i}^{t_{i+1}} f^\lambda(s,X_s,Y_s^\lambda)\mathrm{d}s \right) \Delta W_{i} \right] \right|^2\right] \leq d(\Delta t_i)^2 \mathbb E\left[ \int_{t_i}^{t_{i+1}} |f^\lambda(s,X_s,Y_s^\lambda)|^2\mathrm{d}s\right].
\end{equation*}

Consequently,
\begin{align*}
\Delta t_i\, \mathbb E\left[|\bar Z_{i}^\lambda-\mathcal Z_i^\lambda|^2\right] \leq{}& C \left( \mathbb E\left[|e_{i+1}|^2\right] - \mathbb E\left[ |\mathbb E_{t_i}[e_{i+1}]|^2\right] \right) + C\Delta t_i\, \mathbb E\left[ \int_{t_i}^{t_{i+1}} |f^\lambda(s,X_s,Y_s^\lambda)|^2\mathrm{d}s\right].
\end{align*}
	
Summing \eqref{eq:Z-decomposition} over \(i\) gives
\begin{align}
\mathbb E \left[ \sum_{i=0}^{N-1} \int_{t_i}^{t_{i+1}} |Z_s^\lambda-\mathcal Z_i^\lambda|^2\mathrm{d}s \right] &\leq  \sum_{i=0}^{N-1} \left( \mathbb E\left[|e_{i+1}|^2\right] - \mathbb E \left[|\mathbb E_{t_i}[e_{i+1}]|^2\right] \right) \nonumber\\
	&\qquad+ C\varepsilon^{Z^\lambda}(\pi) + C\Delta t_i\, \mathbb E\left[ \int_0^T |f^\lambda(s,X_s,Y_s^\lambda)|^2\mathrm{d}s  \right].
\label{eq:Z-global-first}
\end{align}

\textit{Step 2.} We now treat the first term in right side of the above inequality. Let $i\in\{0,\ldots,N-1\}$. By using the Young's inequality
$$(a+b)^2\leq  (1+|\pi|)a^2 + \left(1+\frac{1}{|\pi|}\right)b^2 ,$$

with inequality \eqref{eq:YV-L2-conditional} we have
\begin{align*}
\mathbb{E}\left[\left|Y_{t_i}^\lambda-V_i^\lambda\right|^2\right]& \leq  (1+C|\pi|)\mathbb{E}\left[\left|\mathbb E_{t_i}[e_{i+1}]\right|^2\right] + \frac{C}{|\pi|}\mathbb{E}\left[\left|\mathbb{E}_{t_i}\left[\mathcal{E}_i^{f,\lambda}\right]\right|^2\right]  +\frac{C}{|\pi|}\lambda^2|\pi|^3\\
	&\leq  (1+C|\pi|)\mathbb{E}\left[\left|\mathbb E_{t_i}[e_{i+1}]\right|^2\right] + \frac{C}{|\pi|}\mathbb{E}\left[\left|\mathbb{E}_{t_i}\left[\mathcal{E}_i^{f,\lambda}\right]\right|^2\right] + C\lambda^2|\pi|^2.
\end{align*}

On the other hand, by subtracting and adding $V_i^\lambda$ in the expression of $e_i$ and using Young's inequality we obtain (up to a change of the generic constant \(C\))
\begin{align*}
\mathbb{E}\left[|e_i|^2\right] &\leq  (1+C|\pi|) \mathbb{E}\left[\left|Y_{t_i}^\lambda-V_i^\lambda\right|^2\right] + \left(1+\frac{1}{|\pi|}\right)\mathbb{E}\left[\left|V_i^\lambda - \mathcal{U}_i^\lambda(X_{i}^\pi)\right|^2\right]\\
	&\leq  (1+C|\pi|)\mathbb{E}\left[\left|\mathbb E_{t_i}[e_{i+1}]\right|^2\right] + \frac{C}{|\pi|}\mathbb{E}\left[\left|\mathbb{E}_{t_i}\left[\mathcal{E}_i^{f,\lambda}\right]\right|^2\right] + C\lambda^2|\pi|^2+\frac{C}{|\pi|}\mathbb{E}\left[\left|V_i^\lambda - \mathcal{U}_i^\lambda(X_{i}^\pi)\right|^2\right],
\end{align*}

it follows that
\begin{align}
\mathbb{E}\left[|e_i|^2\right] - \mathbb{E}\left[\left|\mathbb E_{t_i}[e_{i+1}]\right|^2\right]  &\leq  C|\pi|\mathbb{E}\left[\left|\mathbb E_{t_i}[e_{i+1}]\right|^2\right] + \frac{C}{|\pi|}\mathbb{E}\left[\left|\mathbb{E}_{t_i}\left[\mathcal{E}_i^{f,\lambda}\right]\right|^2\right] + C\lambda^2|\pi|^2 \nonumber \\
	&\qquad +\frac{C}{|\pi|}\mathbb{E}\left[\left|V_i^\lambda - \mathcal{U}_i^\lambda(X_{i}^\pi)\right|^2\right]. \label{inq:sum-e_i}
\end{align}

Now by changing indices in the sum of the right hand side of inequality \eqref{eq:Z-global-first}, we obtain
\begin{equation*}
\sum_{i=0}^{N-1} \left( \mathbb E\left[|e_{i+1}|^2\right] - \mathbb E \left[|\mathbb E_{t_i}[e_{i+1}]|^2\right] \right) \leq  \mathbb{E}\left[|e_N|^2\right] + \sum_{i=0}^{N-1} \left( \mathbb E\left[|e_{i}|^2\right] - \mathbb E \left[|\mathbb E_{t_i}[e_{i+1}]|^2\right] \right),
\end{equation*}

by using \eqref{eq:terminal-L2} and \eqref{inq:sum-e_i}, we obtain

\begin{align*}
\sum_{i=0}^{N-1} \left( \mathbb E\left[|e_{i+1}|^2\right] - \mathbb E \left[|\mathbb E_{t_i}[e_{i+1}]|^2\right] \right) &\leq  C \Big(\max_{0\leq  i\leq  N-1}\mathbb{E}\left[\left|\mathbb E_{t_i}[e_{i+1}]\right|^2\right]  + \frac{1}{|\pi|}\sum_{i=0}^{N-1}\mathbb{E}\left[\left|\mathbb{E}_{t_i}\left[\mathcal{E}_i^{f,\lambda}\right]\right|^2\right]\\
	&\qquad + |\pi|+\lambda^2|\pi| + N\sum_{i=0}^{N-1} \mathbb{E}\left[\left|V_i^\lambda - \mathcal{U}_i^\lambda(X_{i}^\pi)\right|^2\right] \Big).
\end{align*}

By using Jensen's inequality for the first term, \eqref{eq:local-consistency} of Lemma \ref{lemma:local-consistency} for the second term, and \eqref{eq:V_U-NN} for the last term, we obtain
\begin{equation*}
\sum_{i=0}^{N-1} \left( \mathbb E\left[|e_{i+1}|^2\right] - \mathbb E \left[|\mathbb E_{t_i}[e_{i+1}]|^2\right] \right) \leq  C \Big(\max_{0\leq  i\leq  N-1}\mathbb{E}\left[\left|e_{i+1}\right|^2\right] + |\pi|+\lambda^2|\pi| + N\sum_{i=0}^{N-1}\varepsilon_i^{\mathcal{N},v^\lambda} \Big),
\end{equation*}

which implies by using \eqref{eq:coarser-estimate}

\begin{align*}
\sum_{i=0}^{N-1} \left( \mathbb E\left[|e_{i+1}|^2\right] - \mathbb E \left[|\mathbb E_{t_i}[e_{i+1}]|^2\right] \right) &\leq  C \Big(|\pi|+\lambda^2|\pi| + N\sum_{i=0}^{N-1}\varepsilon_i^{\mathcal{N},v^\lambda} \Big).
\end{align*}

\textit{Step 3. }By plugging the above inequality into \eqref{eq:Z-global-first} we obtain
\begin{equation*}
\mathbb E \left[ \sum_{i=0}^{N-1} \int_{t_i}^{t_{i+1}} |Z_s^\lambda-\mathcal Z_i^\lambda|^2\mathrm{d}s \right] \leq  C\left( |\pi|+\lambda^2|\pi| + N\sum_{i=0}^{N-1}\varepsilon_i^{\mathcal{N},v^\lambda} + \varepsilon^{Z^\lambda}(\pi)  + |\pi|\, \mathbb E\left[ \int_0^T |f^\lambda(s,X_s,Y_s^\lambda)|^2\mathrm{d}s  \right] \right).
\end{equation*}

For the last term, by definition of $f^\lambda$ and the Lipschitz property of $f$ and $r\mapsto r^-$, together with the uniform bounds \eqref{eq:uniform-penalization} and \eqref{Xmoment}, we have
$$ \mathbb E\left[ \int_0^T |f^\lambda(s,X_s,Y_s^\lambda)|^2\mathrm{d}s  \right] \leq  C\lambda^2. $$

Finally, by using Theorem \ref{thm:Z-regularity}, and with $\lambda\geq 1$ we obtain
\begin{equation*}
\mathbb E \left[ \sum_{i=0}^{N-1} \int_{t_i}^{t_{i+1}} |Z_s^\lambda-\mathcal Z_i^\lambda|^2\mathrm{d}s \right] \leq  C\left( |\pi| +\lambda^2|\pi| + N\sum_{i=0}^{N-1}\varepsilon_i^{\mathcal{N},v^\lambda} \right),
\end{equation*}

which completes the proof.
\end{proof}

\subsection{Global error estimates for the reflected BSDE}

We now combine the error estimate for \(Y\) approximation of
the penalized BSDE \eqref{PBSDE} with the penalization error in order to derive a global error estimate for the value component of the original
reflected BSDE. As an additional result, we also derive a global error
estimate for the control process \eqref{Zlearned}, reconstructed from the learned value functions.

We consider the following additional assumptions on the barrier:
\begin{enumerate}
\mylabel{H\textsubscript{B}}\label{HB} The barrier $h(\cdot,X_\cdot)$ admits the following general semi martingale representation
	$$ h(t,X_t) = h(0,X_0) + \int_0^tU^h_s\mathrm{d}s + \int_0^tV^h_s\mathrm{d}W_s + A^h_t, \qquad 0\leq t\leq T,$$
	where $U^h$ and $V^h$ are two continuous adapted stochastic processes, and $A^h$ is a continuous adapted increasing process such that 
	$$\mathbb{E}\left[\sup_{0\leq  t\leq  T}|U^h_t|^2\right] + \mathbb{E}\left[\sup_{0\leq  t\leq  T}|V^h_t|^2\right] + \mathbb{E}\left[|A^h_T|^2\right]<+\infty .$$
\end{enumerate}

Then from Theorem 3.2 in \cite{gobet}, under assumptions \ref{HL} and \ref{HB}, we have the following estimate
\begin{equation}
	\label{eq:penalization-errors}
	\mathbb E\left[
	\sup_{0\leq t\leq T}|Y_t-Y_t^\lambda|^2
	\right]
	+
	\mathbb E\left[
	\int_0^T|Z_s-Z_s^\lambda|^2\mathrm{d}s
	\right]
	\leq
	\frac{C}{\lambda}.
\end{equation}

Combining \eqref{eq:penalization-errors} with \eqref{eq:coarser-estimate} we obtain
\begin{equation}
	\label{eq:Y-total-before-lambda}
	\max_{0\leq i\leq N-1}
	\mathbb E\left[
	\left|
	Y_{t_i}
	-
	\mathcal U_i^\lambda(X_{i}^\pi)
	\right|^2
	\right]
	\leq
	C\left(
	\frac{1}{\lambda}
	+
	|\pi|
	+
	\lambda^2|\pi|
	+
	N\sum_{i=0}^{N-1}
	\varepsilon_i^{\mathcal N,v^\lambda}
	\right).
\end{equation}

Similarly, \eqref{eq:penalization-errors} and \eqref{eq:main-Z-error} imply
\begin{equation}
	\label{eq:Z-total-before-lambda}
	\mathbb E\left[
	\sum_{i=0}^{N-1}
	\int_{t_i}^{t_{i+1}}
	\left|
	Z_s-\mathcal Z_i^\lambda
	\right|^2\mathrm{d}s
	\right]
	\leq
	C\left(
	\frac{1}{\lambda}
	+
	|\pi|
	+
	\lambda^2|\pi|
	+
	N\sum_{i=0}^{N-1}
	\varepsilon_i^{\mathcal N,v^\lambda}
	\right).
\end{equation}

Leaving aside for the moment the neural-network approximation term,
the two competing \(\lambda\)-dependent contributions in
\eqref{eq:Y-total-before-lambda} and \eqref{eq:Z-total-before-lambda} are $ \frac{1}{\lambda}$ and $\lambda^2|\pi|. $
Balancing these terms gives the optimal choice
\begin{equation*}
	\lambda
	=
	c_0|\pi|^{-1/3},
\end{equation*}
where \(c_0>0\) is fixed independently of \(|\pi|\). Notice that this
choice balances only the penalization and time-discretization
contributions. In general, the neural-network approximation error
\(\varepsilon_i^{\mathcal N,v^\lambda}\) also depends on \(\lambda\).

\begin{corollary}
	\label{coro:convergence-NN}
	Assume that the hypotheses of Theorem~\ref{thrm:0}, together with
	the assumptions ensuring the penalization estimate
	\eqref{eq:penalization-errors}, are satisfied. Let $
	\lambda=c_0|\pi|^{-1/3}$, where \(c_0>0\) is fixed independently of \(|\pi|\). Then, for \(|\pi|\) sufficiently small, the DBDP-PBSDE approximation satisfies
	\begin{equation}
		\label{global-error-Y}
		\max_{0\leq i\leq N-1}
		\mathbb E\left[
		\left|
		Y_{t_i}
		-
		\mathcal U_i^\lambda(X_{i}^\pi)
		\right|^2
		\right]
		\leq
		C\left(
		|\pi|^{1/3}
		+
		N\sum_{i=0}^{N-1}
		\varepsilon_i^{\mathcal N,v^\lambda}
		\right).
	\end{equation}
	
	If, in addition, the assumptions of Theorem~\ref{Z-error} are
	satisfied, the reconstructed control process \eqref{Zlearned} satisfies
	\begin{equation}
		\label{global-error-Z}
		\mathbb E\left[
		\sum_{i=0}^{N-1}
		\int_{t_i}^{t_{i+1}}
		\left|
		Z_s-\mathcal Z_i^\lambda
		\right|^2\mathrm{d}s
		\right]
		\leq
		C\left(
		|\pi|^{1/3}
		+
		N\sum_{i=0}^{N-1}
		\varepsilon_i^{\mathcal N,v^\lambda}
		\right),
	\end{equation}
	for some constant \(C>0\) independent of \(N\), \(|\pi|\), and \(\lambda\).
\end{corollary}

\begin{proof}
By substituting $ \lambda=c_0|\pi|^{-1/3}$ into \eqref{eq:Y-total-before-lambda}, we obtain
$$ \frac{1}{\lambda} + |\pi| + \lambda^2|\pi| \leq  c_0^{-1}|\pi|^\frac{1}{3} + |\pi|^\frac{1}{3} + c_0^2 |\pi|^\frac{1}{3} $$

Consequently,
	\[
	\max_{0\leq i\leq N-1}
	\mathbb E\left[
	\left|
	Y_{t_i}
	-
	\mathcal U_i^\lambda(X_{i}^\pi)
	\right|^2
	\right]
	\leq
	C\left(
	|\pi|^{1/3}
	+
	N\sum_{i=0}^{N-1}
	\varepsilon_i^{\mathcal N,v^\lambda}
	\right),
	\]
	which proves \eqref{global-error-Y}.
	
	If the assumptions of Theorem~\ref{Z-error} are also satisfied,
	substituting the same choice of \(\lambda\) into
	\eqref{eq:Z-total-before-lambda} gives
	\[
	\mathbb E\left[
	\sum_{i=0}^{N-1}
	\int_{t_i}^{t_{i+1}}
	\left|
	Z_s-\mathcal Z_i^\lambda
	\right|^2\mathrm{d}s
	\right]
	\leq
	C\left(
	|\pi|^{1/3}
	+
	N\sum_{i=0}^{N-1}
	\varepsilon_i^{\mathcal N,v^\lambda}
	\right),
	\]
	which proves \eqref{global-error-Z}.\\
	
It remains to verify the step-size conditions. We have $\lambda|\pi| = c_0|\pi|^{2/3} \longrightarrow0$ as $|\pi|\to0$. Thus, $ \lambda|\pi|\leq C_0 $ for \(|\pi|\) sufficiently small. Likewise, $ [f]|\pi|\leq\frac12 $ for \(|\pi|\) sufficiently small. Hence all the step-size conditions required in Theorems~\ref{thrm:0} and~\ref{Z-error} are satisfied.
\end{proof}

\subsection{Rate of convergence}
Let's introduce the flow of the Euler scheme. For $x\in\mathbb{R}^d$ we define
$$ X_{j+1}^{\pi,k,x} = X_j^{\pi,k,x} +\mu(t_j,X_j^{\pi,k,x})\Delta t_j +  \sigma(t_j,X_j^{\pi,k,x})\Delta W_j \;\; , \;\; j\in\{k,\cdots , N\} .$$

and $X_k^{\pi,k,x} = x$ at $j=k\in\mathbb{N}^*$.\\

From Lemma 4.3.1 in \cite{germain}, under \ref{HL}(\ref{HL_1}) we have for every $x,x'\in\mathbb{R}^d$
\begin{equation}
\mathbb{E}\left[\left|X_{k+1}^{\pi,k,x} - X_{k+1}^{\pi,k,x'}\right|^2\right] \leq  (1+C\Delta t_i)|x-x'|^2.
\label{inq:estimate-flow}
\end{equation}

From Lemma~\ref{lemma-v}, for every \(i=0,\ldots,N-1\), $V_i^\lambda = v_i^\lambda(X_i^\pi)$,
where
\[
V_i^\lambda
=
\mathbb E_{t_i}
\left[
\mathcal U_{i+1}^\lambda(X_{i+1}^\pi)
\right]
+
f^\lambda(t_i,X_i^\pi,V_i^\lambda)\Delta t_i.
\]
In terms of the Euler flow, this yields
\begin{equation}
v_i^\lambda(x) = \mathbb E \left[ \mathcal U_{i+1}^\lambda
(X_{i+1}^{\pi,i,x}) \right] + f^\lambda (t_i,x,v_i^\lambda(x)) \Delta t_i.
\label{eq:def-vi-flow}
\end{equation}

\begin{lemma}
	\label{lemma:lip-v}
	Assume \ref{HL} and suppose that $[f]|\pi|\leq\frac{1}{2}.$
	Then there exists a family of positive constants
	\((K_i)_{0\leq i\leq N}\), independent of \(\lambda\), and a
	constant \(K>0\), independent of \(i\), \(N\), and \(\lambda\), satisfying
	\[
	\sup_{0\leq i\leq N}K_i\leq K,
	\]
	such that, the admissible neural-network class at time \(t_i\) can be
	chosen as $ \mathcal G_{K_i,d,\ell,m}^{\rho^\kappa}$,
	and
	\begin{equation*}
		[v_i^\lambda]\leq K_i\leq K,
		\qquad
		[\mathcal U_i^\lambda]\leq K_i\leq K,
		\qquad
		i=0,\ldots,N-1.
	\end{equation*}

\end{lemma}

\begin{proof}
We proceed backward in time. At the terminal time, $\mathcal U_N^\lambda=g$. Set $K_N:=\max\{[g],[h]\}$, then
\[ [\mathcal U_N^\lambda]=[g]\leq K_N. \]

Suppose now that, for some \(i\in\{0,\ldots,N-1\}\),
	\[
	\mathcal U_{i+1}^\lambda
	\in
	\mathcal G_{K_{i+1},d,\ell,m}^{\rho^\kappa},
	\]
	so that
	\begin{equation*}
		[\mathcal U_{i+1}^\lambda]\leq K_{i+1}.
	\end{equation*}
	
Let \(x,x'\in\mathbb R^d\), and set
\[\delta v_i := v_i^\lambda(x)-v_i^\lambda(x'),
\qquad \delta\psi_i := \mathbb E \left[ \mathcal U_{i+1}^\lambda(X_{i+1}^{\pi,i,x}) - \mathcal U_{i+1}^\lambda(X_{i+1}^{\pi,i,x'}) \right], \qquad \delta h_i := h(t_i,x)-h(t_i,x'), \]

\[ \beta_i := \frac{ f(t_i,x,v_i^\lambda(x)) - f(t_i,x,v_i^\lambda(x')) }{ \delta v_i } \mathbf 1_{\{\delta v_i\neq0\}}, \qquad  \rho_i := f(t_i,x,v_i^\lambda(x')) - f(t_i,x',v_i^\lambda(x')). \]

For the penalization term, set
	\begin{align*}
		\alpha_i^\lambda
		:={}&
		\lambda
		\frac{
			(v_i^\lambda(x)-h(t_i,x))^-
			-
			(v_i^\lambda(x')-h(t_i,x'))^-
		}{
			\delta v_i-\delta h_i
		} 
		\mathbf 1_{\{\delta v_i-\delta h_i\neq0\}}.
	\end{align*}
	Since \(r\mapsto r^-\) is non-increasing and \(1\)-Lipschitz,
we have $
	-\lambda
	\leq
	\alpha_i^\lambda
	\leq
	0.
$
	Subtracting the equations
	\eqref{eq:def-vi-flow} corresponding to \(x\) and \(x'\), we obtain
	\[
	\delta v_i
	=
	\delta\psi_i
	+
	\Delta t_i
	\left[
	\beta_i\delta v_i
	+
	\rho_i
	+
	\alpha_i^\lambda
	(\delta v_i-\delta h_i)
	\right].
	\]
	Therefore,
	\begin{equation}
		\left(
		1-\beta_i\Delta t_i
		-\alpha_i^\lambda\Delta t_i
		\right)
		\delta v_i
		=
		\delta\psi_i
		+
		\Delta t_i\rho_i
		-
		\alpha_i^\lambda\Delta t_i\delta h_i.
		\label{eq:discrete-monotonicity}
	\end{equation}
	
	Set
	\[
	q_i:=-\alpha_i^\lambda\Delta t_i\geq0.
	\]
	By \ref{HL} we have $|\beta_i|\leq[f]$, it follows that
	\[
	1-\beta_i\Delta t_i+q_i
	\geq
	1-[f]\Delta t_i+q_i,
	\]
	we deduce from \eqref{eq:discrete-monotonicity} that
	\begin{align*}
		|\delta v_i|
		\leq
		\frac{
			|\delta\psi_i|
			+
			\Delta t_i|\rho_i|
			+
			q_i|\delta h_i|
		}{
			1-[f]\Delta t_i+q_i
		}.
	\end{align*}

For $|\delta\psi_i|$, by using Jensen's inequality followed by Hölder's inequality together with \eqref{inq:estimate-flow}, we obtain after changing the constant $C$
$$ |\delta\psi_i| \leq K_{i+1}(1+C\Delta t_i)|x-x'| . $$
Moreover, always with \ref{HL}, we have 
\[ |\rho_i| \leq C|x-x'|, \qquad |\delta h_i| \leq [h]|x-x'|. \]

It follows that
\begin{equation*}
|\delta v_i| \leq \frac{K_{i+1}(1+C\Delta t_i) + C\Delta t_i + q_i[h]}{1-[f]\Delta t_i+q_i}|x-x'|
\end{equation*}
For \(A,a>0\), \(B\geq0\), and \(q\geq0\), by using the elementary inequality 
\[ \frac{A+qB}{a+q} \leq \max\left\{ \frac Aa,B \right\},\]

we obtain
\begin{equation*}
|\delta v_i| \leq K_i|x-x'|,
\end{equation*}
with 
$$ K_i := \max\left\{\frac{ K_{i+1}(1+C\Delta t_i) + C\Delta t_i}{1-[f]\Delta t_i}	,[h]\right\} .$$

On one hand, the above inequality gives
\begin{equation*}
[v_i]\leq K_i.
\end{equation*}
On the other hand, since $ [f]\Delta t_i \leq \frac{1}{2}$,  we have $ \frac1{1-[f]\Delta t_i} \leq 1+2[f]\Delta t_i$. After changing \(C\), this gives
\[ K_{i} \leq (1+C\Delta t_i)K_{i+1} + C\Delta t_i. \]

Since \(K_N\geq[h]\), backward iteration gives
\[ K_i \leq K_N \prod_{j=i}^{N-1} (1+C\Delta t_j) + C \sum_{k=i}^{N-1} \Delta t_k \prod_{j=i}^{k-1} (1+C\Delta t_j). \]

Using $ 1+C\Delta t_j \leq e^{C\Delta t_j}$,  and $ \sum_{j=0}^{N-1}\Delta t_j=T$,  we obtain
\[ \sup_{0\leq i\leq N}K_i \leq (K_N+CT)e^{CT} =:K. \]

At time \(t_i\), choose the admissible network class $\mathcal G_{K_i,d,\ell,m}^{\rho^\kappa}$. Every network in this class is \(K_i\)-Lipschitz. Hence the minimizer selected by the algorithm satisfies
\[ [\mathcal U_i^\lambda] \leq K_i \leq K. \]

This closes the backward induction. Consequently,
\[ [v_i^\lambda]\leq K, \qquad [\mathcal U_i^\lambda]\leq K, \qquad i=0,\ldots,N-1, \]
where \(K\) is independent of \(i\), \(N\), and \(\lambda\).

\end{proof}

\begin{theorem}
\label{thm:errors-estimation-groupsort}
Assume \ref{HL}, let \(\lambda\geq1\)and assume that $ [f]|\pi|\leq\frac12$ and $\lambda|\pi|\leq C_0.$ Assume moreover that at each time step the minimisation problem \eqref{minimisation_problem} admits a minimizer \(\theta_{i,\lambda}^*\in\Theta_m\). Then there exist $\kappa,\ell\in\mathbb{N}^*$ and $m =(m_i)_{0\leq i\leq \ell-1} \in\mathbb{N}^\ell$, such that by choosing at each time level $i$ of DBDP-PBSDE scheme, the GroupSort class $\mathcal G_{K_i,d,\ell,m}^{\rho^\kappa}$, where \((K_i)_{0\leq  i\leq  N-1}\) are the constants of Lemma~\ref{lemma:lip-v}, we have
\begin{equation}
\label{eq:NN-total-error}
N\sum_{i=0}^{N-1} \varepsilon_i^{\mathcal N,v^\lambda} \leq C|\pi|.
\end{equation}
Consequently, and with $\lambda = c_0|\pi|^{-\frac{1}{3}}$, the DBDP-PBSDE approximation satisfies
\begin{equation*}
\max_{0\leq i\leq N-1} \mathbb E \left[ \left| Y_{t_i} - \mathcal U_i^\lambda(X_{i}^\pi) \right|^2 \right] \leq C|\pi|^\frac{1}{3}.
\end{equation*}
	
If in addition, the assumptions of Theorem~\ref{Z-error} are satisfied and the control process is reconstructed after training according to \eqref{Zlearned}, then
\begin{equation*}
\mathbb E \left[ \sum_{i=0}^{N-1} \int_{t_i}^{t_{i+1}} \left| Z_t-\mathcal Z_i^\lambda \right|^2\mathrm{d}s \right] \leq C|\pi|^\frac{1}{3},
\end{equation*}
where \(C>0\) is independent of \(N\), \(|\pi|\), and \(\lambda\).
\end{theorem}

\begin{proof}
Fix \(i\in\{0,\ldots,N-1\}\), and let $\kappa,\ell\in\mathbb{N}^*$ and $m =(m_i)_{0\leq i\leq \ell-1} \in\mathbb{N}^\ell$ to be chosen later.

By considering the neural network class $\mathcal G_{K_i,d,\ell,m}^{\rho^\kappa}$, then $\varepsilon_i^{\mathcal N,v^\lambda}$ defined by \eqref{eq:NN-v-error} can be written as
\[ \varepsilon_i^{\mathcal N,v^\lambda} = \inf_{\mathcal U\in \mathcal G_{K_i,d,\ell,m}^{\rho^\kappa}} \mathbb E \left[ \left| v_i^\lambda(X_{i}^\pi) - \mathcal U(X_{i}^\pi) \right|^2 \right]. \]

\(R>0\), and \(\alpha\in(0,1]\).
	By Lemma~\ref{lemma:lip-v},
	\[
	[v_i^\lambda]\leq K_i\leq K.
	\]

Proposition~\ref{thm:univ approx2} therefore yields a \(K_i\)-Lipschitz GroupSort network $\mathcal U_i^{\alpha,R} \in \mathcal G_{K_i,d,\ell,m}^{\rho^\kappa} $
	such that
	\begin{equation}
		\sup_{x\in[-R,R]^d}
		\left|
		v_i^\lambda(x)
		-
		\mathcal U_i^{\alpha,R}(x)
		\right|
		\leq
		2RK_i\alpha
		\leq
		2RK\alpha.
		\label{eq:GS-uniform-improved}
	\end{equation}
	
	We decompose
	\[
	\mathbb E
	\left[
	\left|
	v_i^\lambda(X_i^\pi)
	-
	\mathcal U_i^{\alpha,R}(X_i^\pi)
	\right|^2
	\right]
	=
	I_{i,R}+J_{i,R},
	\]
	where
	\[
	I_{i,R}
	:=
	\mathbb E
	\left[
	\left|
	v_i^\lambda(X_i^\pi)
	-
	\mathcal U_i^{\alpha,R}(X_i^\pi)
	\right|^2
	\mathbf 1_{\{|X_i^\pi|\leq R\}}
	\right],\qquad
	J_{i,R}
	:=
	\mathbb E
	\left[
	\left|
	v_i^\lambda(X_i^\pi)
	-
	\mathcal U_i^{\alpha,R}(X_i^\pi)
	\right|^2
	\mathbf 1_{\{|X_i^\pi|>R\}}
	\right].
	\]
	Since $\{|x|\leq R\} \subset[-R,R]^d$, \eqref{eq:GS-uniform-improved} gives
	\begin{equation}
		I_{i,R}
		\leq
		4R^2K^2\alpha^2.
		\label{eq:I-R-improved}
	\end{equation}
	
	For the second term, both
	\(v_i^\lambda\) and
	\(\mathcal U_i^{\alpha,R}\) are \(K\)-Lipschitz.
	Furthermore, since \(0\in[-R,R]^d\),
	\[
	|v_i^\lambda(0)
	-
	\mathcal U_i^{\alpha,R}(0)|
	\leq
	2RK\alpha.
	\]
Thus, for every \(x\in\mathbb R^d\),
\begin{align*}
|v_i^\lambda(x) - \mathcal U_i^{\alpha,R}(x)| &\leq |v_i^\lambda(x)-v_i^\lambda(0)| + |v_i^\lambda(0) - \mathcal U_i^{\alpha,R}(0)| + |\mathcal U_i^{\alpha,R}(0) - \mathcal U_i^{\alpha,R}(x)| \\
	&\leq 2K|x| + 2RK\alpha.
\end{align*}
	Consequently,
	\[
	|v_i^\lambda(x)
	-
	\mathcal U_i^{\alpha,R}(x)|^2
	\leq
	CK^2
	\left(
	|x|^2
	+
	R^2\alpha^2
	\right).
	\]
	Hence
	\begin{align}
		J_{i,R}
		\leq{}&
		C
		\mathbb E
		\left[
		|X_i^\pi|^2
		\mathbf 1_{\{|X_i^\pi|>R\}}
		\right]
		\nonumber+
		CR^2\alpha^2
		\mathbb P(|X_i^\pi|>R).
		\label{eq:J-R-improved}
	\end{align}
	
	Under \ref{HL}, the Euler scheme satisfies
	\[ \sup_{0\leq i\leq N} \mathbb E \left[|X_i^\pi|^4 \right]\leq C. \]
	Therefore, by H\"older's and Markov's inequalities,
	\[
	\mathbb E
	\left[
	|X_i^\pi|^2
	\mathbf 1_{\{|X_i^\pi|>R\}}
	\right]
	\leq
	\frac C{R^2},\qquad\text{and}\qquad
	\mathbb P(|X_i^\pi|>R)
	\leq
	\frac C{R^4}.
	\]
	Substituting into the above bound of $J_{i,R}$, we obtain
	\[
	J_{i,R}
	\leq
	C
	\left(
	\frac1{R^2}
	+
	\frac{\alpha^2}{R^2}
	\right).
	\]
	Since \(\alpha\leq1\),
	\[
	J_{i,R}
	\leq
	\frac C{R^2}.
	\]
	
	Combining this estimate with
	\eqref{eq:I-R-improved} yields
	\begin{equation*}
		\varepsilon_i^{\mathcal N,v^\lambda}
		\leq
		C
		\left(
		R^2\alpha^2
		+
		\frac1{R^2}
		\right).
	\end{equation*}
	Consequently,
	\[
	N
	\sum_{i=0}^{N-1}
	\varepsilon_i^{\mathcal N,v^\lambda}
	\leq
	CN^2
	\left(
	R^2\alpha^2
	+
	\frac1{R^2}
	\right).
	\]
Choose $ R=N^{3/2}$ and $\alpha=N^{-3}$, then
\[ N^2R^2\alpha^2 = N^{-1}, \qquad \frac{N^2}{R^2} = N^{-1}. \]

Since $|\pi| = \frac{T}{N}$, we obtain
$$ N\sum_{i=0}^{N-1}\varepsilon_i^{\mathcal N,v^\lambda} \leq C|\pi|$$
which proves \eqref{eq:NN-total-error}.

For \(d>1\), Proposition~\ref{thm:univ approx2} allows us to choose
\[ \kappa = \left\lceil \frac{2\sqrt d}{\alpha} \right\rceil = \left\lceil 2\sqrt d N^3 \right\rceil, \]
with depth and width satisfying
\[ \ell+1=O(d^2),\qquad  \sum_{i=0}^{\ell-1}m_i = O\left( \left(2\sqrt d N^3 \right)^{d^2-1} \right).\]
For \(d=1\), we choose
$$\kappa = \left\lceil N^3\right\rceil,\qquad \ell + 1 = 3,\qquad \text{and}\qquad \sum_{i=0}^{\ell -1}m_i = O(N^3).$$

Finally, inserting \eqref{eq:NN-total-error} into the global error estimates \eqref{global-error-Y} and \eqref{global-error-Z}, and using $|\pi|\leq C|\pi|^\frac{1}{3}$, we obtain the desired result.
\end{proof}

\begin{remark}
	\label{rem:silu-numerics}
The convergence analysis is based on Lipschitz-constrained GroupSort networks, which provide explicit control of the neural-network approximation error. The uniform Lipschitz estimates established above allow these constraints to be chosen independently of the penalization parameter \(\lambda\).

For the numerical experiments, we instead use standard feedforward networks with SiLU activation for computational convenience. The DBDP recursion and loss function are unchanged; only the approximation class is modified. Thus, GroupSort networks provide the theoretical approximation framework, while SiLU networks are used as a practical implementation. No separate quantitative approximation result for the SiLU architecture is claimed.
\end{remark}
\section{Numerical Experiments}
\label{sec:numerics}

The purpose of this section is to illustrate the error analysis developed in
Section~\ref{sec:convergence}, rather than to benchmark the proposed scheme
against alternative numerical methods. The reference values used below are
either available in closed form or obtained from an independent high-accuracy
benchmark, so that the reported errors measure the distance to a reference
solution rather than the discrepancy between two numerical approximations.

As discussed in Remark~\ref{rem:silu-numerics}, the GroupSort architecture is used as a theoretical approximation tool, whereas the SiLU networks provide the practical implementation of the same backward approximation principle. However, we use standard feedforward networks with SiLU activation for computational convenience.

\subsection{Implementation details}

All experiments are carried out with the same set of hyperparameters. We set
\[
\lambda=c_0N^{1/3},
\qquad
c_0=10.
\]

At each time level \(t_i\), the value function is approximated by a
feedforward neural network
\[
\mathcal U(\cdot;\theta):\mathbb R^d\longrightarrow\mathbb R.
\]
Unless otherwise specified, we use \(\ell=2\) hidden layers with \(d+10\)
neurons per layer and SiLU activation. The standardized state
\(X_{t_i}^{\pi}\) is used as input.

The parameters obtained at time \(t_{i+1}\) are used to initialize the
optimization at time \(t_i\). We perform \(2\,000\) gradient steps at the
first backward optimization step and \(400\) gradient steps at each subsequent
step. The loss is minimized using the Adam optimizer, with learning rate
\(10^{-2}\) at the first backward step and \(2\times10^{-3}\) afterwards.

Mini-batches of size \(4096\) are drawn from a fixed sample of
$
M=2^{16},
$
forward trajectories, generated once for each independent run. 

The control process is reconstructed only in
Example~\ref{subsec:ex1}, with \(d=1\), in order to assess the accuracy of
the post-training approximation~\eqref{zfunction} against the known exact
control. Once the backward sweep has been completed, \(z_i^\lambda\) is
obtained by least-squares regression on a polynomial basis of degree \(7\).
The conditional expectation in~\eqref{zfunction} is estimated using \(16\)
antithetic Brownian increments generated at each training point. The control
reconstruction is performed after training and does not enter the loss used
to optimize the value networks.

The errors are evaluated on an independent sample of
$
M_{\mathrm{test}}=2^{17},
$
paths. Writing
\[
\bigl(X^{\pi,p}\bigr)_{1\leq p\leq M_{\mathrm{test}}}
\]
for this sample, we report
\begin{equation*}
	\widehat{\mathcal E}_Y
	=
	\max_{0\leq i\leq N-1}
	\frac{1}{M_{\mathrm{test}}}
	\sum_{p=1}^{M_{\mathrm{test}}}
	\left|
	u\bigl(t_i,X_{t_i}^{\pi,p}\bigr)
	-
	\mathcal U_i^\lambda
	\bigl(X_{t_i}^{\pi,p}\bigr)
	\right|^2.
\end{equation*}

When the control process is reconstructed, we also report
\begin{equation*}
	\widehat{\mathcal E}_Z
	=
	\frac{1}{M_{\mathrm{test}}}
	\sum_{p=1}^{M_{\mathrm{test}}}
	\sum_{i=0}^{N-1}
	\sum_{j=0}^{n_{\mathrm{sub}}-1}
	\delta
	\left|
	Z\bigl(s_{i,j},X_{s_{i,j}}^p\bigr)
	-
	z_i^\lambda
	\bigl(X_{t_i}^{\pi,p}\bigr)
	\right|^2,
\end{equation*}
where
\[
s_{i,j}=t_i+j\delta,
\qquad
\delta=\frac{\Delta t_i}{n_{\mathrm{sub}}},
\qquad
n_{\mathrm{sub}}=5.
\]
The fine sub-grid is simulated jointly with the coarse grid.

Each configuration is repeated \(50\) times with independent random seeds, and we report the empirical mean and standard deviation over these runs.
All computations are performed in single precision using PyTorch on two
NVIDIA T4 GPUs.

\subsection{Reflected BSDE with an active obstacle and an explicit solution}

Let \(W\) be a \(d\)-dimensional Brownian motion, and let the forward process
be the \(d\)-dimensional Ornstein--Uhlenbeck process
\begin{equation}
	\mathrm dX_t^k
	=
	\theta(m-X_t^k)\,\mathrm dt
	+
	\sigma\,\mathrm dW_t^k,
	\qquad
	X_0^k=x_0,
	\qquad
	k=1,\ldots,d.
\end{equation}

For \(x=(x_1,\ldots,x_d)\in\mathbb R^d\), define
\[
\bar x
=
\frac1d\sum_{k=1}^d x_k,
\qquad
\varphi(s)
=
\sqrt{1+s^2},
\qquad
c(t)
=
1+T-t,
\]
together with
\[
\rho(s)
=
\frac{(s^+)^2}{1+s^+},
\qquad
b(t)
=
b_0+a\sin\left(\frac{2\pi t}{T}\right).
\]

We consider the reflected BSDE
\begin{equation*}
  \left\{
  \begin{aligned}
    &Y_t = \xi + \int_t^T f(s,X_s,Y_s)\,\mathrm{d}s + K_T - K_t
      - \int_t^T Z_s\,\mathrm{d}W_s, \qquad 0\leq  t\leq  T,\\
    &Y_t \geq h(t,X_t), \qquad 0\leq  t\leq  T,\\
    &K \text{ continuous and non-decreasing},\quad K_0 = 0,\\
    &\int_0^T \big(Y_t - h(t,X_t)\big)\,\mathrm{d}K_t = 0 ,
  \end{aligned}
  \right.
\end{equation*}

where the terminal condition is
\[
\xi=g(X_T),
\qquad
g(x)=\varphi(\bar x),
\]
the obstacle is given by
\begin{equation*}
	h(t,x)
	=
	c(t)\varphi(\bar x)
	-
	\kappa_1
	\rho\bigl(\bar x-b(t)\bigr),
\end{equation*}
and the generator is defined by
\begin{equation*}
	f(t,x,y)
	=
	-r y
	+
	\gamma\sin y
	+
	q(t,x)
	-
	\psi(t,x),
\end{equation*}
where
\[
\psi(t,x)
=
\kappa_2
\left(
1-e^{-(b(t)-\bar x)^+}
\right)
\]
and
\begin{equation*}
	\begin{aligned}
		q(t,x)
		={}&
		\varphi(\bar x)
		-
		c(t)
		\left[
		\theta(m-\bar x)\varphi'(\bar x)
		+
		\frac{\sigma^2}{2d}\varphi''(\bar x)
		\right]
		\\
		&+
		r\,c(t)\varphi(\bar x)
		-
		\gamma\sin\bigl(c(t)\varphi(\bar x)\bigr).
	\end{aligned}
\end{equation*}

The reflected BSDE admits the explicit solution
\begin{equation*}
Y_t = (1+T-t) \sqrt{1+\bar X_t^{\,2}},\qquad Z_t^j = \frac{\sigma}{d} (1+T-t) \frac{\bar X_t} {\sqrt{1+\bar X_t^{\,2}}},\;\; j=1,\ldots,d,
\end{equation*}

and the reflection process is absolutely continuous and satisfies
\begin{equation*}
	\mathrm dK_t
	=
	\psi(t,X_t)\,\mathrm dt.
\end{equation*}

We take
\[
T=1,
\qquad
x_0=\frac12,
\qquad
\theta=1,
\qquad
m=\frac12,
\qquad
\sigma=\frac12,
\]
\[
r=\frac12,
\qquad
\gamma=\frac{3}{10},
\qquad
\kappa_1=1,
\qquad
\kappa_2=2,
\qquad
b_0=\frac12,
\qquad
a=\frac{3}{10}.
\]

The functions entering the example satisfy the regularity assumptions used in the analysis.

\subsubsection{One-dimensional case}
\label{subsec:ex1}

We first consider \(d=1\).
Table~\ref{tab:err_example1_Dim1} reports the approximations of the initial
value \(Y_0\), together with the empirical \(L^2\)-errors of the value and control processes.

The approximation of the value component improves steadily as the time grid
is refined. The relative error on \(Y_0\) decreases from \(0.386\%\) for
\(N=10\) to \(0.101\%\) for \(N=80\), while
\(\widehat{\mathcal E}_Y\) decreases from
\(7.24\times10^{-4}\) to \(1.73\times10^{-4}\).

For the control component, the global empirical error
\(\widehat{\mathcal E}_Z\) also decreases consistently with the time
discretization, from
$
3.43\times10^{-3}
$
for \(N=10\) to
$
4.95\times10^{-4}
$
for \(N=80\).

\begin{table}[H]
	\centering
	\setlength{\tabcolsep}{3pt}
	\resizebox{\textwidth}{!}{%
		\begin{tabular}{|ccc|ccccc|cc|}
			\hline
			$d$ & $N$ & $\lambda$
			& $Y^{\mathrm{net}}_0$ (mean)
			& std
			& $Y_0^{\mathrm{ref}}$
			& $|\mathrm{err}|$
			& rel. err.
			& $\widehat{\mathcal E}_Y$ (std)
			& $\widehat{\mathcal E}_Z$ (std)
			\\
			\hline
			1 & 10 & 21.543 & 2.227437 & 1.03e-04 & 2.236068 & 8.631e-03 & 0.386\% & 7.24e-04 (5.5e-06) & 3.43e-03 (2.1e-05)  \\
			1 & 20 & 27.141 & 2.231310 & 8.79e-05 & 2.236068 & 4.758e-03 & 0.213\% & 4.24e-04 (4.8e-06) & 1.71e-03 (1.9e-05) \\
			1 & 40 & 34.195 & 2.232856 & 7.28e-05 & 2.236068 & 3.212e-03 & 0.144\% & 2.69e-04 (2.2e-06) & 9.02e-04 (6.8e-06) \\
			1 & 80 & 43.082 & 2.233819 & 8.80e-05 & 2.236068 & 2.249e-03 & 0.101\% & 1.73e-04 (1.4e-06) & 4.95e-04 (7.2e-06) \\
			\hline
		\end{tabular}%
	}
	\caption{\centering
		Dimension \(d=1\): approximation of \(Y_0\), and empirical \(L^2\)-errors of \(Y\) and \(Z\) as functions of the number of time steps \(N\).
	}
	\label{tab:err_example1_Dim1}
\end{table}
Figure~\ref{fig:trajectory_example1_Dim1} displays one realization of the
learned solution along a trajectory of the forward process, together with the
corresponding obstacle and exact solution. The learned value process closely
tracks the exact solution and reproduces the contact region where the
reflection is active. Since the numerical method is based on penalization,
the reflection constraint is not imposed exactly for finite \(\lambda\).

The reconstructed control follows the exact control over the time interval.

\begin{figure}[H]
	\centering
	\begin{subfigure}[b]{0.48\textwidth}
		\includegraphics[width=\textwidth]{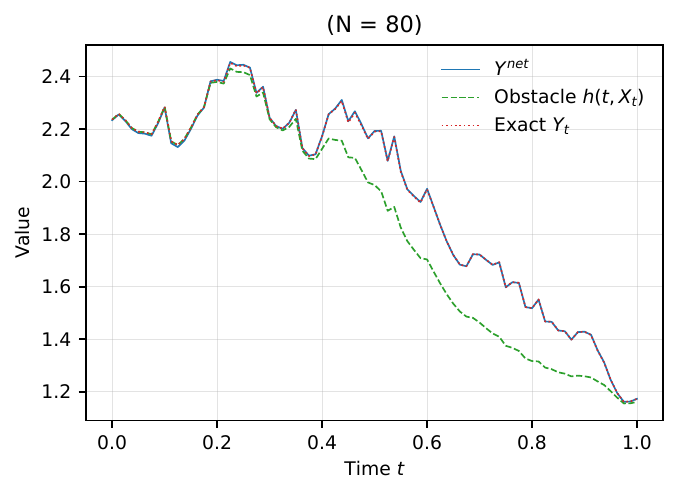}
	\end{subfigure}
	\hfill
	\begin{subfigure}[b]{0.48\textwidth}
		\includegraphics[width=\textwidth]{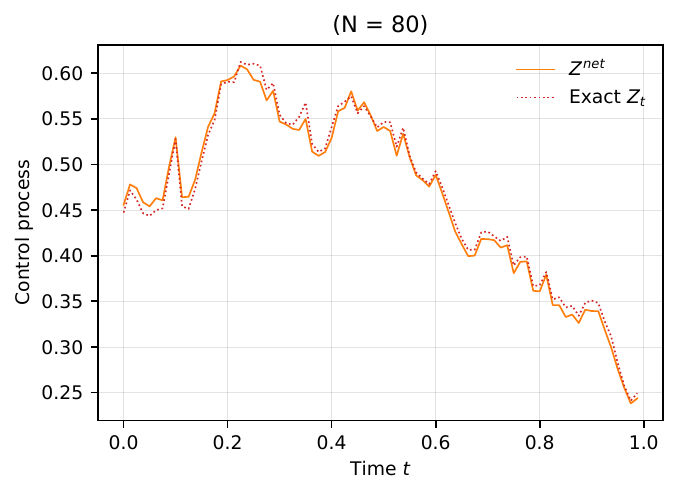}
	\end{subfigure}
	\caption{\centering
		Left: sample paths of the learned and exact value processes, together with
		the corresponding obstacle \(h(t,X_t)\).
		Right: sample paths of the exact control process and its post-training
		reconstruction. All curves correspond to the same realization of the
		forward process.
	}
	\label{fig:trajectory_example1_Dim1}
\end{figure}

\subsubsection{Multi-dimensional case}

We next investigate the numerical behavior of the proposed scheme as the
dimension increases. Table~\ref{tab:exacte_solution_MultiDim} below reports the
approximation of the initial value \(Y_0\), the corresponding relative error,
and the empirical \(L^2\)-error of the learned value process for several
dimensions and time discretizations. We do not report errors for the control
process in this experiment, since the post-training reconstruction
\eqref{zfunction} of \(Z\) is examined only in the one-dimensional case.

For $d\in\{150,250,300\}$, we used networks with 3 hidden layers, and 64 neurones in each layer, while for $d=50$, we used two hidden layers with $60 = d + 10$ neurons.

\begin{table}[htbp]
	\centering
	\begin{tabular}{|ccc|cccc|c|}
		\hline
		$d$ & $N$ & $\lambda$ & $Y^{\mathrm{net}}_0$ (mean) & std
		& $Y_0^{\mathrm{ref}}$ & rel. err.
		& $\widehat{\mathcal E}_Y$ (std) \\
		\hline
		50 & 10 & 21.543 & 2.228050 & 1.67e-04 & 2.236068 & 0.359\% & 5.47e-04 (3.8e-06) \\
		50 & 20 & 27.141 & 2.231849 & 1.64e-05 & 2.236068 & 0.189\% & 3.24e-04 (4.3e-07) \\
		50 & 40 & 34.195 & 2.233327 & 6.67e-06 & 2.236068 & 0.123\% & 2.12e-04 (2.2e-07) \\
		50 & 80 & 43.082 & 2.234260 & 6.33e-06 & 2.236068 & 0.081\% & 1.38e-04 (2.5e-07) \\
		\hline
		150 & 10 & 21.543 & 2.227824 & 2.70e-04 & 2.236068 & 0.369\% & 5.51e-04 (3.6e-06) \\
		150 & 20 & 27.141 & 2.231844 & 1.75e-05 & 2.236068 & 0.189\% & 3.25e-04 (1.1e-06) \\
		150 & 40 & 34.195 & 2.233327 & 4.52e-06 & 2.236068 & 0.123\% & 2.12e-04 (2.1e-07) \\
		150 & 80 & 43.082 & 2.234256 & 3.99e-06 & 2.236068 & 0.081\% & 1.40e-04 (4.4e-06) \\
		\hline
		250 & 10 & 21.543 & 2.228311 & 6.07e-04 & 2.236068 & 0.347\% & 5.53e-04 (5.2e-06) \\
		250 & 20 & 27.141 & 2.231799 & 2.66e-05 & 2.236068 & 0.191\% & 3.27e-04 (1.5e-06) \\
		250 & 40 & 34.195 & 2.233321 & 3.17e-05 & 2.236068 & 0.123\% & 2.12e-04 (5.9e-07) \\
		250 & 80 & 43.082 & 2.234255 & 3.57e-06 & 2.236068 & 0.081\% & 1.39e-04 (4.8e-07) \\
		\hline
		300 & 10 & 21.543 & 2.228298 & 5.02e-04 & 2.236068 & 0.347\% & 5.57e-04 (7.1e-06) \\
		300 & 20 & 27.141 & 2.231795 & 2.80e-05 & 2.236068 & 0.191\% & 3.27e-04 (1.8e-06) \\
		300 & 40 & 34.195 & 2.233325 & 3.57e-06 & 2.236068 & 0.123\% & 2.13e-04 (7.3e-07) \\
		300 & 80 & 43.082 & 2.234255 & 6.74e-06 & 2.236068 & 0.081\% & 1.39e-04 (5.4e-07) \\
		\hline
	\end{tabular}
	\caption{\centering
		Approximation of \(Y_0\) and empirical \(L^2\)-error of the value process
		for \(d\in\{50,150,250,300\}\) and several numbers of time steps \(N\).
	}
	\label{tab:exacte_solution_MultiDim}
\end{table}

The results show that the approximation remains stable as the dimension
increases. For every tested dimension, refining the time grid improves both
the approximation of \(Y_0\) and the empirical \(L^2\)-error
\(\widehat{\mathcal E}_Y\). For \(N=80\), the relative error is approximately
\(0.081\%\) for all four dimensions. Likewise,
\(\widehat{\mathcal E}_Y\) decreases from values of order
\(5.5\times10^{-4}\) for \(N=10\) to values of order
\(1.4\times10^{-4}\) for \(N=80\).

In this experiment, no significant deterioration is observed when the
dimension increases from \(50\) to \(300\). This stability is consistent
with the particular structure of the exact solution, which depends on the
state only through the empirical mean
\[
\bar X_t
=
\frac1d\sum_{k=1}^d X_t^k.
\]
Thus, the example should be interpreted as a structured high-dimensional
benchmark rather than as a completely generic \(d\)-dimensional problem.

Figure~\ref{fig:paths_example1_d50_150_250_300} below  confirms the same qualitative
behavior. The learned value process closely follows the exact solution and
reproduces the region where the obstacle becomes active. Since the numerical
method is based on penalization, exact satisfaction of the reflection
constraint is not imposed for finite \(\lambda\).

\begin{figure}[!ht]
	\centering
	
	\begin{subfigure}[b]{0.45\textwidth}
		\centering
		\includegraphics[width=\textwidth]{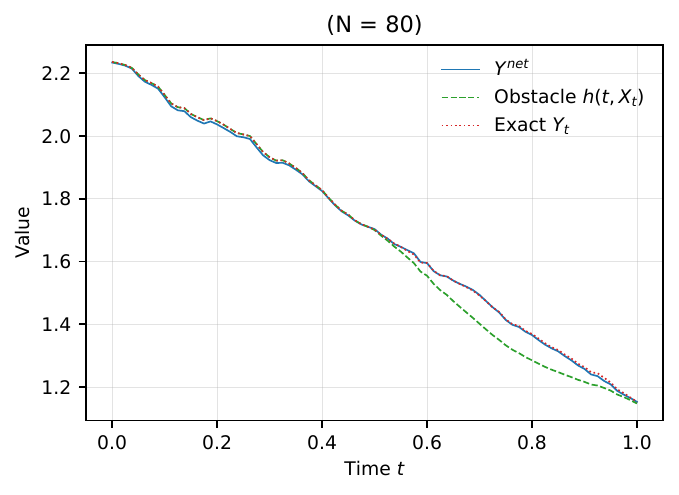}
		\caption{$d=50$}
	\end{subfigure}
	\hfill
	\begin{subfigure}[b]{0.45\textwidth}
		\centering
		\includegraphics[width=\textwidth]{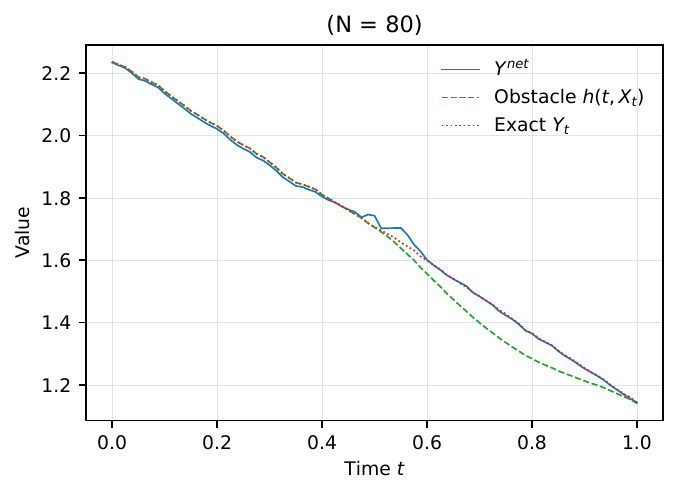}
		\caption{$d=150$}
	\end{subfigure}
	
	\vspace{0.5em}
	
	\begin{subfigure}[b]{0.45\textwidth}
		\centering
		\includegraphics[width=\textwidth]{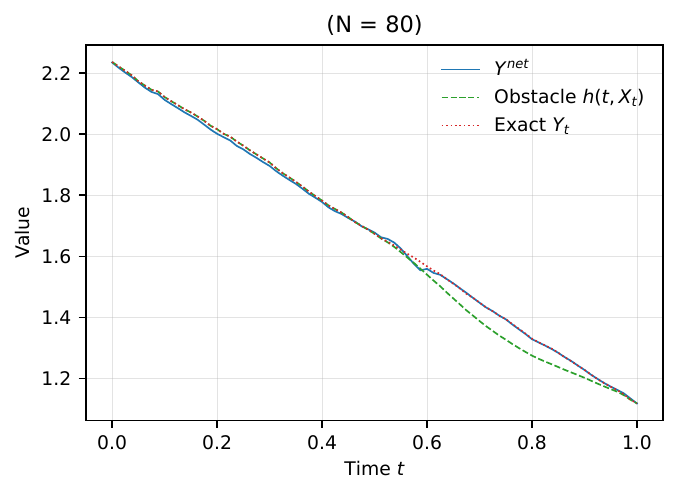}
		\caption{$d=250$}
	\end{subfigure}
	\hfill
	\begin{subfigure}[b]{0.45\textwidth}
		\centering
		\includegraphics[width=\textwidth]{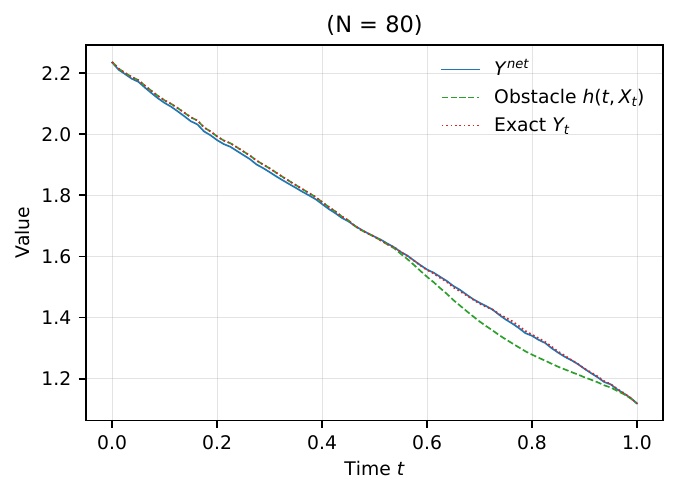}
		\caption{$d=300$}
	\end{subfigure}
	
	\caption{\centering Sample paths of the learned value process $Y^{\mathrm{net}}$, the exact value process $Y$, and the obstacle, associated with the same realization of the forward process, for $d\in\{50,150,250,300\}$.}
	\label{fig:paths_example1_d50_150_250_300}
\end{figure}
\subsection{American Put Option -- One-dimensional case}

In this example, we consider the pricing of an American put option under the
risk-neutral probability measure. Let \((X_t)_{t\in[0,T]}\) denote the
one-dimensional underlying asset price process, governed by the
Black--Scholes dynamics
\[
\mathrm dX_t
=
rX_t\,\mathrm dt
+
\sigma X_t\,\mathrm dW_t,
\qquad
X_0=100,
\]
where \(W\) is a standard Brownian motion under the risk-neutral measure.
We take
\[
r=3\%,
\qquad
\sigma=0.2,
\qquad
K_{\mathrm{str}}=100,
\qquad
T=1.
\]
The payoff function is
\[
h(x)
=
(K_{\mathrm{str}}-x)^+.
\]

The American put price function is characterized by the optimal stopping
problem
\[
u(t,x)
=
\sup_{\tau\in\mathcal T_{t,T}}
\mathbb E
\left[
e^{-r(\tau-t)}
h(X_\tau)
\,\middle|\,
X_t=x
\right],
\]
where \(\mathcal T_{t,T}\) denotes the set of stopping times taking values in
\([t,T]\).

The price function admits the probabilistic representation
\[
u(t,X_t)=Y_t,
\]
where \((Y,Z,K)\) solves the reflected BSDE
\[
Y_t
=
h(X_T)
+
\int_t^T
(-rY_s)\,\mathrm ds
+
K_T-K_t
-
\int_t^T
Z_s\,\mathrm dW_s,
\]
subject to the reflection conditions
\[
Y_t\geq h(X_t),
\qquad
0\leq t\leq T,
\]
and
\[
\int_0^T
\bigl(Y_t-h(X_t)\bigr)\,\mathrm dK_t
=
0.
\]
The reference value $Y_0^{\mathrm{ref}} = 6.742045$ is obtained with a $1000$-step Cox–Ross–Rubinstein binomial tree, independently of the schemes under study.

\begin{table}[H]
	\centering
	\begin{tabular}{cccccc}
		\hline
		$N$ & $\lambda$ & $Y^{\mathrm{net}}_0$ (mean)
		& std & $Y_0^{\mathrm{ref}}$ & rel. err. \\
		\hline
		10 & 21.543 & 6.690921 & 1.82e-02 & 6.742045 & 0.758\% \\
		20 & 27.141 & 6.716370 & 1.91e-02 & 6.742045 & 0.381\% \\
		40 & 34.195 & 6.729796 & 1.69e-02 & 6.742045 & 0.182\% \\
		80 & 43.082 & 6.737160 & 1.74e-02 & 6.742045 & 0.072\% \\
		\hline
	\end{tabular}
	\caption{\centering
		American put option in dimension \(d=1\): mean, standard deviation, and
		relative error over \(50\) independent runs of the predicted initial price
		\(Y_0^{\mathrm{net}}\) with respect to the reference value
		\(Y_0^{\mathrm{ref}}\).
	}
	\label{tab:AmericanPut_D1}
\end{table}

The results reported in Table~\ref{tab:AmericanPut_D1} show a systematic
improvement as the time grid is refined. The relative error decreases from
\(0.758\%\) for \(N=10\) to \(0.072\%\) for \(N=80\). The standard deviation
remains of the same order over the \(50\) independent runs, indicating stable
training across the tested time discretizations. Moreover, the predicted
initial value approaches the reference value monotonically as \(N\) increases.

Two realizations of the learned value process are displayed in
Figure~\ref{fig:paths_PutOption_1d}. Since the numerical method solves a penalized BSDE for finite \(\lambda\), the reflection constraint is not imposed exactly at the numerical
level.

\begin{figure}[htbp]
	\centering
	\begin{subfigure}[b]{0.45\textwidth}
		\includegraphics[width=\textwidth]{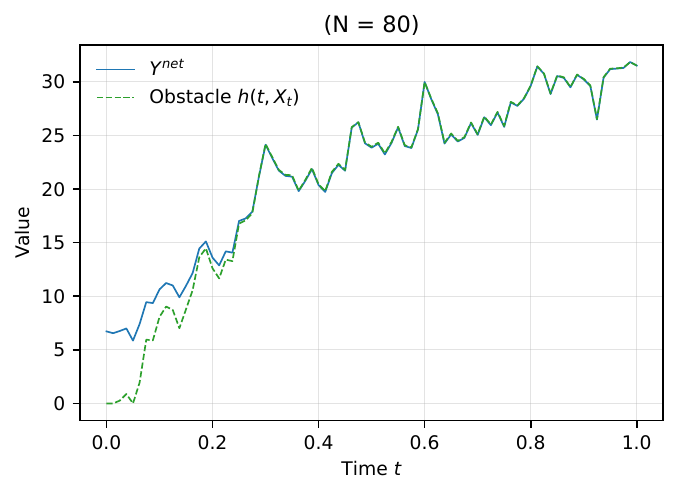}
	\end{subfigure}
	\hfill
	\begin{subfigure}[b]{0.45\textwidth}
		\includegraphics[width=\textwidth]{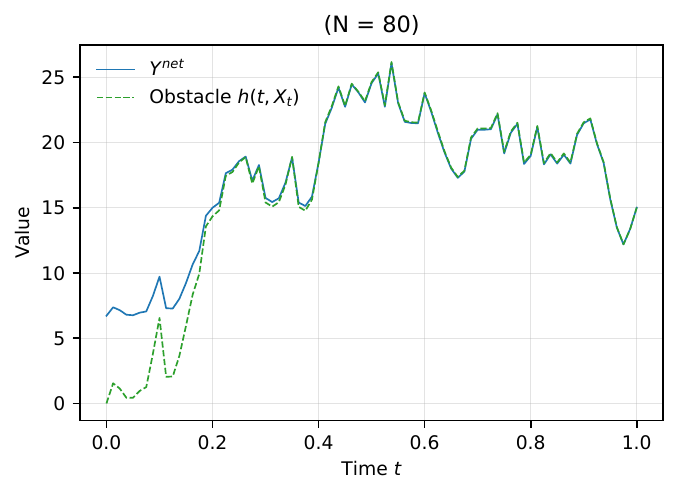}
	\end{subfigure}
	\caption{\centering
		American put option in dimension \(d=1\): two independent sample paths of
		the learned value process \(Y^{\mathrm{net}}\), together with the
		corresponding payoff processes \(h(X_t)\). In each panel, the learned
		value process and the payoff are associated with the same realization of
		the underlying asset process.
	}
	\label{fig:paths_PutOption_1d}
\end{figure}

Figure~\ref{fig:learned_value_function} displays the learned value functions $ x\longmapsto\mathcal U_i^\lambda(x)$
at several time levels, together with the payoff function $x\longmapsto h(x).$
As \(t_i\) approaches maturity, the learned value functions approach the
terminal payoff, consistently with the terminal condition
\[
u(T,x)=h(x).
\]

\begin{figure}[h]
	\centering
	\includegraphics[width=0.5\textwidth]{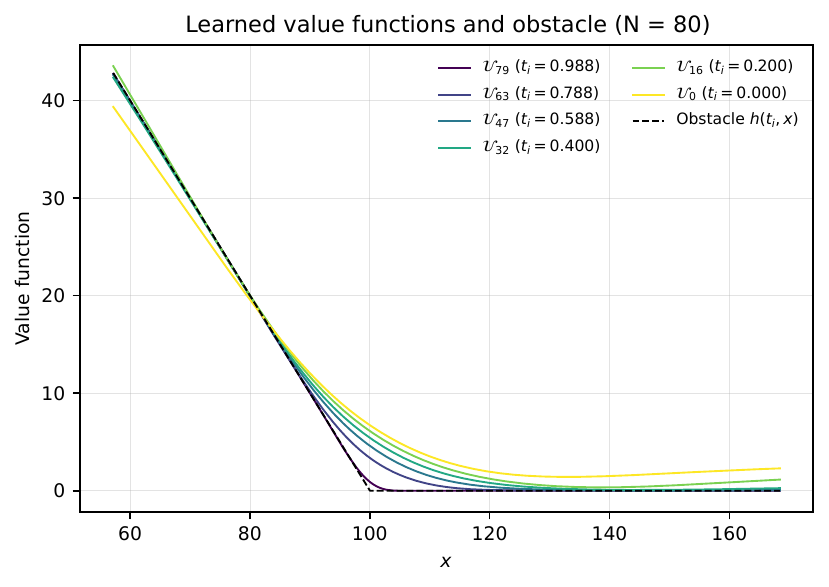}
	\caption{\centering
		American put option in dimension \(d=1\): learned value functions at
		different time levels \(t_i\), together with the payoff function \(h\).
	}
	\label{fig:learned_value_function}
\end{figure}

\subsection{American Put Option -- Multi-dimensional case}

Finally, we consider an American put option on the product of \(d\) assets
in a multi-dimensional Black--Scholes model. Let
\[
X_t=(X_t^1,\ldots,X_t^d),
\qquad
0\leq t\leq T,
\]
denote the vector of asset prices. Under the risk-neutral probability measure,
the components of \(X\) satisfy
\[
\mathrm dX_t^i
=
X_t^i
\left(
r\,\mathrm dt
+
\sum_{j=1}^d
\sigma_{ij}\,\mathrm dW_t^j
\right),
\qquad
X_0^i=1,
\qquad
i=1,\ldots,d,
\]
where \(W\) is a \(d\)-dimensional Brownian motion, $r=5\%$ is the risk-free interest rate, and $\sigma=0.2\,I_d$ is the volatility matrix. The maturity and strike price are respectively $T=1$ and $K_{\mathrm{str}}=1$.

The payoff is given by
\[
g(x)
=
\left(
K_{\mathrm{str}}
-
\prod_{k=1}^d x_k
\right)^+,
\qquad
x=(x_1,\ldots,x_d)\in(0,\infty)^d,
\]
where \(a^+=\max(a,0)\).

The arbitrage-free price process is characterized by the solution
\((Y,Z,K)\) of the reflected BSDE
\[
Y_t
=
g(X_T)
+
\int_t^T(-rY_s)\,\mathrm ds
+
K_T-K_t
-
\int_t^T Z_s\,\mathrm dW_s,
\]
subject to
\[
Y_t\geq g(X_t),
\qquad
0\leq t\leq T,
\]
and
\[
\int_0^T
\bigl(Y_t-g(X_t)\bigr)\,\mathrm dK_t
=
0.
\]

For numerical purposes, we use a logarithmic change of variables. Let \(u\)
denote the value function in the original variables and define \(v\) by
\[
u(t,x)
=
e^{rt}v(t,\log x),
\]
where the logarithm is applied componentwise. Equivalently,
$
v(t,y)
=
e^{-rt}u(t,e^y).
$
Then \(v\) solves the variational inequality
\begin{equation}
	\label{eq:log-variational}
	\left\{
	\begin{aligned}
		&\min\bigl(
		-\partial_t v-\mathcal Lv,\,
		v-\widehat g
		\bigr)=0,
		&&\text{on }[0,T)\times\mathbb R^d,
		\\
		&v(T,\cdot)=\widehat g(T,\cdot),
		&&\text{on }\mathbb R^d,
	\end{aligned}
	\right.
\end{equation}
where
\[
\widehat g(t,x)
=
e^{-rt}g(e^x)
=
e^{-rt}
\left(
K_{\mathrm{str}}
-
e^{\sum_{i=1}^d x_i}
\right)^+
\]
and
\[
\mathcal Lv
=
\frac12
\sum_{i=1}^d
\sigma_i^2D_{x_ix_i}^2v
+
\sum_{i=1}^d
\left(
r-\frac12\sigma_i^2
\right)
D_{x_i}v,
\qquad
\sigma_i=0.2.
\]

In logarithmic coordinates, the forward process is the arithmetic Brownian
motion
\[
\mathrm d\widetilde X_t^i
=
\left(
r-\frac12\sigma_i^2
\right)\mathrm dt
+
\sigma_i\,\mathrm dW_t^i,
\qquad
\widetilde X_0^i=0,
\]
and the corresponding RBSDE has zero driver, since the discounting
has been incorporated into the transformation.

The transformed obstacle depends on the state through the scalar quantity
$
\sum_{i=1}^d x_i
$
rather than through the product
$
\prod_{i=1}^d x_i.
$
The transformation is exact and preserves the initial option value:
\[
Y_0
=
v(0,\log X_0)
=
u(0,X_0).
\]
It also leaves the penalization parameter unchanged. Indeed, if
\[
\widehat Y_t^\lambda=e^{-rt}Y_t^\lambda,
\]
then
\[
e^{-rt}
\lambda
\bigl(
Y_t^\lambda-g(X_t)
\bigr)^-
=
\lambda
\bigl(
\widehat Y_t^\lambda-\widehat g(t,\widetilde X_t)
\bigr)^-.
\]
Thus the same value of \(\lambda\) is used in the transformed penalized
problem.

Although the transformed state variable remains \(d\)-dimensional, the
obstacle depends only on the one-dimensional statistic
$
\sum_{i=1}^d x_i.
$
This example should therefore be viewed as a structured high-dimensional
benchmark for the proposed numerical method.

The numerical experiments are performed for
$
d\in\{5,10,20,40\}
$
and several time discretizations. The computed initial prices are compared
with the reference values reported in Table~11 of \cite{pham}.\\

\begin{table}[H]
	\centering
	\begin{tabular}{ccccccc}
		\hline
		$d$ & $N$ & $\lambda$
		& $Y^{\mathrm{net}}_0$ (mean)
		& std
		& $Y_0^{\mathrm{ref}}$
		& rel. err. \\
		\hline
		5 & 10 & 21.543 & 0.103895 & 3.06e-04 & 0.107380 & 3.245\% \\
		5 & 20 & 27.141 & 0.105331 & 2.68e-04 & 0.107380 & 1.908\% \\
		5 & 40 & 34.195 & 0.106112 & 2.64e-04 & 0.107380 & 1.181\% \\
		5 & 80 & 43.082 & 0.106641 & 1.85e-04 & 0.107380 & 0.688\% \\
		\addlinespace
		10 & 10 & 21.543 & 0.123564 & 3.92e-04 & 0.129960 & 4.921\% \\
		10 & 20 & 27.141 & 0.126225 & 2.72e-04 & 0.129960 & 2.874\% \\
		10 & 40 & 34.195 & 0.127880 & 2.65e-04 & 0.129960 & 1.601\% \\
		10 & 80 & 43.082 & 0.129096 & 2.74e-04 & 0.129960 & 0.665\% \\
		\addlinespace
		20 & 10 & 21.543 & 0.140489 & 4.23e-04 & 0.151000 & 6.961\% \\
		20 & 20 & 27.141 & 0.145351 & 4.51e-04 & 0.151000 & 3.741\% \\
		20 & 40 & 34.195 & 0.148285 & 4.32e-04 & 0.151000 & 1.798\% \\
		20 & 80 & 43.082 & 0.150849 & 3.66e-04 & 0.151000 & 0.100\% \\
		\addlinespace
		40 & 10 & 21.543 & 0.151255 & 7.54e-04 & 0.168000 & 9.967\% \\
		40 & 20 & 27.141 & 0.159408 & 6.62e-04 & 0.168000 & 5.114\% \\
		40 & 40 & 34.195 & 0.164631 & 4.82e-04 & 0.168000 & 2.005\% \\
		40 & 80 & 43.082 & 0.168540 & 5.22e-04 & 0.168000 & 0.321\% \\
		\hline
	\end{tabular}
	\caption{\centering
		American put option on the product of the assets in the multi-dimensional
		setting: mean, standard deviation, and relative error over \(50\) independent
		runs of the predicted initial price \(Y_0^{\mathrm{net}}\) with respect to
		the benchmark value \(Y_0^{\mathrm{ref}}\), for
		\(d\in\{5,10,20,40\}\).
	}
	\label{tab:AmericanPut_MultiDim}
\end{table}

The results reported in Table~\ref{tab:AmericanPut_MultiDim} show a systematic
improvement as the time discretization is refined for every tested dimension. For each dimension, the relative
error decreases substantially as \(N\) increases. The empirical standard
deviations over the \(50\) independent runs remain small compared with the
corresponding benchmark values, indicating stable training across the tested
configurations.

The errors for coarse time grids tend to increase with the dimension; for
example, for \(N=10\) the relative error increases from \(3.245\%\) at
\(d=5\) to \(9.967\%\) at \(d=40\). This dimension effect becomes much less
pronounced as the time discretization is refined: for \(N=80\), all reported
relative errors are below \(0.7\%\).

\begin{figure}[!ht]
	\centering
	\begin{minipage}{0.85\textwidth}
		\centering
		
		\begin{subfigure}[b]{0.48\linewidth}
			\centering
			\includegraphics[width=\linewidth]{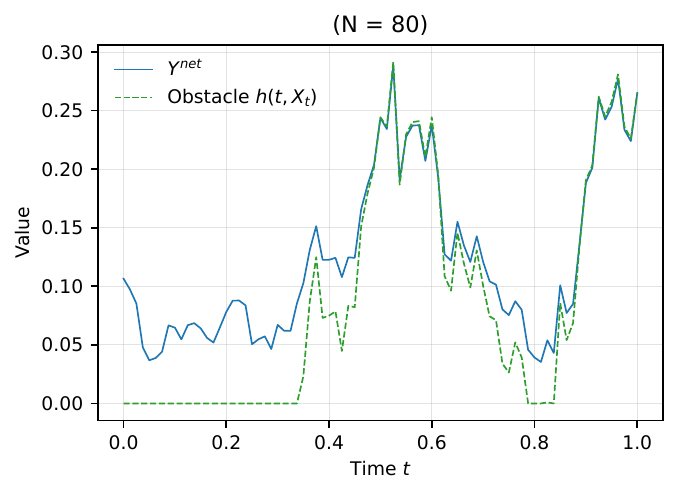}
			\caption{$d=5$.}
		\end{subfigure}
		\hfill
		\begin{subfigure}[b]{0.48\linewidth}
			\centering
			\includegraphics[width=\linewidth]{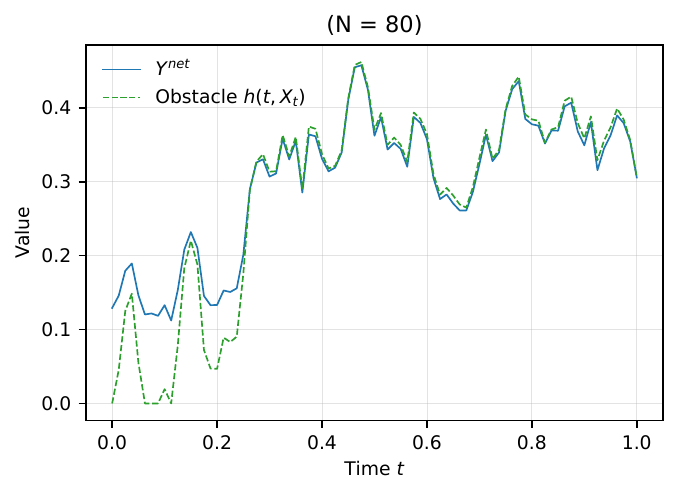}
			\caption{$d=10$.}
		\end{subfigure}
		
		\vspace{0.5em}
		
		\begin{subfigure}[b]{0.48\linewidth}
			\centering
			\includegraphics[width=\linewidth]{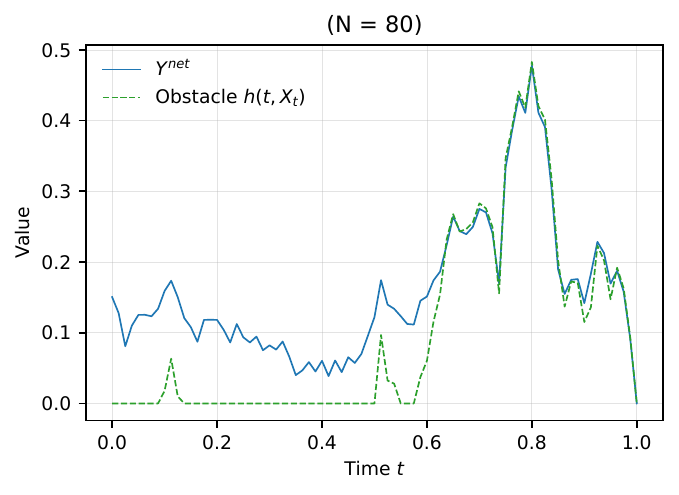}
			\caption{$d=20$.}
		\end{subfigure}
		\hfill
		\begin{subfigure}[b]{0.48\linewidth}
			\centering
			\includegraphics[width=\linewidth]{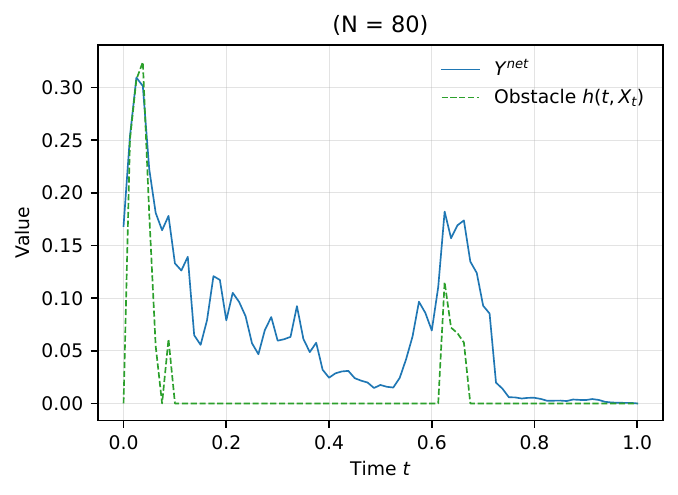}
			\caption{$d=40$.}
		\end{subfigure}
		
		\caption{\centering American put option in the multi-dimensional setting: sample paths of the learned value process $Y^{\mathrm{net}}$ and the corresponding transformed obstacle, associated with the same realization of the forward process, for $d\in\{5,10,20,40\}$.}
		\label{fig:PutOption_MultiDim_Paths}
	\end{minipage}
\end{figure}
Figure~\ref{fig:PutOption_MultiDim_Paths} displays representative trajectories
of the learned value process and the corresponding obstacle. The qualitative
behavior remains stable over the tested dimensions: the learned process
approaches the obstacle in the region where the penalization becomes active.
As in the previous experiments, exact pathwise satisfaction of the obstacle
constraint is not imposed for finite \(\lambda\).
\vspace*{0.2cm}

\section*{Acknowledgments}

The second author acknowledges financial support from
the Moroccan National Center for Scientific and Technical
Research (CNRST) through the PhD-Associate Scholarship
(PASS) Program.


\begin{thebibliography}{9}  

\bibitem{bouchard chassagneux} Bouchard, B., and Chassagneux, J.-F.: Discrete-time approximation for continuously and discretely reflected BSDEs, Stochastic Processes and their Applications, Vol. 118, No. 12, 2269–2293 (2008).

\bibitem{touzi} Bouchard, B., Touzi, N.: Discrete-time approximation and Monte-Carlo simulation of backward stochastic differential equations. Stochastic Processes and their Applications, Vol. 111, No. 2, 175-206 (2004).

\bibitem{bismut} Bismut, J. M.,Théorie probabiliste du contrôle des diffusions, Mem. Amer. Math. Soc., Vol. 4, No. 176 (1973).

\bibitem{karoui} El Karoui, N., Kapoudjian, C., Pardoux, E., Peng, S., and Quenez, M. C.: Reflected solutions of backward SDE's, and related obstacle problems for PDE's. The Annals of Probability, Vol. 25, No. 2, 702–737 (1997).

\bibitem{germain} Germain, M., Pham H. and Warin, X., Approximation Error Analysis of Some Deep Backward Schemes for Nonlinear PDEs. SIAM Journal on Scientific Computing, Vol. 44, No. 1, A28-A56 (2022).

\bibitem{gobet} Gobet, E., Wang, W.: Improved Convergence Rate for Reflected BSDEs by Penalization Method. Applied Mathematics \& Optimization, Vol. 93, No. 1, 10 (2026).


\bibitem{hornik 89} Hornik, K., Stinchcombe, M., and White, H.: Multilayer feedforward networks are universal approximators. Neural Networks, Vol. 2, No. 5, 359-366 (1989).

\bibitem{pham} Huré, C., Pham, H.,and Warin, X.: Deep backward schemes for high-dimensional nonlinear PDEs. Mathematics of Computation, Vol. 89, No. 324, 1547-1579 (2020).


\bibitem{peter imkeller} Imkeller, P., and Dos Reis, G.: Path regularity and explicit convergence rate for BSDE with truncated quadratic growth. Stochastic Processes and their Applications, Vol. 120, No. 3, 348-379 (2010).

\bibitem{platen} Kloeden, P. E., Platen, E.: Numerical solution of stochastic differential equations, Applications of Mathematics : Stochastic Modelling and Applied Probability 23, Springer (1995).


\bibitem{lepeltier} Lepeltier, J.-P., and Xu, M.: Penalization method for reflected backward stochastic differential equations with one r.c.l.l. barrier. Statistics and Probability Letters, Vol. 75, No. 1, 58-66 (2005).



\bibitem{ma zhang} Ma, J., and Zhang, J.: Representations and regularities for solutions to BSDEs with reflections. Stochastic Processes and their Applications, Vol. 115, No. 4, 539-569 (2005).



\bibitem{pardoux peng} Pardoux, E., and Peng, S.: Adapted solution of backward stochastic differential equation. Systems \& Control Letters, Vol. 14, No. 1, 55-61 (1990).


\bibitem{zhang} Zhang, J.: A numerical scheme for BSDEs. The Annals of Applied Probability, Vol. 14, No. 1, 459–488 (2004).

\bibitem{zhang book} Zhang, J.: Backward Stochastic Differential Equations. Probability Theory and Stochastic Modelling 86, Springer (2017).

\end{thebibliography}
\end{document}